\documentclass{amsart}

\usepackage[utf8]{inputenc}

\usepackage[dvips]{epsfig}
\usepackage{graphicx}
\usepackage{latexsym}
\usepackage{amsmath}
\usepackage{amsthm}
\usepackage{amssymb}
\usepackage{ulem}

\usepackage{amsmath,amsthm,amsfonts,amssymb,amscd,amsbsy,dsfont,hyperref}

\usepackage{palatino}

\usepackage{eucal}
\usepackage{float}
\usepackage{xcolor}
\usepackage{multirow}

\usepackage{textcomp}
\usepackage{lmodern}

\usepackage{amsfonts} 
\usepackage{amssymb} 
\usepackage{amsmath} 
\usepackage{amsthm} 
\usepackage{mathrsfs} 
\usepackage{mathtools} 
\usepackage{siunitx} 
\usepackage{bbm} 

\usepackage{graphicx} 
\usepackage{float} 
\usepackage{caption} 
\usepackage{subcaption}

\usepackage{tikz}
\usetikzlibrary{matrix}

\usepackage{algorithm}
\usepackage{algorithmic}

\usepackage{tabularx}

\usepackage{enumitem}
\newcommand{\subscript}[2]{$#1 _ #2$}
\newcommand{\concatenate}[2]{#1.#2}

\newenvironment{tight_itemize}{
\begin{itemize}
  \setlength{\itemsep}{0pt}
  \setlength{\parskip}{0pt}
}{\end{itemize}}

\newenvironment{tight_enumerate}{
\begin{enumerate}
  \setlength{\itemsep}{0pt}
  \setlength{\parskip}{0pt}
}{\end{enumerate}}

\DeclareMathOperator*{\argmin}{arg\,min}

\DeclareMathOperator*{\supess}{ess\,sup}

\newcommand*{\transp}[2][-1mu]{\ensuremath{\mskip1mu\prescript{\smash{\mathrm t \mkern#1}}{}{\mathstrut#2}}}
\DeclareMathOperator{\Tr}{Tr}

\newcommand{\norm}[1]{{\vert\kern-0.25ex\vert #1 
  \vert\kern-0.25ex\vert}}
\newcommand{\bignorm}[1]{{\big\vert\kern-0.25ex\big\vert #1 
  \big\vert\kern-0.25ex\big\vert}}

\newcommand{\opnorm}[1]{{\vert\kern-0.25ex\vert\kern-0.25ex\vert #1 
  \vert\kern-0.25ex\vert\kern-0.25ex\vert}}

\newcommand{\1}{\mathbbm{1}}

\newcommand{\lint}{[\![} 
\newcommand{\rint}{]\!]} 

\newcommand{\setN}{\mathbb{N}}

\newcommand{\salA}{\mathcal{A}}

\newcommand{\salC}{\mathcal{C}}

\newcommand{\salH}{\mathcal{H}}

\newcommand{\salL}{\mathcal{L}}

\newcommand{\salS}{\mathcal{S}}

\newcommand{\salX}{\mathcal{X}}

\newcommand{\zproba} { \mathbb{P} }

\newcommand{\ieproba}[3] { \mathbb{P}_{ #1 }^{ #2 }  ( #3  ) }

\newcommand{\zEX} { \mathbb{E} }
\newcommand{\EX}[1] { \mathbb{E}  [ #1  ] }

\newcommand{\ieEX}[3] { \mathbb{E}_{ #1 }^{ #2 }  [ #3  ] }

\DeclareMathOperator{\Cov}{Cov}
\DeclareMathOperator{\Var}{Var}

\newcommand{\gaussian}[2]{\mathcal{N}( #1, #2 )}

\newtheorem{theorem}{Theorem}
\newtheorem{definition}{Definition}
\newtheorem{assumption}{Assumption}
\newtheorem{proposition}{Proposition}

\newtheorem{lemma}{Lemma}

\allowdisplaybreaks

\begin{document}

\date{\today}

\title[Estimation of the Hurst parameter under high-frequency asymptotic]{Optimal estimation of the rough Hurst parameter in additive noise}

\author{Grégoire Szymanski}

\address{Grégoire Szymanski, Ecole Polytechnique, CMAP,
route de Saclay,
91128 Palaiseau
}
\email{gregoire.szymanski@polytechnique.edu}

\begin{abstract} 
We estimate the Hurst parameter $H \in (0,1)$ of a fractional Brownian motion from discrete noisy data, observed along a high frequency sampling scheme.
When the intensity $\tau_n$ of the noise is smaller in order than $n^{-H}$ we establish the LAN property with optimal rate $n^{-1/2}$. Otherwise, we establish that the minimax rate of convergence is  $(n/\tau_n^2)^{-1/(4H+2)}$ even when $\tau_n$ is of order 1. Our construction  of an optimal procedure relies on a Whittle type construction possibly pre-averaged, together with techniques developed in Fukasawa {\it et al.} \cite{fukasawa2019volatility}. We establish in all cases a central limit theorem with explicit variance, extending the classical results of Gloter and Hoffmann \cite{gloter2007estimation}.
\end{abstract}

\maketitle



\tableofcontents

\section{Introduction}

\subsection{Setting}
A fractional Brownian motion $W^H$ with Hurst index $H$ with $0 < H< 1$ is the unique self-similar Gaussian process with stationary increments such that 
$$\Cov(W^H_t, W^H_s) = \frac{1}{2}(t^{2H} + s^{2H} - |t-s|^{2H}).$$
Suppose we observe at discrete time a blurred version of the process $X_t = \sigma W^H_t$, where $\sigma > 0$: we have data at times $i/n, i= 0, \dots, n$ where each observation is contaminated by a noise. More precisely, we observe
\begin{align}
\label{eq:intro:model}
X_i^n = \sigma W^H_{i/n} + \tau_n \varepsilon_{i}^n,\;\;0 \leq i \leq n,
\end{align}
where the $\varepsilon_{i}^n$ are independent standard Gaussian variables independent from $W^H$ and $\tau_n$ is the noise level. 

The increments $\mathbf{Y}^n = (X_{i}^n-X_{i-1}^n)_{1 \leq i \leq n}$  of the observations $(X_i^n)_{0 \leq i \leq n}$ form a stationary Gaussian sequence, and in some parts of the paper,  we will actually consider a more general framework, where we observe
a stationary sequence $\mathbf{Y}^n$ having spectral density 
\begin{align*}
f_{H, \sigma}^{n}(\lambda) = \frac{\sigma^2}{n^{2\gamma(H)}} f_H(\lambda) + \tau_n^2 l(\lambda)
\end{align*}
where $\gamma(H)$ is explicit and $f_H$ and $l$ are explicit functions. In the case of \eqref{eq:intro:model}, the self similarity of the fractional Brownian motion gives $\gamma(H) = H$.\\

From these observations, our goal is to recover $H$ and $\sigma$.

\subsection{Main results and organization of the paper}

Several asymptotics are considered, depending on whether $\tau_n n^{\gamma(H)} \to 0$ or not.\\

In Section \ref{subsec:case1:estimator} we show how the method of \cite{fukasawa2019volatility} can readily be extended to the case $\tau_n n^{\gamma(H)} \to 0$: we obtain a consistent estimator $(\widehat{H}_n, \widehat{\sigma}_n)$ of $(H,\sigma)$ such that moreover:
\begin{align*}
\begin{pmatrix}
\sqrt{n}
( \widehat{H}_n - H)
\\
\frac{\sqrt{n}
}{\log n}
(
\widehat{\sigma}_n - \sigma
)
\end{pmatrix}
\to 
\begin{pmatrix}
X
\\
\sigma \gamma '( H ) X
\end{pmatrix}
\end{align*}
in distribution,  where $X$ is a Gaussian variable with explicit variance.  We thus have a central limit theorem with degenerate covariance matrix and unusual convergence rates. Extending the techniques of \cite{brouste2018lan}, we prove that these models satisfy the LAN property (Local Asymptotic Normality) in the sense of Ibragimov and Hasminskii. This property enables us to prove that this convergence rate together with the asymptotic variances are indeed optimal in a minimax sense, thanks to Hajek's convolution theorem (we refer to the classical textbook \cite{ibragimov2013statistical}).\\

In the second part of the paper, we relax the hypothesis $\tau_n n^{\gamma(H)} \to 0$ and treat in particular the most interesting case when $\tau_n$ of order $1$. The problem becomes substantially more intricate and is closer in spirit to the works of Gloter and Hoffmann \cite{gloter2007estimation}, and more generally statistics of random processes under additive noise (see for instance \cite{jacod2009microstructure, hoffmann2012adaptive} and the references therein). We restrict ourselves to the original model \eqref{eq:intro:model} and put $\tau_n=1$ for simplicity (and with no significant loss of generality). 
We use a pre-averaging technique suitably tuned by a two steps procedure in order to build a consistent estimator $(\widehat{H}^{\mathrm{opt}}_n, \widehat{\sigma}^{\mathrm{opt}}_n)$ that moreover satisfies
\begin{align*}
\begin{pmatrix}
n^{1/(4H+2)}
( \widehat{H}^{\mathrm{opt}}_n - H)
\\
\frac{n^{1/(4H+2)}
}{\log n}
(
\widehat{\sigma}^{\mathrm{opt}}_n - \sigma
)
\end{pmatrix}
\to 
\begin{pmatrix}
c_1(H) X
\\
c_2(H) X
\end{pmatrix}
\end{align*}
with explicit (and non-degenerate) constants $c_i(H)$. We recover in particular the speed $n^{1/(4H+2)}$ for every $0 < H < 1$ and we strictly extend the techniques developed in \cite{gloter2007estimation} to show that this rate is indeed optimal even for $H < 1/2$. The same procedure with an arbitrary $\tau_n$ would yield to the convergence rate $(n/\tau_n^2)^{1/(4H+2)}$.

\subsection{Relation to other works}

The estimation the Hurst parameter of a diffusion is a long-standing problem. Long-range dependence in Gaussian series were studied in depth by Fox and Taqqu in \cite{fox1987central, fox1986large} through Whittle estimators. Fractional Brownian motion with Hurst index $H>\frac{1}{2}$ shares these properties and is therefore encompassed in this framework. Istas and Lang  \cite{istas1997quadratic} put forward a new strategy by using  instead a more flexible estimator based on energy levels or quadratic variations across times scales. These were studied more precisely in \cite{coeurjolly2001estimating} in the specific case of the fractional Brownian motion with Hurst index $0<H<1$.\\ 

However, there is no observation noise in all these approaches, which is perhaps unrealistic in practice.  In \cite{gloter2007estimation}, model \eqref{eq:intro:model} is studied precisely with $\tau_n = 1$ and under the crucial assumption $H > 1/2$. A wavelet estimator based on energy levels thresholded as an optimal scale is shown to converge with the unusual rate $n^{-1/(4H+2)}$ and this rate is proved to be asymptotically minimax. The question whether this result holds in greater generality has been left open since. This is due in particular to the fact that when $H<1/2$, the Riemann approximations of the wavelet coefficients use by Gloter and Hoffmann are not sufficiently accurate. Our method overcomes this defect.\\

Closely related to the work of \cite{gloter2007estimation} is the estimation of the smoothness of the volatility in finance, as initiated by \cite{rosenbaum2008estimation}. 
A classic way to model the behaviour of the price of a financial asset $S$ is to use a continuous process of the form
\begin{align*}
dS_t = S_t ( \mu_t dt + \sigma_t dB_t).
\end{align*}
where $\mu_t$ is the drift and $\sigma_t$ the volatility. This diffusion is driven by a Brownian Motion $B$. The main subject of interest in financial markets is the volatility. Even if it was first thought to be constant (see \cite{black1973pricing}, it is now widely accepted that it is better modelled as a stochastic process, for both financial and empirical reasons. The groundbreacking result of Gatheral, Jaisson and Rosenbaum \cite{gatheral2018volatility} established that $\sigma_t$ follows a rough dynamic.  In \cite{gatheral2018volatility}, the authors propose the volatility model
\begin{align}
\label{eq:intro:rfsv}
\sigma_t = \exp \big(m + \int_{-\infty}^t e^{-\alpha(t-s)}dW_s^H\big)
\end{align}
where the integral is a pathwise Riemann–Stieltjes integral as explained in \cite{cheridito2003fractional}. Taking $H \approx 0.1$ is consistent to a number of empirical evidences that cannot be fitted within the Markov framework previously used in the literature. Since then, increasing evidence has been found for rough volatility, see \cite{bennedsen2016decoupling, jusselin2020noarbitrage, livieri2018rough} and some alternative models have been proposed, see \cite{bayer2016pricing, eleuch2019roughening}. However, as noted in \cite{fukasawa2019volatility}, the fits $H \approx 0.1$ obtained in \cite{gatheral2018volatility} should not be understood as statistical estimators, but rather as calibrated best fits in order to model the data.

But inferring the Hurst parameter from market data is a bit more complex since volatility is not directly observable. Instead, we can only access the asset price at discrete times, say at time $i/n$, and a proxy of the instant volatility is used, such as
\begin{align*}
\sigma^{2}_{i/n} \approx n (S_{(i+1)/n} - S_{i/n})^2.
\end{align*}

Using this proxy lead to a multiplicative noise instead of the additive noise of model \eqref{eq:intro:model}. In the long memory case, i.e. when $H>\frac{1}{2}$, Rosenbaum developed in \cite{rosenbaum2008estimation} an efficient estimator  in model \eqref{eq:intro:rfsv} extending the previous work of Gloter and Hoffmann \cite{ gloter2007estimation}. This estimator also has the unusual rate $n^{-1/(4H+2)}$ and this rate is optimal in this model. Several attempts were made in these directions in the past years, see {\it e.g.} \cite{fukasawa2019volatility} or  \cite{bolko2020GMM} in order to extend these results to the rough case $0 < H < 1/2$, and in particular with a rough volatility index $H \approx 0.1$.  However, all these approaches have to make significant simplifications or extraneous assumptions and a complete picture yet needs to be found. Other approaches include the estimation of the smoothness of random curves from several independent observations (see \cite{golovkine2022learning}). Noisy continuous observations of a fractional Brownian motion could also be considered as in \cite{dozzi2013asymptotic, chigansky2022estimation}. When polluted with a (continuous) noise os size $\varepsilon$, the unusual rate $\varepsilon^{1/(4H+2)}$ also appears to be optimal (see \cite{chigansky2022estimation}).\\  

In this paper, we remain more modest, keeping-up to the additive case \eqref{eq:intro:model}, but we give a complete picture of the situation that considerably generalises the results of \cite{gloter2007estimation} by extending their work to rough exponents and by showing central limit theorems for our estimators.

\section{Fast decaying noise}
\label{sec:case1}

This corresponds to the case $\tau_n n^{\gamma(H)} \to 0$.

\subsection{Construction of the model}

Let $0 < \sigma_- < \sigma_+ < \infty$, $0 < H_- < H_+ < 1$ and $\Theta = [H_-, H_+] \times [\sigma_-, \sigma_+]$. We write $\theta = (H, \sigma)$ the elements of $\Theta$. For each $n$, consider a probability space $( \Omega^n, \salA^n )$ on which we define random variables $\mathbf{Y}^{n} = ( Y_1^{n}, \dots, Y_n^{n})$ and a family of probability measures $\zproba_{H, \sigma}^{n}$ for each $(H, \sigma) \in \Theta$ such that $\mathbf{Y}^{n}$ is a Gaussian process with spectral density function 
\begin{align}
\label{eq:decaying:psd}
f_{H, \sigma}^{n}(\lambda) = \frac{\sigma^2}{n^{2\gamma(H)}} f_H(\lambda) + \tau_n^2 l(\lambda)
\end{align}
under $\zproba_{H, \sigma}^{n}$, where $f_H$ and $l$ are two known spectral densities and $\gamma: [ H_-, H_+  ] \rightarrow (0,\infty)$ is increasing and differentiable. We also need the following regularity properties for $f_H$.

\begin{assumption}
\label{assumption:decaying:psd}
The spectral density $f_H$ satisfies the following properties:
\begin{enumerate}[label=(\subscript{H}{{\arabic*}}), ref=(\subscript{H}{{\arabic*}})]
  \setlength{\itemsep}{0pt}
  \setlength{\parskip}{0pt}
\item \label{assumption:decaying:psd:definition}
$\lambda \in [-\pi,\pi] \backslash \{ 0 \} \mapsto  f_H(\lambda)$ is even, nonnegative and integrable for every $H \in [H_- ,H_+]$.
\item\label{assumption:decaying:psd:regularity} $(\lambda, H) \mapsto f_H(\lambda) \in \salC^{3,2} (  [ H_-, H_+  ] \times [-\pi,\pi] \backslash \{ 0 \} )$.
\item \label{assumption:decaying:psd:identifiability}
If $( H_1, \sigma_1 ),( H_2, \sigma_2 )$ are two distinct elements of $ [ H_-, H_+  ] \times [\sigma_-,\sigma_+]$, then the set $\{\lambda \in [-\pi, \pi]:\; \sigma_1^2 f(\lambda, H_1) \neq \sigma_2^2 f(\lambda, H_2)\}$ has positive Lebesgue measure.
\item \label{assumption:decaying:psd:bounds}
There exists  a continuous and increasing function $\alpha :  [ H_-, H_+  ] \to (-1,1)$ such that
\begin{align*}
c_1  | \lambda  |^{-\alpha(H)}
\leq 
f(\lambda, H)
\leq 
c_2  | \lambda  |^{-\alpha(H)}.
\end{align*}
for some $c_1, c_2 >0$ and every 
$( H, \lambda ) \in  [ H_-, H_+  ] \times [-\pi, \pi] \backslash \{ 0 \}$.

\item \label{assumption:decaying:psd:bounds_derf}
For any $r > 0$, there exists $c_3 = c_3(r)$ such that \begin{align*}
 |
\partial_\lambda^\ell \partial_H^j f_H(\lambda)
|
\leq \frac{c_3(r)}{ | \lambda  |^{\alpha(H) +\ell + r}}.
\end{align*}
for every $( H, \lambda ) \in  [ H_-, H_+  ] \times [-\pi, \pi] \backslash \{ 0 \}$
and $j=0,1,2,3, \ell=0,1,2$.

\end{enumerate}
\end{assumption}

Assumption \ref{assumption:decaying:psd} is satisfied in most cases, see {\it e.g.} \cite{fukasawa2019asymptotically} when $Y_i^{n}$ is the increment of a fractional Brownian motion. Eventually, we require the noise spectral density $l$ to be controlled near $0$. This is used when applying results that control the convergence of Gaussian quadratic forms, 
see for instance \cite{fox1986large} and \cite{fox1987central}.

\begin{assumption}
\label{assumption:decaying:noise}
There exists $c_4 > 0$ such that
$| l(\lambda)| \leq c_4 \lambda ^2 $ for any $-\pi \leq \lambda \leq \pi$.
\end{assumption}

Assumption \ref{assumption:decaying:noise} can be relaxed (at a slight notational technical cost) as follows: we may allow $l$ to depend of $(H, \sigma)$ too (but not $c_4$). The construction of the estimator in section \ref{subsec:case1:estimator} then has to be changed accordingly.

\subsection{Construction of the estimator}
\label{subsec:case1:estimator}

As explained in \cite{brouste2018lan}, \cite{fukasawa2019asymptotically} and \cite{fukasawa2019volatility}, the effects of $H$ and $\sigma$ interplay in the behaviour of estimators asymptotically as $n \rightarrow \infty$: slight variations of $H$ can be offset by slight variations of $\sigma$, resulting in degenerate Fisher information matrix, see \cite{kawai2013fisher}. To correct these effects, we use techniques developed in \cite{fukasawa2019volatility} using a reparametrisation of our model. 
We introduce a new parameter $\nu$ defined as
\begin{align*}
\nu = n^{-\gamma(H)} \sigma \in  [ n^{-\gamma(H_+)}\sigma_-, n^{-\gamma(H_-)}\sigma_+ ].
\end{align*}
We will infer $(H, \sigma)$ through the reparametrisation $(H,\nu)$ and recover $\sigma$ via $\sigma = \nu n^{\gamma(H)}$. Therefore we define the spectral density of $\mathbf{Y}^{n}$ with respect to $( H, \nu )$ instead of $( H, \sigma )$ by setting:
\begin{align*}
g_{H, \nu}^{n}(\lambda) = \nu^2 f_H(\lambda) + \tau_n^2 l(\lambda),
\end{align*}
and introduce a Whittle (negative) contrast with respect to parameters $( H,\nu )$ defined as
\begin{align*}
U_n 
( 
H, \nu
)
=
\frac{1}{4\pi}
\int_{-\pi}^\pi
\log 
( 
g^n_{H,\nu}(\lambda)
)
+
\frac{
I_n( \lambda, \mathbf{Y}^n )
}{
g^n_{H,\nu}(\lambda)
}
d\lambda
\end{align*}
where 
$$I_n( \lambda, \mathbf{Y}^{n} ) = n^{-1}  | \sum_{k=1}^n e^{ik\lambda} Y_k^{n}  |^2$$ is the periodogram associated to observations $\mathbf{Y}^n$.  
We define an estimator of the parameters $( H, \nu )$ through the minimisation of the Whittle contrast
\begin{align}
\label{eq:decaying:wle}
( \widehat{H}_n, \widehat{\nu}_n ) = \argmin U_n (H, \nu)
\end{align}
where the minimum is taken over all $( H, \nu ) \in [H_-, H_+]\times  [ n^{-\gamma(H_+)}\sigma_-, n^{-\gamma(H_-)}\sigma_+ ]$. Since our parameter of interest is $\sigma = n^{\gamma( H)} \nu$, we naturally define $\widehat{\sigma}_n = n^{\gamma( \widehat{H}_n )}\widehat{\nu}_n$. 

\begin{theorem}
\label{thm:decaying}
Work under Assumptions \ref{assumption:decaying:psd} and \ref{assumption:decaying:noise} and suppose in addition that
$( n^{\gamma( H_+ )} \tau_n )_n$ is bounded.
Let $( H, \sigma) $ be in the interior of $\Theta$. Then 
\begin{align*}
\begin{pmatrix}
\sqrt{n}
( \widehat{H}_n - H)
\\
\frac{\sqrt{n}
}{\log n}
(
\widehat{\sigma}_n - \sigma
)
\end{pmatrix}
\to 
\begin{pmatrix}
X
\\
\sigma \gamma '( H ) X
\end{pmatrix}
\end{align*}
under $\zproba_{H, \sigma}^n$ in distribution, where $X$ is a centred Gaussian variable with variance
\begin{align*}
2
\bigg(			\frac{1}{2\pi}
			\int_{-\pi}^\pi 
			\Big(
			\frac{\partial_H f_{H} (\lambda) }
			{ f_{H}(\lambda) } \Big)^2 d\lambda
		-
			\Big(
			\frac{1}{2\pi}
			\int_{-\pi}^\pi 
	\frac{\partial_H f_{H} (\lambda) }{ f_{H}(\lambda) } d\lambda
			\Big)^2
\bigg)^{-1}.
\end{align*}
\end{theorem}

The proof of Theorem \ref{thm:decaying} is similar in essence to that of Theorem 2.8 in \cite{fukasawa2019volatility} with the slight improvement of a central limit theorem.

\subsection{The LAN property}

In this section, we prove a structural result about the statistical experiment $\{\mathbb P_{\sigma, H}^n, (\sigma, H) \in \Theta\}$ via the LAN property. It implies in particular a lower bound for the achievable convergence rate of an estimator; as a consequence, we show that the rate of Theorem \ref{thm:decaying} is optimal in a minimax sense. 
Recall the classical definition (see {\it e.g.} \cite{ibragimov2013statistical}) of a regular LAN experiment:

\begin{definition}
\label{def:decaying:LAN}
Let $\Theta \subset \mathbb{R}^d$. A (sequence of) family of dominated measures $\{ P^n_\theta, \theta \in \Theta \}_{n}$ is LAN (locally asymptotically normal), at $\theta_0$ in the interior of $\Theta$ if there exist two non degenerate $d\times d$ matrices $\varphi_n = \varphi_n ( \theta_0 )$ and $I = I ( \theta_0 )$ such that for any $u \in \mathbb{R}^d$ such that $\theta_0 + \varphi_n u \in \Theta$, the likelihood ratio has representation
\begin{align*}
\frac{dP^n_{\theta_0 + \varphi_n u}}{dP^n_{\theta_0}}
=
\exp (
\transp{u}
\zeta_n - \tfrac{1}{2}\transp{u} I u + r_n
)
\end{align*}
where $\zeta_n = \zeta_n(\theta_0)$ is a $d$-dimensional random vector that converges in $P^n_{\theta_0}$-distribution toward a standard normal $\gaussian{0}{I}$ and $r_n = r_n(\theta_0, u)$ converges to $0$ in $P^n_{\theta_0}$-probability.
\end{definition}

In particular, we have Hajek's convolution theorem (see \cite{hajek1972local} and \cite{lecam1972limits}):
\begin{theorem}[Theorem II.12.1 in \cite{ibragimov2013statistical}]
\label{thm:hajek_lecam}
Suppose that the (sequence of) dominated experiment(s) $\{ P^n_\theta, \theta \in \Theta \}_{n}$ is LAN at $\theta_0$ in the interior of $\Theta \subset \mathbb R^d$. Then for any sequence of estimators $\widehat{\theta}_n$ and any symmetric nonnegative quasi-convex function $L$ such that $e^{-\varepsilon |z|^2} L(z) \to 0$ as $|z| \to \infty$ and any $\varepsilon > 0$, we have
\begin{align*}
\liminf_{\delta \to 0}
\liminf_{n \to \infty}
\sup_{|\varphi_n^{-1} (\theta - \theta_0)| < \delta} 
\ieEX{\theta}{n}{L(\varphi_n^{-1}(\widehat{\theta}_n - \theta))}
\geq 
(2\pi)^{-d/2}\int_{\mathbb{R}^d} L( I^{-1/2} z ) \exp(-\tfrac{|z|^2}{2})dz
\end{align*}
where $\phi$ is the density of the $d$-dimensional standard Gaussian distribution.
\end{theorem}



\begin{theorem}
\label{thm:decaying:LAN}
Work under Assumptions \ref{assumption:decaying:psd} and \ref{assumption:decaying:noise}. Assume moreover that
$( n^{\gamma( H_+ )} \tau_n )_n$ is bounded. Consider $\theta_0 = ( H_0, \sigma_0 )$ a point in the interior of $\Theta$. Define
\begin{align*}
\varphi_n = \frac{1}{\sqrt{n}}
\begin{pmatrix}
\alpha_n &\overline{\alpha_n}
\\
\beta_n &\overline{\beta_n}
\end{pmatrix},
\end{align*}
and suppose in addition that it satisfies the following properties:
\begin{tight_enumerate}
\item $\varphi_n$ is non-degenerate (i.e. $\alpha_n \overline{\beta_n} \neq \beta_n\overline{\alpha_n}$),
\item $\alpha_n \to \alpha$ and $\overline{\alpha_n} \to \overline{\alpha}$,
\item $\gamma_n = \beta_n\sigma_0^{-1} - \alpha_n \gamma'(H_0) \log(n) \to \gamma$ and $\overline{\gamma_n} = \overline{\beta_n} \sigma_0^{-1} - \overline{\alpha_n} \gamma'(H_0)\log(n) \to \overline{\gamma}$, 
\item $\alpha\overline{\gamma} - \overline{\alpha}\gamma \neq 0$.
\end{tight_enumerate}

Then the family $\{\mathbb P_{H,\sigma}, (H,\sigma) \in \Theta\}_n$ is LAN at $\vartheta_0$ with rate matrices  $\varphi_n$ and Fisher information
\begin{align*}
I =
\frac{1}{4\pi}
\begin{pmatrix}
\gamma & -\alpha
\\
\overline{\gamma} & -\overline{\alpha}
\end{pmatrix}
\begin{pmatrix}
8\pi
&
-2\int \frac{ \partial_H f_{H_0} }{f_{H_0}}
\\
-2\int \frac{ \partial_H f_{H_0} }{f_{H_0}}
&
\int (\frac{ \partial_H f_{H_0} }{f_{H_0}})^2
\end{pmatrix}
\begin{pmatrix}
\gamma&\overline{\gamma}
\\
-\alpha&-\overline{\alpha}
\end{pmatrix}.
\end{align*}

\end{theorem} 
The proof of this theorem closely follows \cite{brouste2018lan} and is delayed until Section \ref{sec:proof:LAN}. The normalisation $\varphi_n$ originates from the work of Brouste and Fukasawa in \cite{brouste2018lan} in their Theorem 3.1. Here, we have a slight extension to more general self-similar random processes and allow for high frequency observations to be perturbed by a fast decaying noise. Following again \cite{brouste2018lan}, we identify several specific rate matrices and use them to compute the optimal estimation rates for $H$ and $\sigma$. More precisely, we have the following:
\begin{theorem}
\label{thm:decaying:optimalrates}
Let $\widehat{\theta}_n = ( \widehat{H}_n, \widehat{\sigma}_n )$ be a sequence of estimators of $( H, \sigma )$ and let $\theta_0 =(H_0, \sigma_0)$ be in the interior of $\Theta$. Then 
\begin{align*}
&\liminf_{\delta \to 0}
\liminf_{n \to \infty}
\sup_{|\varphi_n^{-1} (\theta - \theta_0)| < \delta} 
n \ieEX{\theta}{n}{(\widehat{H}_n - H)^2 }
\geq 
v_0^2
\\
&\liminf_{\delta \to 0}
\liminf_{n \to \infty}
\sup_{|\varphi_n^{-1} (\theta - \theta_0)| < \delta} 
\frac{n}{\log(n)^2} \ieEX{\theta}{n}{(\widehat{\sigma}_n - \sigma)^2 }
\geq 
v_0^2 \sigma_0^2\gamma'(H_0)^2
\end{align*}
where
$v_0 = 
\sqrt{2}
\bigg(
\frac{1}{2\pi}
\int
\Big (
	\frac
		{\partial_H f_{H_0}  }
		{f_{H_0}}
\Big)^2
-
\bigg(
\frac{1}{2\pi}
\int
	\frac
		{\partial_H f_{H_0}  }
		{f_{H_0}}
\bigg)^2
\bigg )^{-1/2}
$.
\end{theorem}

\begin{proof}
The proof consists in applying Theorems \ref{thm:hajek_lecam} and \ref{thm:decaying:LAN}  for successive suitable choices of $\varphi_n$ and $L$.
Let $\varphi_n = \varphi_n(H, \sigma)$ be the rate matrix defined by:
\begin{align*}
\varphi_n = \frac{1}{\sqrt{n}} 
\begin{pmatrix}
1 & 0
\\
\sigma \gamma'(H) \log n & 1
\end{pmatrix}.
\end{align*}
It satisfies the hypotheses of Theorem \ref{thm:decaying:LAN} with $\alpha = 1$, $\overline{\alpha} = 0$, $\gamma = 0$ and $\overline{\gamma} = \sigma^{-1}$ and we have:
\begin{align*}
I 
&=
\begin{pmatrix}
\frac{1}{4\pi}
\int
(
	\frac
		{\partial_H f_{H_0}  }
		{f_{H_0}}
)^2
&
\frac{1}{2\sigma\pi}
\int
	\frac
		{\partial_H f_{H_0}  }
		{f_{H_0}}
\\
\frac{1}{2\sigma\pi}
\int
	\frac
		{\partial_H f_{H_0}  }
		{f_{H_0}}
&
	\frac{2}{\sigma^2}
\end{pmatrix}.
\end{align*}
Theorem \ref{thm:hajek_lecam} with $L(x,y) = x^2$ gives
the first part of Theorem \ref{thm:decaying:optimalrates}.
For the second part, we consider $\varphi_n = \varphi_n(H,\sigma)$ defined by
\begin{align*}
\varphi_n = \frac{1}{\sqrt{n}} 
\begin{pmatrix}
- \log(n)^{-1} & 1
\\
0 & \sigma \gamma'(H_0) \log n
\end{pmatrix}.
\end{align*}
and the conclusion follows.
\end{proof}

\section{Slowly decaying or non-vanishing noise}
\label{sec:slow}

This corresponds to the case where the condition $\tau_n n^{\gamma(H)} \to 0$ of section \ref{sec:case1} fails. For simplicity, we will only consider here the case where $\tau_n = \tau >0$ is constant, but the  subsequent results readily carry over to more general (bounded) sequences.
 
\subsection{Presentation of the model}

Let  $\Theta = [H_-, H_+] \times [\sigma_-, \sigma_+] \subset ( 0, 1  ) \times (0, \infty)$ with a non-empty interior and consider a point $( H_0, \sigma_0 )$ in this interior.

We consider for each $n$ a probability space $( \Omega^n, \salA^n )$ on which we define random variables $\mathbf{Z}^{n} = ( Z_1^{n}, \dots, Z_n^{n})$ and a family of probability measures $\zproba_{H, \sigma}^{n}$ for each $(H, \sigma) \in \Theta$ such that under $\zproba_{H, \sigma}^{n}$, the observation $\mathbf{Z}^{n}$ is given by 
$$Z_i^{n} = \sigma W_{i/n}^H + \tau \xi_i^n,\;\; i= 0, \ldots, n,
$$
where the $\xi_i^n$ are independent standard Gaussian variables, independent from a fractional Brownian motion $W^H$ with Hurst parameter $H$ and where $\tau > 0$ is a constant.

\subsection{Construction of the estimator}
\label{subsec:slow:construction}

The construction of the estimator is divided into two steps. First we build a pilot estimator with sub-optimal rate of convergence and then we use it to fine tune the procedure and get an estimator with an optimal convergence rate.\\

Because of the persisting noise ($\tau_n = \tau$ is constant), we cannot apply results from section \ref{sec:case1} and in particular Theorem \ref{thm:decaying}.
To remove this noise, we use a pre-averaging of the data to regularise them. More precisely, we consider sequences $k = k(n)$ and $N = N(n) =  \lfloor n/k  \rfloor$ and we define for $0 \leq i \leq N-1$:
\begin{align*}
\widehat{Z}_i^{n} 
&=
\frac{1}{k}
\sum_{j=0}^{k-1} Z_{ik+j}^{n} 
=
\frac{\sigma}{k}
\sum_{j=0}^{k-1} W_{\frac{ik+j}{n}}^{H} 
+ \frac{\tau}{\sqrt{k}}\widehat{\xi}^n_i
\end{align*}
where $(\widehat{\xi}^n_i = k^{-1/2}
\sum_{j=0}^{k-1} \xi_{ik+j}^{n} )_i$ is a family of i.i.d. standard Gaussian independent of $W^H$. This sequence has stationary increments and we write $\mathbf{Y}^{n} = ( Y_i^{n})_{i=1, \dots, n} = ( \Delta \widehat{Z}_i^{n})_i$ these increments. Since the fractional Brownian motion has a scale invariance property, its power spectral density is given by 
\begin{align}
\label{eq:slow:spectral}
f^{n}_{H,\sigma} (\lambda) 
=
\sigma^2\Big( \frac{k}{n} \Big)^{2H} f_{H, k} ( \lambda ) + \frac{\tau^2}{k} l( \lambda )
\end{align}
where $f_{H, k}$ is the power spectral density of the stationary process 
$ \big( \frac{1}{k}
\sum_{j=0}^{k-1} W_{i+1+j/k}^{H} - W_{i+j/k}^{H}\big)_i$  and where $l( \lambda ) = 2 (1 - \cos \lambda)$ is the power spectral density of $\Delta  \widehat{\xi}^n$. We also introduce $f_{H,\infty}$ as the power spectral density of the process $\big( \int_i^{i+1} W^H_{s} - W^H_{s-1} ds \big)_i$. It can be seen as the limit of the functions $(f_{H, k})_k$ (see Lemma \ref{lemma:slow:preaveraged_psd:uniform_conv}). Note that explicit formulas for  $f_{H, k}$ and $f_{H, \infty}$ can also be found in Lemma \ref{lemma:slow:preaveraged_psd}.\\

Then we mimick the approach of Section \ref{subsec:case1:estimator} and we consider a re-parametrisation of our model. More precisely, we write $\nu = \sigma ( k/n )^{H}$ and we define $g^{n}_{H,\nu} = \nu^2 f_{H, k} + \tau^2 l/k$ the power spectral density of $\mathbf{Y}^{n}$ with respect to this new parameter. Then we build a Whittle maximum likelihood estimator of $( H, \widetilde{\nu} )$ from the observations $\mathbf{Y}^{n}$. In that extend, we define a Whittle contrast $U_n 
( 
H, \nu
)$ by
\begin{align*}
U_n 
( 
H, \nu
)
=
\frac{1}{4\pi}
\int_{-\pi}^\pi
\log 
( 
g^{n}_{H,\nu}(\lambda)
)
+
\frac{
I_N( \lambda, \mathbf{Y}^{n} )
}{
g^{n}_{H,\nu}(\lambda)
}
d\lambda
\end{align*}
where $I_N( \lambda, \mathbf{Y}^{n} ) = N^{-1}  | \sum_{k=1}^N e^{ik\lambda} Y_k^{n}  |^2$ is the periodogram associated to observations $\mathbf{Y}^n$. Note that the vector $\mathbf{Y}^n$ is of size $N$ so we need to consider the $N$-periodogram here. Then we define a first estimator by minimizing $U_n$ over $( H, \nu ) \in [H_-, H_+]\times  [ ( k/n)^{H_+}\sigma_-, ( k/n )^{H_-} \sigma_+  ]$. We write
\begin{align*}
( \widehat{H}_n, \widehat{\nu}_n ) = \argmin U_n.
\end{align*}
Since  our parameter of interest is $\sigma = \nu ( n/k )^{H}$, we rather define an estimator of $\sigma_0$ by $\widehat{\sigma}_n = \widehat{\nu}_n ( n/k )^{\widehat{H}_n}$. This estimator is consistent under some hypothesis on the choice of $( k(n) )_n$ though its convergence rate is not optimal. Since it is a major milestone in our proof, we state directly a convergence result for this estimator:
\begin{proposition}
\label{propo:slow:first}
Suppose that $k(n) \to \infty$, $\frac{n}{k(n)} \to \infty$, and that $\frac{n^{2H_+}}{k^{2H_++1}}$ is bounded. Then we have
\begin{align*}
\begin{pmatrix}
\sqrt{N}
( \widehat{H}_n - H_0)
\\
\frac{\sqrt{N}
}{\log N}
(
\widehat{\sigma}_n - \sigma_0
)
\end{pmatrix}
\to 
\begin{pmatrix}
X
\\
\sigma_0 X
\end{pmatrix}
\end{align*}
 in $\zproba_{H_0, \sigma_0}^n$-distribution  where $X$ has explicit variance (the same as in Theorem \ref{thm:decaying}).
\end{proposition}

The best choice of sequence $k(n)$ satisfying all the hypotheses of this result is given by $k(n) = n^{2H_+/(2H_++1)} (1+o(1))$, which gives $N = n^{1/(2H_++1)}(1+o(1))$. In that case, we obtain a convergence rate $n^{1/(4H_++2)}$ for the estimation of $H$.\\

We now use this consistent estimator $\widehat{H}_n$ of $H_0$ that is asymptotically Gaussian with convergence rate $v_n = n^{1/(4H_++2)}$ to fine-tune this procedure and build an estimator with convergence rate $v_n = n^{1/(4H_0+2)}$. Notice here that we could use another consistent pilot estimator $\widehat{H}_n$, as long as it has convergence rate $v_n \to \infty$.\\

Let $q=q(n)$ be a fixed sequence of positive real numbers and $m = m(n)$ be a sequence of positive integers. Some assumptions on these sequences will be precised later on. We also introduce $U$ a new random variable independent of $W^H$ and $\xi$, uniformly distributed on $[0,1]$. We first discretise the interval $[H_-, H_+]$ with mesh $m(n)^{-1}(H_+ - H_-)$ and we take an estimate of $H_0$ on this grid using $\widehat{H}_n$:
\begin{align*}
\widehat{h}_n = H_- + \frac{\widehat{i}_n}{m(n)}( H_+ - H_- ) \text{ where } \widehat{i}_n =  \Big \lceil m(n) \frac{\widehat{H}_n - H_- + q(n)}{H_+ - H_-} + U  \Big \rceil.
\end{align*}

The presence of $q(n)$ in this expression ensures that $\widehat{h}_n > H_0$ with high probability while $U$ is used to prevent some dependency issues between the pilot estimator $\widehat{H}_n$ and the following.\\

Then we preaverage our previous data, but this time, we choose a special random sequence $
k^{\mathrm{opt}}(n) =  \lfloor n^{2\widehat{h}_n/(2\widehat{h}_n+1)}  \rfloor
$. Recall in addition that $N^{\mathrm{opt}}(n) =  \lfloor n / k^{\mathrm{opt}} \rfloor$ is now also random. 
We define a new estimator based on the preaveraging:
\begin{align*}
( \widehat{H}_n^{\mathrm{opt}}, \widehat{\nu}_n^{\mathrm{opt}} ) = \argmin U_n^{\mathrm{opt}} (H, \nu) \;\text{ and }\;
\widehat{\sigma}^{\mathrm{opt}}_n = \widehat{\nu}_n( n/k)^{\widehat{H}^{\mathrm{opt}}_n}
\end{align*}
where $U_n^{\mathrm{opt}}$ is defined as $U_n$ but using $k^{\mathrm{opt}}(n)$ instead of $k(n)$ and where the minimum is taken over $( H, \eta ) \in [H_-, \widehat{h}_n] \times  [ ( k/n)^{H_+}\sigma_-, ( k/n )^{H_-} \sigma_+  ]$. We obtain
\begin{theorem}
\label{thm:slow:optimal}
Suppose that $\log(n) q(n) \to \delta^* \in (0,\infty)$, $m(n)^{-1}\log n \to 0$ and $m(n)^{-1}v_n \to \infty$. Write
$$
\gamma^*
=
\exp
(
-2\delta^*/(2H_0+1)
)
.$$
Then 
\begin{align*}
\begin{pmatrix}
n^{1/(4H_0+2)}
( \widehat{H}^{\mathrm{opt}}_n - H_0)
\\
\frac{n^{1/(4H_0+2)}
}{\log n}
(
\widehat{\sigma}^{\mathrm{opt}}_n - \sigma_0
)
\end{pmatrix}
\to 
\begin{pmatrix}
(\gamma^*)^{-1/(4H_0+2)} X
\\
\frac{(\gamma^*)^{-1/(4H_0+2)} \sigma_0}{2H_0+1}  X
\end{pmatrix}
\end{align*}
under $\zproba_{H_0, \sigma_0}^n$ in distribution, where $X$ is a centred Gaussian variable with variance
\begin{align*}
	\frac{
	\frac{4\pi}{\sigma_0^4}
		\int_{-\pi}^\pi 
		\frac{f_{H_0, \infty}(\lambda)^2}
		{f^\ast_{H_0,\sigma_0}(\lambda)} 
		d\lambda
	}{
	(	
	\int_{-\pi}^\pi 
	\frac{f_{H_0, \infty}(\lambda)^2}
	{f^\ast_{H_0,\sigma_0}(\lambda)} 
	d\lambda
	)
	(
	\int_{-\pi}^\pi 
	\frac{\partial_H f_{H_0, \infty} (\lambda)^2}
	{f^\ast_{H_0,\sigma_0}(\lambda)} 
	d\lambda
	)
	-
	(
	\int_{-\pi}^\pi 
	\frac{\partial_H f_{H_0, \infty} (\lambda)f_{H_0, \infty}(\lambda)}
	{f^\ast_{H_0,\sigma_0}(\lambda)} 
	d\lambda
	)^2
	},
\end{align*}
where we denote
$$f^\ast_{H_0,\sigma_0}(\lambda) = ( \sigma_0^2 f_{H_0, \infty}(\lambda) + \gamma^* \tau^2 l(\lambda) )^2.$$
\end{theorem}

\subsection{Optimality}

In the previous subsection, we have found a consistent estimator of $H_0$ and $\sigma_0$ converging respectively at rate $n^{1/(4H_0+2)}$ and
$n^{1/(4H_0+2)} \log(n)^{-1}$. The optimality of this convergence rate is proven in \cite{gloter2007estimation} when $H > 1/2$. In Theorem \ref{thm:slow:optimality} below, we  extend this result to $H\leq 1/2$. First, we say that $v_n = v_n(H)$ is a lower rate of convergence over $\Theta$ for the estimation of $H$ if there exists $c > 0$ such that 
\begin{align*}
\liminf_{n \to \infty}
\inf_{\widehat{H}}
\sup_{( H,\sigma ) \in \Theta}
\ieproba{H,\sigma}{n}
{
v_n^{-1}
 |
\widehat{H} - H
 |
\geq c
}
> 0
\end{align*}
where the infimum is taken over all estimators $\widehat{H}$. The same definition holds for $\sigma$ instead of $H$. 

\begin{theorem}
\label{thm:slow:optimality}
The rate $v_n = v_n(H) = n^{-1/(4H+2)}$ is a lower rate of convergence for the estimation of $0 < H < 1$ and $v_n = v_n(H) = n^{-1/(4H+2)}\log(n)$ is a lower rate of convergence for the estimation of $\sigma$.
\end{theorem}

\section{Proof of Theorem \ref{thm:decaying}}
\label{sec:decaying:proof}

\subsection{Outline of the proof}
\label{subsec:decaying:outline}

Recall from Equation \eqref{eq:decaying:wle} that $( \widehat{H}_n, \widehat{\sigma}_n )$ is defined through the parametrisation $( H, \nu )$ by
\begin{align*}
( \widehat{H}_n, \widehat{\nu}_n ) = \argmin U_n
\text{ and } \widehat{\sigma}_n = n^{\gamma( \widehat{H}_n )}\widehat{\nu}_n
\end{align*}
where the minimum is taken over $[H_-, H_+] \times  [ \sigma_-n^{-\gamma( H_+)}, \sigma_+n^{-\gamma( H_-)}  ]$.
But $\sigma_0n^{-\gamma( H_0)}$ depends on $n$ and converges to $0$. Therefore, following \cite{fukasawa2019volatility}, we introduce a rescaled process $\mathbf{\widetilde{Y}}^{n}$ by
\begin{align*}
\mathbf{\widetilde{Y}}^{n} = n^{\gamma( H_0)} {\mathbf{Y}}^{n} = ( n^{\gamma( H_0)} Y_1^{n}, \dots, n^{\gamma( H_0)}Y_n^{n}).
\end{align*}
This process $\mathbf{\widetilde{Y}}^n$ is unobservable since $H_0$ is not known a priori. Its spectral density is given by:
\begin{align*}
\widetilde{g}^{n}_{H,\widetilde \nu}(\lambda) = \widetilde{\nu}^2 f_H(\lambda) + n^{2\gamma( H_0)} \tau_n^2 l(\lambda) \;\text{ with }\; \widetilde{\nu} := n^{\gamma( H_0)} \nu.
\end{align*}
We also define the Whittle oracle likelihood with respect to $( H, \widetilde{\nu})$ with unobserved process $\widetilde{\mathbf{Y}}^{n}$:
\begin{align}
\label{eq:decaying:Utilde}
\widetilde{U}_n 
( 
H, \widetilde \nu
)
=
\frac{1}{4\pi}
\int_{-\pi}^\pi
\log 
( 
\widetilde{g}^{n}_{H,\widetilde \nu}(\lambda)
)
+
\frac{
I_n\big( \lambda, \mathbf{\widetilde{Y}}^n \big) 
}{
\widetilde{g}^{n}_{H,\widetilde \nu}(\lambda)
}
d\lambda.
\end{align}
It is defined for $( H, \widetilde \nu ) \in \widetilde \Theta^{n} =  [ H_-, H_+  ] \times   [ \sigma_-n^{\gamma(H_0)-\gamma(H_+)}, \sigma_+n^{\gamma(H_0)-\gamma(H_-)}  ]$. The main properties of $\widetilde{U}_n$ are summarized in the following proposition:
\begin{proposition}
\label{propo:decaying:likelihood}
Let $K$ be a compact interval of $(0,\infty)$ and $\delta > 0$. Let also $\nabla$ stand for $\begin{pmatrix}
\partial_H \\ \partial_{\widetilde \nu}
\end{pmatrix}$. We define 
$H_-(\delta) = \inf  \{ H \in  [H_-, H_+  ] :  \alpha(H) \geq \alpha(H_0) - 1 + \delta  \}
$ for $\delta > 0$ and
\begin{align}
\label{eq:decaying:tilde_U_infty}
\widetilde{U}_\infty
( 
H, \widetilde \nu
)
=
\frac{1}{4\pi}
\int_{-\pi}^\pi
\log 
( 
\widetilde \nu^2 f_H(\lambda)
)
+
\frac{
\sigma_0^2 f_{H_0}(\lambda)
}{
\widetilde \nu^2 f_H(\lambda)
}
d\lambda.
\end{align}
Then we have:
\begin{enumerate}[label=(\roman*), ref=\concatenate{\ref{propo:decaying:likelihood}}{{(\roman*)}}]
  \setlength{\itemsep}{0pt}
  \setlength{\parskip}{0pt}
\item \label{propo:decaying:conv_like}
$\widetilde{U}_n \to \widetilde{U}_\infty$ uniformly on $ [ H_-(\delta ), H_+  ] \times K$ under $\zproba_{( H_0, \sigma_0 )}^{n}$.
\item \label{propo:decaying:conv_nabla_like} 
$\sqrt{n} \nabla \widetilde{U}_n (H_0,\sigma_0) \to \gaussian{0}{J}
$ in distribution under $\zproba_{( H_0, \sigma_0 )}^{n}$, where $J = \nabla^2 \widetilde{U}_\infty (H_0,\sigma_0)$.
\item \label{propo:decaying:conv_nabla2_like}
$\nabla^2\widetilde{U}_n \to \nabla^2\widetilde{U}_\infty$ uniformly on $ [ H_-(\delta ), H_+  ] \times K$ under $\zproba_{( H_0, \sigma_0 )}^{n}$.
\end{enumerate}
\end{proposition}

The proofs of the first two points are gathered in section \ref{subsec:decaying:proof_score}, while the third one, being very similar to the first one, is omitted. 

Notice that $
( \widehat{H}_n, n^{\gamma ( H_0 )} \widehat{\nu}_n )  = 
\argmin \widetilde U_n
$ on $ \Theta^{n} $ so the properties of $( \widehat{H}_n, \widehat{\nu}_n )  $ and therefore those of $( \widehat{H}_n, \widehat{\sigma}_n )  $ can be deduced from $( \widehat{H}_n, n^{\gamma ( H_0 )} \widehat{\nu}_n )$. We write $\widetilde{\sigma}_n = n^{\gamma ( H_0 )} \widehat{\nu}_n$ for conciseness. Unfortunately, we cannot study directly $\widetilde{\sigma}_n$ because the minimisation on the whole set $\widetilde \Theta^{n}$ involves some technical difficulties:
\begin{tight_itemize}
\item The set $\widetilde \Theta^{n}$ is compact for each $n$, however, it depends on $n$ and it is not bounded uniformly in $n$: $\cup_n \widetilde \Theta^{n} =  [ H_-, H_+  ] \times (0, \infty)$. We correct this by restricting the values for $\widetilde \nu$ to a compact uniform in $n$ and we will deal separately with the remaining part of the set $\widetilde \Theta^{n}$.
\item The proof relies on convergence of some quadratic forms, such as those studied by Fox and Taqqu (for instance see \cite{fox1986large} and \cite{fox1987central}). These results depend crucially on the control we have at $0$ on the spectral density of the process considered. Therefore, we have to split the space of eligible $H$ into two subcategories: those which can be handled with these tools and those which cannot. The frontier lays at $\alpha(H) = \alpha(H_0) - 1$ and justifies the definition of $H_-(\delta)$ in Proposition \ref{propo:decaying:likelihood}.
\end{tight_itemize}

In view of these remarks, let $L > \max ( \sigma_0, \sigma_0^{-1} )$ and $\delta > 0$. We restrict the space $\widetilde{\Theta}^{n}$ to 
\begin{align}
\label{eq:decaying:boundedspace}
\widetilde{\Theta}^{n}( \delta, L ) 
=
 [ H_-(\delta ), H_+  ] \times 
(   [ \sigma_-n^{\gamma(H_0)-\gamma(H_+)}, \sigma_+n^{\gamma(H_0)-\gamma(H_-)}  ] \cap  [ L^{-1},  L ]
).
\end{align}
We also introduce a new couple $\widehat{\theta}^{(1)}_n (\delta, L) = ( \widehat{H}^{(1)}_n(\delta, L),\widehat{\sigma}^{(1)}_n(\delta, L) )$ by
\begin{align}
\label{eq:decaying:boundedest}
\widehat{\theta}^{(1)}_n (\delta, L) = ( \widehat{H}^{(1)}_n(\delta, L),\widehat{\sigma}^{(1)}_n(\delta, L) )
=
\argmin_{( H, \widetilde{\nu} ) \in \widetilde{\Theta}^{n}( \delta, L ) }
\widetilde{U}_n ( H, \widetilde{\nu} )
\end{align}
where the minimum is taken on $\widetilde{\Theta}^{n}( \delta, L ) $. In the following we will write $
\widehat{\theta}^{(1)}_n$ for $\widehat{\theta}^{(1)}_n(\delta, L)$ when the context is clear. The proof of Theorem \ref{thm:decaying} is now reduced to the three following steps:
\begin{tight_itemize}
\item \textit{\textbf{Step 1.}} We show that for any $\delta > 0$ and for $L$ large enough, $\widehat{\theta}^{(1)}_n (\delta, L) \to ( H_0, \sigma_0 )$ in distribution under $\zproba_{H_0, \sigma_0}^{n}$, using mostly Proposition \ref{propo:decaying:conv_like}.
\item \textit{\textbf{Step 2.}} We show a central limit theorem for $(\widehat{H}_n, \widetilde{\sigma}_n )$ by technical computations.
\item \textit{\textbf{Step 3.}} We get back to $( \widehat{H}_n, \widehat{\sigma}_n )$ by elementary manipulation.
\end{tight_itemize}

\subsection{Completion of the proof of Theorem \ref{thm:decaying}}
\label{subsec:decaying:completion}

\paragraph{Step 1.}
First take $n$ is large enough so that $\widetilde{\Theta}^{n}( \delta, L ) 
=
 [ H_-(\delta ), H_+  ] \times 
 [ L^{-1},  L ]
$ does not depend on $n$. For $\varepsilon > 0$, we write $\widetilde{\Theta}_\varepsilon^{n}( \delta, L ) = \{( H, \widetilde{\nu} ) 
		\in 
		\widetilde{\Theta}^{n}( \delta, L ) : \norm{( H, \widetilde{\nu} ) - ( H_0, \sigma_0 )} \geq \varepsilon \}$. Then we have:
\begin{align*}
\ieproba{H_0, \sigma_0}{n}{
	(
			\widehat{H}^{(1)}_n(\delta, L), 
			\widehat{\sigma}^{(1)}_n(\delta, L)
	)
	\in
	\widetilde{\Theta}_\varepsilon^{n}( \delta, L )
}
\leq 
\zproba_{H_0, \sigma_0}^{n}
\Bigg(
	\inf_{
		( H, \widetilde{\nu} ) 
		\in 
		\widetilde{\Theta}_\varepsilon^{n}( \delta, L )
		}
	\widetilde{U}_n ( H, \widetilde{\nu} ) 
	\leq 
	\widetilde{U}_n (H_0, \sigma_0)
	\Bigg).
\end{align*}
Using Proposition \ref{propo:decaying:conv_like}, this converges toward:
\begin{align*}
\zproba_{H_0, \sigma_0}^{n}
\Bigg(
	\inf_{
		( H, \widetilde{\nu} ) 
		\in 
\widetilde{\Theta}_\varepsilon^{n}( \delta, L )
	} 
	\widetilde{U}_\infty ( H, \widetilde{\nu} ) 
	\leq 
	\widetilde{U}_\infty (H_0, \sigma_0)
	\Bigg)
\end{align*}
which is null as the inequality $\log x < x - 1$ for $x \neq 1$ implies :
\begin{align*}
\inf_{
		( H, \widetilde{\nu} ) 
		\in 
\widetilde{\Theta}_\varepsilon^{n}( \delta, L )	} 
	\widetilde{U}_\infty ( H, \widetilde{\nu} ) 
	<
	\widetilde{U}_\infty (H_0, \sigma_0).
\end{align*}

\paragraph{Step 2}

We start by showing that $( \widehat{H}_n ,
\widetilde{\sigma}_n ) \to ( H_0 ,
\sigma_0 ) $ in distribution  under $\zproba_{H_0, \sigma_0}^{n}$. Take $n$ big enough so that $\widetilde{\Theta}^{n}( \delta, L ) 
=
 [ H_-(\delta ), H_+  ] \times 
 [ L^{-1},  L ]
$. By Step $1$, it is enough to show that
\begin{align*}
\zproba_{H_0, \sigma_0}^{n}
\big(
	\big(
			\widehat{H}^{(1)}_n(\delta, L) ,
			\widehat{\sigma}^{(1)}_n(\delta, L)
	\big)
			\neq 
		\big(
			\widehat{H}_n
		, \widetilde{\sigma}_n
		\big)
\big) \to 0
\end{align*}
for some $\delta > 0$ and $L$ big enough. But this quantity
is necessarily bounded by
\begin{align*}
\zproba_{H_0, \sigma_0}^{n}
\Bigg(
	\inf_{
		( H, \widetilde{\nu} ) 
		\in 
		\widetilde{\Theta}^{n}
		\backslash
		\widetilde{\Theta}^{n}( \delta, L ) 
	} 
	\widetilde{U}_n ( H, \widetilde{\nu} ) 
\leq
	\inf_{
		( H, \widetilde{\nu} ) 
		\in 
		\widetilde{\Theta}^{n}( \delta, L ) 
	} 
	\widetilde{U}_n ( H, \widetilde{\nu} ) 
\Bigg).
\end{align*}
By definition,
$
\inf_{
		( H, \widetilde{\nu} ) 
		\in 
		\widetilde{\Theta}^{n}( \delta, L ) 
	} 
	\widetilde{U}_n ( H, \widetilde{\nu})
	=
	\widetilde{U}_n ( \widehat{\theta}^{(1)}_n (\delta, L)
)
$. Moreover, by the triangle inequality:
\begin{align*}
 \Big|
	\widetilde{U}_n \big( \widehat{\theta}^{(1)}_n \big)
	-
	\widetilde{U}_\infty \big( H_0, \sigma_0 \big)
 \Big|
\leq 
 \Big|
	\widetilde{U}_n \big( \widehat{\theta}^{(1)}_n \big)
	-
	\widetilde{U}_\infty \big( \widehat{\theta}^{(1)}_n \big)
\Big |
+
 \Big|
	\widetilde{U}_\infty \big( \widehat{\theta}^{(1)}_n \big)
	-
	\widetilde{U}_\infty \big( H_0, \sigma_0 \big)
\Big |.
\end{align*}
The first term converges to $0$ by Proposition \ref{propo:decaying:conv_like} and the second is negligible as well by Step 1 and continuity of $\widetilde{U}_\infty$. Thus for any $\varepsilon > 0$, we have
\begin{align*}
\zproba_{H_0, \sigma_0}^{n}
\Big(
	 \big|
		\widetilde{U}_n ( \widehat{\theta}^{(1)}_n )
		-
		\widetilde{U}_\infty ( H_0, \sigma_0 )
	 \big|
	> 
	\varepsilon
\Big)
\to 
0.
\end{align*}
Then we write $		\widetilde{\Theta}^{n}
		\backslash
		\widetilde{\Theta}^{n}( \delta, L ) 
$
as the union of four disjoint subspaces and we deal with each of them separately in the following lemma that guarantees the convergence $( \widehat{H}_n ,
\widetilde{\sigma}_n ) \to ( H_0 ,
\sigma_0 ) $:

\begin{lemma}
\label{lemma:decaying:technical}
Let $\varepsilon > 0$ be fixed For any set $A \subset [H_-, H_+] \times (0,\infty)$, we write 
\begin{align*}
\salS( A )
=
\zproba_{H_0, \sigma_0}^{n}
\bigg(
	\inf_{
		( H, \widetilde{\nu} ) 
		\in 
		A
	} 
	\widetilde{U}_n ( H, \widetilde{\nu} ) 
\leq
	\widetilde{U}_\infty ( H_0, \sigma_0 ) + \varepsilon
\bigg).
\end{align*}
\begin{enumerate}[label=(\roman*), ref=\concatenate{\ref{lemma:decaying:technical}}{{(\roman*)}}]
  \setlength{\itemsep}{0pt}
  \setlength{\parskip}{0pt}
\item \label{lemma:decaying:technical:1}
For $L>1$ and $\varepsilon > 0$ fixed, $\salS(  [ H_-; H_-(\delta)  ]
		\times 
		 [ L^{-1}, L  ] ) \to 0$ for all $\delta$ small enough.

\item \label{lemma:decaying:technical:2}
For $\varepsilon > 0$ fixed, $
\salS( 
		 [ H_-( \delta ), H_+  ] \times  [ \sigma_-n^{\gamma ( H_0 )-\gamma ( H_+)}, L^{-1}  ]
)\to
0
$ for all $L$ large enough.

\item \label{lemma:decaying:technical:3}
For $L>1$ and $\varepsilon > 0$ fixed $
\salS( 
		 [ H_-,  H_-( \delta )  ] \times  [ \sigma_-n^{\gamma ( H_0 )-\gamma ( H_+)}, L^{-1}  ]
)\to
0
$ for all $\delta$ small enough.

\item \label{lemma:decaying:technical:4}
For $\varepsilon > 0$ fixed, $
\salS( 
		 [ H_-, H_+  ]  \times  [ L, \sigma_+n^{\gamma ( H_0 )-\gamma ( H_-)}  ]
)\to
0
$ for all $L$ large enough.
\end{enumerate}
\end{lemma}

The proofs of these results are delayed until section \ref{subsec:decaying:technical}. We now turn to the the proof of the asymptotic behavior of $( \widehat{H}_n, \widetilde{\sigma}_n )$. By Taylor formula, we have:
\begin{align*}
\nabla \widetilde U_n (H, \widetilde \nu)
=
\widetilde U_n (H_0, \widetilde \nu_0)
+
\int_0^1 
\nabla^2 \widetilde U_n ( 
H_0 + u (H - H_0),
\widetilde \nu_0 + u (\widetilde \nu - \widetilde \nu_0)
)
\cdot
\begin{pmatrix}
H - H_0
\\
\widetilde \nu - \widetilde \nu_0
\end{pmatrix}
du
\end{align*}
whenever $( H, \widetilde \nu ) \in \widetilde{\Theta}^{n}$. We apply this to $( H, \widetilde \nu ) = ( \widehat{H}_n, n^{ \gamma( H_0 )} \widehat{\nu}_n ) =
\argmin \widetilde U_n$ on $\widetilde{\Theta}^{n}$. Thus $\nabla \widetilde U_n ( \widehat{H}_n, n^{ \gamma( H_0 )} \widehat{\nu}_n )= 0$ whenever $( \widehat{H}_n, n^{ \gamma( H_0 )} \widehat{\nu}_n )$ is in the interior of 
$\widetilde{\Theta}^{n}$. We already know that $( \widehat{H}_n, n^{ \gamma( H_0 )} \widehat{\nu}_n ) \to ( H_0, \sigma_0)$ which is in this interior so it is asymptotically realised.

Moreover, Proposition \ref{propo:decaying:conv_nabla2_like}, continuity of $\nabla^2 \widetilde U_\infty $ and $( \widehat{H}_n, n^{ \gamma( H_0 )} \widehat{\nu}_n ) \to ( H_0, \sigma_0)$ show that
$
\nabla^2 \widetilde U_n ( 
H_0 + u (\widehat{H}_n - H_0),
\widetilde \nu_0 + u (\widetilde \sigma_n - \sigma_0)
)
-
\nabla^2 \widetilde U_\infty ( 
H_0,\sigma_0
)
$ is bounded in $\zproba_{H_0, \sigma_0}^{n}$-probability uniformly for $0 \leq u \leq 1$. Recall $J = \nabla^2 \widetilde U_\infty ( 
H_0,\sigma_0
)$, so that by Proposition \ref{propo:decaying:conv_nabla_like}, we have:
\begin{align}
\label{eq:decaying:CLT_oracle}
\sqrt{n} 
\begin{pmatrix}
\widehat{H}_n - H_0
\\
 \widetilde \sigma_n - \sigma_0
\end{pmatrix}
\to 
J^{-1} \gaussian{0}{J} \sim \gaussian{0}{J^{-1}} \text{ under } \zproba_{H_0, \sigma_0}^{n}.
\end{align}

\paragraph{Step 3.}

We are ready to conclude the proof of Theorem \ref{thm:decaying}. By \eqref{eq:decaying:CLT_oracle}, the convergence $\gamma ( \widehat{H}_n ) -  \gamma( H_0 ) \to \gaussian{0}{\gamma' ( H_0 )^2 ( J^{-1} )_{1,1}}$ holds in distribution. Thus we have:
\begin{align*}
\frac{\sqrt{n}}{\log n} \big( \widehat{\sigma}_n - \sigma_0 \big)
&=
\frac{\sqrt{n}}{\log n} \big( \widetilde{\sigma}_n - \sigma_0 \big)
+
\frac{\sqrt{n}}{\log n} \big( \widehat{\sigma}_n -  \widetilde{\sigma}_n \big)
\\
&=
\frac{\sqrt{n}}{\log n} \big( \widetilde{\sigma}_n - \sigma_0 \big)
+
\frac{\sqrt{n}\widetilde{\sigma}_n }{\log n} \big( n^{\gamma ( \widehat{H}_n ) -  \gamma( H_0 )} - 1 \big) 
\\
&=
\frac{\sqrt{n}}{\log n} \big( \widetilde{\sigma}_n - \sigma_0 \big)
+
\frac{\sqrt{n}\widetilde{\sigma}_n }{\log n} \Big( \log(n) \big( \gamma ( \widehat{H}_n ) -  \gamma( H_0 ) \big)+ O_{\zproba_{H_0, \sigma_0}^{n}} \big( n^{-1}\log(n)^2\big) \Big)
\\
&=
\frac{\sqrt{n}}{\log n} \big( \widetilde{\sigma}_n - \sigma_0 \big)
+
\widetilde{\sigma}_n \frac{\gamma ( \widehat{H}_n ) -  \gamma( H_0 )}{\sqrt{n}} + O_{\zproba_{H_0, \sigma_0}^{n}} \big( n^{-1/2}\log(n) \big)
\end{align*}
since $\log^2(n) ( \gamma ( \widehat{H}_n ) -  \gamma( H_0 ) )^2 = O_{\zproba_{H_0, \sigma_0}^{n}} ( n^{-1}\log(n)^2 )$. Moreover, $n^{-1/2}\log(n) ( \widetilde{\sigma}_n - \sigma_0 )
\to 0$ and $\widetilde{\sigma}_n \to \sigma_0$ so we can conclude. Note also that we can compute explicitly the limit variance because by differentiation under the integral sign, we have
\begin{align*}
J:= \nabla^2 \widetilde{U}_\infty
( 
H_0, \sigma_0
)
=
\begin{pmatrix}
	\frac{1}{4\pi} \int_{-\pi}^\pi 
	(
	\frac{\partial_H f_{H_0} (\lambda) }{ f_{H_0}(\lambda) } )^2 d\lambda
	&
	\frac{1}{2\pi  \sigma_0}
	\int_{-\pi}^\pi 
	\frac{\partial_H f_{H_0} (\lambda) }{ f_{H_0}(\lambda) }d\lambda
\\
	\frac{1}{2\pi  \sigma_0}
	\int_{-\pi}^\pi 
	\frac{\partial_H f_{H_0} (\lambda) }{ f_{H_0}(\lambda) }d\lambda
	&
		\frac{2}{\sigma_0^2}
\end{pmatrix}.
\end{align*}

\subsection{Behaviour of the score function}
\label{subsec:decaying:proof_score}

\subsubsection{Convergence of the quadratic forms of \texorpdfstring{$\mathbf{\widetilde{Y}}^n$}{tilde Yn}}
\label{subsec:decaying:quad}

For any integrable function $\varphi$ defined on $[-\pi, \pi]$, we define a $n \times n$ matrix $\Sigma_n( \varphi )$ by
\begin{align*}
\Sigma_n( \varphi )_{k,l} = 
\frac{1}{2\pi} 
\int_{-\pi}^\pi \varphi(\lambda) e^{i\lambda( k-l )} d\lambda.
\end{align*}
Then we define the quadratic forms with respect to $\varphi$ associated with observation $\mathbf{\widetilde{Y}}^n$ by
\begin{align*}
Q_n(\varphi ,\mathbf{\widetilde{Y}}^n) = \frac{1}{n} \transp{\mathbf{\widetilde{Y}}^n}\Sigma_n( \varphi ) \mathbf{\widetilde{Y}}^n.
\end{align*}
For any function $\varphi$, we also define
\begin{align}
\label{eq:decaying:Q_infty}
Q_\infty ( \varphi ) = \frac{\sigma_0^2}{2\pi}
\int_{-\pi}^\pi
f_{H_0}(\lambda) \varphi(\lambda)
d\lambda.
\end{align}

We study here the asymptotic behavior of sequences of the form $Q_n ( \varphi_{H,\eta}^{n}, \mathbf{\widetilde{Y}}^n )$ and we show two results. First,  under some regularity assumptions on $\big(\varphi_{H,\eta}^{n}\big)_n$, the quadratic forms converge toward a deterministic limit $Q_\infty ( \varphi_{H,\eta} )$. This is inspired by Fukasawa's work (see annex F in \cite{fukasawa2019volatility}) who works in a slightly different model. Then we study the deviation of $Q_n ( \varphi, \mathbf{\widetilde{Y}}^n )$ around its expectation. This generalizes Theorem 2 of \cite{fox1987central}, who proves this result in the absence of noise.

Recall the closed form for the cumulants of a quadratic form as computed in lemma 2 of \cite{magnus1986exact}.

\begin{lemma}
\label{lemma:moments_quad_form}
Let $\xi$ be a centered $n$-dimensional Gaussian vector with covariance matrix $\Gamma$ and let $\Lambda$ be a $n\times n$ matrix. Then for any $k \geq 0$, the $k$-th cumulant of $\transp{\xi}\Lambda\xi$ is given by:
\begin{align*}
\kappa_k ( \transp{\xi}\Lambda\xi ) = 2^{k-1} (k-1)! \Tr ( \Lambda \Gamma )^k.
\end{align*}
\end{lemma}

We also need an extension of some results initially proven by Fox and Taqqu in \cite{fox1986large} and \cite{fox1987central}) so that we can compute the limit of the trace of Toeplitz matrices.

\begin{proposition}
\label{propo:toeplitz_FT_asymptotic}
Let $\alpha_1 < 1$, $\alpha_2 < 1$ and $p\in\setN$. We consider two sequences of even function $\varphi^n_1, \varphi^n_2: [-\pi, \pi] \to [-\infty; \infty]$ and two functions $\varphi_1^\infty, \varphi_2^\infty$ such that the set of discontinuities of $\varphi_1^\infty$ and $\varphi_2^\infty$ has Lebesgue measure $0$,
$|\lambda|^{\alpha_j}  | \varphi_j^n(\lambda)  | < \infty
$ uniformly for $n\in\setN$ and $\lambda \neq 0$ and $\supess  | \varphi_j^n - \varphi_j^\infty | \to 0
$ for $j=1,2$.
Then we have:
\begin{enumerate}[label=(\roman*)]
  \setlength{\itemsep}{0pt}
  \setlength{\parskip}{0pt}
\item If $p ( \alpha_1 + \alpha_2 ) < 1$, then
$n^{-1} \Tr ( {( \Sigma_n ( \varphi^n_1 ) \Sigma_n ( \varphi^n_2 ) )^p } ) 
\to
 (2\pi)^{-1}
\int_{-\pi}^{\pi} 
( 
\varphi_1^\infty( \lambda )
\varphi_2^\infty( \lambda )
)^p
d\lambda.
$
\item If $p ( \alpha_1 + \alpha_2 ) \geq 1$, then 
$n^{-p ( \alpha_1 + \alpha_2 ) - \delta} \Tr ( {( \Sigma_n ( \varphi^n_1 ) \Sigma_n ( \varphi^n_2 ) )^p } ) \to 0
$ for any $\delta > 0$.
\end{enumerate}
\end{proposition}

This result is proven in \cite{fukasawa2019volatility}, see  theorem E.3. More precisely, we often use a slight generalisation of this result obtained by considering sums and differences of different sequences $\varphi_{j}^\infty( \lambda )$. Under suitable assumptions on the sequences $\varphi^n_{2j}$ and on the functions $\varphi_{2j}$ identical to that of Proposition \ref{propo:toeplitz_FT_asymptotic}, we have:
$n^{-1} \Tr 
( 
\prod_{j=1}^p {( \Sigma_n ( \varphi^n_{2j} ) \Sigma_n ( \varphi^n_{2j+1} ) ) } ) \to (2\pi)^{-1}
\int_{-\pi}^{\pi} 
( 
\prod_{j=1}^p
\varphi_{2j}( \lambda )
\varphi_{2j+1}( \lambda )
)
d\lambda$.

\begin{proposition}
\label{propo:decaying:conv_quad}
\begin{enumerate}[label=(\Roman*), ref=\concatenate{\ref{propo:decaying:conv_quad}}{{(\Roman*)}}]
  \setlength{\itemsep}{0pt}
  \setlength{\parskip}{0pt}
\item \label{propo:decaying:conv_quad:simple}
Let $( \varphi^n)_n$ a sequence of function and $\varphi^\infty$ be a function continuous almost everywhere. Suppose that
$|\lambda|^{\beta}  | \varphi^n(\lambda)  | < \infty
$ is bounded uniformly for $\lambda \neq 0$ and $n\geq 1$ for some $\beta < 1 - \alpha(H_0)$ and that $
\supess  | \varphi^n - \varphi^\infty  | \to 0
$. Then $Q_n ( \mathbf{\widetilde{Y}}^n, \varphi^n )
\to
Q_\infty (
\varphi
)$ in $L^2( \zproba_{H_0, \sigma_0}^{n} )$.

\item \label{propo:decaying:conv_quad:uniform}
Let $K$ be a compact subset of $ [ H_-, H_+  ]  \times \mathbb{R}^d$, $d\geq 1$. Suppose in addition that for each $( H, \eta ) \in K$ we are given
a sequence of functions $( \varphi^n_{H,\eta})_n$ and a function $\varphi_{H,\eta}^\infty$ which is continuous almost everywhere such that there exist $H \mapsto \beta(H)$ and $H \mapsto \gamma(H)$
continuous independent of $\eta$ such that
$|\lambda|^{\beta(H)}  | k^n_{H,\eta}(\lambda)  |$, $| \lambda|^{\gamma(H)}  | \partial_H \varphi^n_{H,\eta} (\lambda)  |$ and $| \lambda|^{\gamma(H)}  | \partial_\eta \varphi^n_{H,\eta} (\lambda)  |$ are bounded, uniformly for $n\in\setN$, $\lambda \neq 0$ and uniformly on $K$. Suppose in addition that $
 \supess  | \varphi^n_{H,\eta} - \varphi^\infty_{H,\eta}| \to 0$. Finally, suppose that for any $( H, \eta )$ in $K$, we have: $\beta(H) < 1 - \alpha(H_0) - \delta$ and $\gamma(H) < 1 - \alpha(H_0) - \delta$ for some fixed $\delta > 0$. Then $
 |
Q_n ( \mathbf{\widetilde{Y}}^n, \varphi^n_{H,\eta} )
-
Q_\infty (
\varphi_{H,\eta}^\infty
)
 |
\to 0 $  uniformly on $K$ under $\zproba_{H_0, \sigma_0}^{n}$.
\end{enumerate}
\end{proposition}

\begin{proof}[Proof of \ref{propo:decaying:conv_quad:simple}]
The bias-variance decomposition yields to 
\begin{align*}
\EX{ | 
Q_n ( \mathbf{\widetilde{Y}}^n, \varphi^n )
-
Q_\infty (
\varphi^\infty )  |^2}
&=
\Var{Q_n ( \mathbf{\widetilde{Y}}^n, \varphi^n )}
+(
\EX{
Q_n ( \mathbf{\widetilde{Y}}^n, \varphi^n )
}
- 
Q_\infty (
\varphi^\infty ) 
)^2.
\end{align*}
Recall now that $\mathbf{\widetilde{Y}}^n$ is Gaussian with spectral density $\widetilde{g}_{H_0, \sigma_0}^n$. We compute the expectation and the variance of $Q_n ( \mathbf{\widetilde{Y}}^n, \varphi^n )$ using Lemma \ref{lemma:moments_quad_form}. We conclude using Proposition \ref{propo:toeplitz_FT_asymptotic}, since it yields
\begin{align*}
\zEX \big[
Q_n ( \mathbf{\widetilde{Y}}^n, \varphi^n )
\big]
&=
n^{-1} \Tr \big( \Sigma_n( \varphi^n ) \Sigma_n( \widetilde{g}_{H_0, \sigma_0}^n) \big)
\to
Q_\infty ( \varphi^\infty ), \text{ and}\\
\Var{Q_n ( \mathbf{\widetilde{Y}}^n, \varphi^n )}
&=
2n^{-2} \Tr \big( ( \Sigma_n( \varphi^n ) \Sigma_n( \widetilde{g}_{H_0, \sigma_0}^n ) )^2 \big) \to 0.
\end{align*}
\end{proof}

\begin{proof}[Proof of \ref{propo:decaying:conv_quad:uniform}]
By compactness of $K$, we find for any given $r>0$ a finite covering of $K$ by sets $( U^{r}_1, \dots, U^{r}_{N(r)} )$ such that for all $i$, $U_i^{r}$ is included in the ball of radius $r$ and center $( H_i^{r}, \eta_i^{r} ) \in U_i^{r}$. Then we have
\begin{align*}
\sup_{( H, \eta ) \in K}
 |
Q_n ( \mathbf{\widetilde{Y}}^n, \varphi^n_{H,\eta} )
-
Q_\infty (
\varphi_{H,\eta}
)
 |
&\leq 
\sup_{1 \leq i \leq N(r)}
\sup_{( H, \eta ) \in U_i^{r}}
 |
Q_n ( \mathbf{\widetilde{Y}}^n, \varphi^n_{H,\eta} )
-
Q_n ( \mathbf{\widetilde{Y}}^n, \varphi^n_{H_i,\eta_i} )
 |
\\
&
\;\;\;\;
+
\sup_{1 \leq i \leq N(r)}
\sup_{( H, \eta ) \in U_i^{r}}
 |
Q_\infty ( \varphi_{H,\eta}^\infty )
-
Q_\infty ( \varphi_{H_i,\eta_i}^\infty )
 |
\\
&
\;\;\;\;
+
\sup_{1 \leq i \leq N(r)}
 |
Q_n ( \mathbf{\widetilde{Y}}^n, \varphi^n_{H_i,\eta_i} )
-
Q_\infty (
\varphi_{H_i,\eta_i}^\infty
)
 |.
\end{align*}
For any fixed $r$, we apply Proposition \ref{propo:decaying:conv_quad:simple} and we get
\begin{align*}
\sup_{1 \leq i \leq N(r)}
 |
Q_n ( \mathbf{\widetilde{Y}}^n, \varphi^n_{H_i,\eta_i} )
-
Q_\infty (
\varphi^\infty_{H_i,\eta_i}
)
 |
\to 
0 \text{ in }  L^2 ( \zproba_{H_0, \sigma_0}^{n} ).
\end{align*}
Moreover, the supremum of $
 |
Q_\infty ( \varphi^\infty_{H,\eta} )
-
Q_\infty ( \varphi^\infty_{H_i,\eta_i} )
 |$ is deterministic and $( H,\eta) \mapsto Q_\infty ( \varphi^\infty_{H,\eta} )$ is continuous because of the conditions on $\varphi^\infty_{H,\eta}$. Thus it is uniformly continuous on the compact $K$ and
\begin{align*}
\lim\limits_{r \to 0}
\sup_{1 \leq i \leq N(r)}
\sup_{( H, \eta ) \in U_i^{r}}
 |
Q_\infty ( \varphi^\infty_{H,\eta} )
-
Q_\infty ( \varphi^\infty_{H_i,\eta_i} )
 |
= 0.
\end{align*}
We conclude the proof by showing that for any $\varepsilon > 0$, there exist $r_0$ such that for $r \leq r_0$
\begin{align*}
\zproba\bigg(
\sup_{1 \leq i \leq N(r)}
\sup_{( H, \eta ) \in U_i^{r}}
 \Big|
Q_n \big( \mathbf{\widetilde{Y}}^n, \varphi^n_{H,\eta} \big)
-
Q_n \big( \mathbf{\widetilde{Y}}^n, \varphi^n_{H_i,\eta_i} \big)
 \Big|
> \varepsilon
\bigg)
\to 0.
\end{align*}

Let $1 \leq i \leq N(r)$ and $( H, \eta ) \in U_i^{r}$. By the mean value theorem, for any $\lambda$, there exist $0 \leq u=u(H,\eta, \lambda) \leq 1$ such that if $( H^*, \eta^* ) = ( H_i + u(H-H_i), \eta_i + u(\eta-\eta_i))$, we have
\begin{align*}
 \big|
\varphi^n_{H,\eta}(\lambda) - \varphi^n_{H_i,\eta_i} (\lambda)
 \big|
&=
 \big|
(H-H_0)
\partial_H \varphi^n_{H^*, \eta^* }(\lambda)
+
(\eta - \eta_0)
\partial_\eta \varphi^n_{H^*, \eta^* }
(\lambda)
 \big|
\\&
\leq 
r \sqrt{ \partial_H \varphi^n_{H^*, \eta^* }(\lambda)^2 + \partial_\eta \varphi^n_{H^*, \eta^* }
(\lambda)^2}
\\
&\leq c r |\lambda|^{\alpha(H_0) - 1 + \delta}
\end{align*}
where $c$ is a constant, independent of $r$, $\lambda$, $\eta$ and $H$. Thus we obtain
\begin{align*}
 \big|
Q_n \big( \mathbf{\widetilde{Y}}^n, \varphi^n_{H,\eta} \big)
-
Q_n \big( \mathbf{\widetilde{Y}}^n, \varphi^n_{H_i,\eta_i} \big)
 \big|
&=
\frac{1}{2\pi}
 \bigg|
\int_{-\pi}^{\pi}
( 
\varphi^n_{H,\eta}(\lambda)
- 
\varphi^n_{H_i,\eta_i}(\lambda)
)
I_n(\mathbf{\widetilde{Y}}^n, \lambda )
d\lambda
 \bigg|
\\
&\leq
\frac{cr}{2\pi}
\int_{-\pi}^{\pi}
|\lambda|^{\alpha(H_0) - 1 + \delta}
I_n(\mathbf{\widetilde{Y}}^n, \lambda )
d\lambda.
\end{align*}

We remark that the right hand side is proportional to $r
Q_n(
\mathbf{\widetilde{Y}}^n
, 
 | \; \cdot \;  |^{\alpha(H_0) - 1 + \delta} 
)$ and does not depend on $i$, $H$ or $\eta$. We apply again Proposition \ref{propo:decaying:conv_quad:simple} to show that it converges toward $Q_\infty (  | \; \cdot \;  |^{\alpha(H_0) - 1 + \delta}  )$ in probability, which is deterministic. We conclude by taking $r$ small enough.
\end{proof}

\begin{proposition}
\label{propo:decaying:clt_quadratic}
Let $\varphi^n$ a sequence of function and $\varphi^\infty$ be a function which is continuous a.e. such that
$
|\lambda|^{\beta}  | \varphi^n(\lambda)  | 
$ is bounded uniformly on $\lambda \neq 0$ and $n\geq 1$ and such that $
\supess_{\lambda}  | \varphi^n(\lambda) - \varphi^\infty(\lambda)  | \to 0
$
for some $\beta < \frac{1}{2} - \alpha ( H_0 )$. Then under $\zproba_{H_0, \sigma_0}^{n}$, we have
\begin{align*}
\sqrt{n} ( 
Q_n ( \mathbf{\widetilde{Y}}^n, \varphi^n )
- 
\EX{Q_n ( \mathbf{\widetilde{Y}}^n, \varphi^n )}
)
\to
\mathcal{N} \bigg( 0, \frac{\sigma_0^4}{\pi} \int_{-\pi}^\pi 
(
f_{H_0}(\lambda) \varphi^\infty(\lambda) )^2 d\lambda \bigg).
\end{align*}
\end{proposition}

\begin{proof}
By Lemma \ref{lemma:moments_quad_form}, the cumulants of $Q_n ( \mathbf{\widetilde{Y}}^n, \varphi^n )$ are 
$2^{k-1} (k-1)! \Tr \big[ ( \Sigma_n( \varphi^n ) \Sigma_n( \widetilde{g}_{H_0,  \sigma_0}^n) )^k \big]
$ so the cumulants of $\sqrt{n} ( 
Q_n ( \mathbf{\widetilde{Y}}^n, \varphi^n )
- 
\EX{Q_n ( \mathbf{\widetilde{Y}}^n, \varphi^n )} )
$ are given by
\begin{align*}
c_n(k) = 
\begin{cases}
0 &\text{ if } k=1,
\\
2^{k-1} (k-1)!n^{-\frac{k}{2}}  \Tr \big [( \Sigma_n( \varphi^n ) \Sigma_n( \widetilde{g}_{H_0,  \sigma_0}^n) )^k\big ]
&\text{ if } k\geq 2.
\end{cases}
\end{align*}
By Proposition \ref{propo:toeplitz_FT_asymptotic}, these cumulants converge towards
\begin{align*}
\lim\limits_{n\to\infty} c_n(k) =
\bigg(
\frac{\sigma_0^4}{\pi} \int_{-\pi}^\pi 
(
f_{H_0}(\lambda) \varphi^\infty(\lambda) )^2 d\lambda
\bigg)
\delta_{k,2}
\end{align*}
where $\delta_{k,2} = 1$ if $k=2$ and $0$ otherwise. We recognize the cumulants of a Gaussian variable. By Markov's inequality, the sequence is tight. We can conclude since these moments uniquely determine this distribution (see \cite{billingsley2008probability} for more details on the methods of moments).
\end{proof}

\subsubsection*{Proof of Proposition \ref{propo:decaying:conv_like}}

By definition, we have
\begin{align*}
 \big|
\widetilde{U}_n 
( 
H, \widetilde \nu
) -
\widetilde{U}_\infty
( 
H, \widetilde \nu
)
 \big|
&\leq
\frac{1}{4\pi}
 \bigg|
\int_{-\pi}^\pi
\log 
\Big(
\widetilde{g}^{n}_{H,\widetilde \nu}(\lambda)
\Big)
-
\log 
( 
\widetilde \nu^2 f_H(\lambda)
)
d\lambda
 \bigg|
\\&\;\;\;\;+
\frac{1}{4\pi}
 \bigg|
\int_{-\pi}^\pi
\frac{
I_n\big( \lambda, \mathbf{\widetilde{Y}}^n \big) 
}{
\widetilde{g}^{n}_{H,\widetilde \nu}(\lambda)
}
-
\frac{
\sigma_0^2 f_{H_0}(\lambda)
}{
\widetilde \nu^2 f_H(\lambda)
}
d\lambda
 \bigg|.
\end{align*}

Since $\tau_n^2 n^{2\gamma ( H_+ )}$ is bounded, we can use Assumption \ref{assumption:decaying:psd:bounds}  and boundedness of $l$ to see that
\begin{align*}
|
\log 
( 
\widetilde{g}^{n}_{H,\widetilde \nu}(\lambda)
)
-
\log 
( 
\widetilde \nu^2 f_H(\lambda)
)
 |
&\leq
\frac{\tau_n^2 n^{2\gamma ( H_0 )} l(\lambda)}{\widetilde \nu^2 f_H(\lambda)}
\leq \frac{l(\lambda)\pi^{\alpha(H_+) - \alpha(H_-)} |\lambda|^{\alpha(H_-)}}{c_1 ( \inf K)^2} 
\tau_n^2 n^{2\gamma ( H_0 )} 
\end{align*}
so that the first term converges to $0$ uniformly. For the second term, we apply Proposition \ref{propo:decaying:conv_quad:uniform} to $\varphi^n_{H,\widetilde \nu} = \big(\widetilde{g}^{n}_{H,\widetilde \nu}\big)^{-1}$ and $\varphi^\infty_{H,\widetilde \nu} = \big(\widetilde \nu^2f_H\big)^{-1}$ (we can check these functions satisfy all the regularity Assumptions needed in Proposition \ref{propo:decaying:conv_quad:uniform} by Assumption \ref{assumption:decaying:psd}).

\subsubsection*{Proof of Proposition \ref{propo:decaying:conv_nabla_like}}

Let $v \in \mathbb{R}^2$. We show here that $\transp{v} \nabla \widetilde{U}_n (H_0,\sigma_0) \to \gaussian{0}{\transp{v}Jv}$ under $\zproba_{H_0, \sigma_0}^{n}$. We start by considering the case of partial derivatives alone, i.e. $v=(0,1)$ or $v=(1, 0)$. In this proof, let $\partial$ stand either for $\partial_H$ or for $\partial_\sigma$.
Recall $\widetilde{U}_n$ is defined in \eqref{eq:decaying:Utilde} so that by differentiation under the integral sign, we have 
\begin{align*}
\partial \widetilde{U}_n ( H,\widetilde \nu )
&=
\frac{1}{4\pi}
\int_{-\pi}^\pi
\frac{
\partial \widetilde{g}^{n}_{H,\widetilde \nu}(\lambda)
}{
\widetilde{g}^{n}_{H,\widetilde \nu}(\lambda)
}
-
\frac{
I_n\big( \lambda, \mathbf{\widetilde{Y}}^n \big) \partial \widetilde{g}^{n}_{H,\widetilde \nu}(\lambda)
}{
\widetilde{g}^{n}_{H,\widetilde \nu}(\lambda)^2
}
d\lambda
\end{align*}
for any $( H,\widetilde \nu)$. We write $h^{n}_\partial = 
\partial \widetilde{g}^{n}_{H_0,\sigma_0}
\widetilde{g}^{n}_{H_0,\sigma_0}(\lambda)^{-2}
$ so that $\partial \widetilde{U}_n ( H,\widetilde \nu )$ is rewritten
\begin{align*}
\frac{1}{2}
\bigg(
\frac{1}{2\pi}
\int_{-\pi}^\pi
\frac{
\partial \widetilde{g}^{n}_{H_0,\sigma_0}(\lambda)
}{
\widetilde{g}^{n}_{H_0,\sigma_0}(\lambda)
}
d\lambda
-
\zEX_{H_0,\sigma_0}^{n} \Big[Q_n \big( \mathbf{\widetilde{Y}}^n, h^{n}_\partial\big)\Big]
\bigg)
+
\frac{1}{2}
\big(
\zEX_{H_0,\sigma_0}^{n} \Big[Q_n \big( \mathbf{\widetilde{Y}}^n, h^{n}_\partial\big)\Big]
-
Q_n( \mathbf{\widetilde{Y}}^n, h^{n}_\partial)
\big).
\end{align*}

First we show that 
\begin{align}
\label{eq:lemma:decaying:nabla_like_lastterm}
\frac{1}{2\pi}
\int_{-\pi}^\pi
\frac{
\partial \widetilde{g}^{n}_{H_0,\sigma_0}(\lambda)
}{
\widetilde{g}^{n}_{H_0,\sigma_0}(\lambda)
}
d\lambda
-
\zEX_{H_0,\sigma_0}^{n} \Big[Q_n \big( \mathbf{\widetilde{Y}}^n, h^{n}_\partial\big)\Big] = O_{\zproba_{H_0, \sigma_0}^{n}} ( n^{-1/2} ).
\end{align}

Let $f*g(x) = \frac{1}{2\pi}\int_{-\pi}^{\pi} f(y) g(x-y) dy$ with $f$ and $g$ some $2\pi$ periodic functions so that $\widehat{f*g}(k) = \widehat{f}(k) \widehat{g}(k)$. Since $\widetilde{g}^{n}_{H_0,\sigma_0}$ and  $h^{n}_\partial$ are integrable and at least one is bounded,  $\widetilde{g}^{n}_{H_0,\sigma_0} * h^{n}_\partial$ is well defined and continuous. To control their Fourier coefficients, we use lemmas E.1 and E.2 of \cite{fukasawa2019volatility} stated below for conciseness.

\begin{lemma}[lemmas E.1 and E.2 in \cite{fukasawa2019volatility}]
\label{lemma:fukasawa:E12}
Let $\beta \in (-1, 0) \cup (0,1)$ and let $n \in \setN$. Suppose a sequence of $2\pi$-periodic functions $k^n: \mathbb{R} \to [-\infty, \infty]$ satisfies the following conditions:
\begin{tight_enumerate}
\item If $\beta \geq 0$, $k^n$ is continuously differentiable on $[-\pi, \pi] \backslash \{ 0 \}$ for each $n$ and 
$ | \lambda  |^\beta 
 | k^n (\lambda)  |
$
and 
$ | \lambda  |^{\beta+1} 
 | \partial_\lambda k^n (\lambda)  |
$
are bounded uniformly for $n \in \setN$ and $\lambda \in [-\pi, \pi] \backslash \{ 0 \}$. 
\item If $\beta < 0$, $k^n$ is integrable and twice continuously differentiable on $[-\pi, \pi] \backslash \{ 0 \}$ for each $n$ and 
$ | \lambda  |^{\beta +1}
 | \partial_\lambda k^n (\lambda)  |
$
and 
$ | \lambda  |^{\beta+2} 
 | \partial^2_{\lambda\lambda} k^n (\lambda)  |
$
are bounded uniformly for $n \in \setN$ and $\lambda \in [-\pi, \pi] \backslash \{ 0 \}$.  
\end{tight_enumerate}

Then the sequence of Fourier coefficients $( \widehat{k^n}(\tau) )_\tau$ satisfies:
\begin{align*}
\sup_{n\in\setN}
 |
\widehat{k^n}(\tau)
 |
=
\begin{cases}
O ( 
 |
\tau
 |^{\beta - 1}
)
& \text{if } \beta \neq 0,
\\
O ( 
|\tau|^{-1}\log |\tau| )
& \text{if } \beta = 0.
\end{cases}
\end{align*}
\end{lemma}

%

We can apply this lemma by Assumption \ref{assumption:decaying:psd}, dealing with the different cases depending on the value of $\alpha ( H_0 )$. Thus for any $\delta > 0$,
$
(  | k  |^{2-\delta} \widehat{\widetilde{g}^{n}_{H_0,\sigma_0}}(k) \widehat{h^{n}_\partial}(k)  )_k
$ is bounded uniformly for $n \geq 1$. Still using Assumption \ref{assumption:decaying:psd}, we can show that $\widetilde{g}^{n}_{H_0,\sigma_0} * h^{n}_\partial $ satisfies the Dirichlet conditions. so that it equals its Fourier series at any continuity point. Thus
\begin{align*}
\frac{1}{2\pi}
\int_{-\pi}^\pi
\frac{
\partial \widetilde{g}^{n}_{H_0,\sigma_0}(\lambda)
}{
\widetilde{g}^{n}_{H_0,\sigma_0}(\lambda)
}
d\lambda
=
\widetilde{g}^{n}_{H_0,\sigma_0} * h^{n}_\partial (0)
=
\sum_{k\in\mathbb{Z}} \widehat{\widetilde{g}^{n}_{H_0,\sigma_0}}(k) \widehat{h^{n}_\partial}(k).
\end{align*}
On the other side, we easily get
\begin{align*}
\ieEX{H_0,\sigma_0}{n}{Q_n( \mathbf{\widetilde{Y}}^n, h^{n}_\partial)}
&=
\frac{1}{n} \sum_{k,l = 1}^{n} \widehat{\widetilde{g}^{n}_{H_0,\sigma_0}}(k-l) \widehat{h^{n}_\partial}(k-l)
=
\sum_{|k| < n} \Big( 1 - \frac{|k|}{n} \Big) \widehat{\widetilde{g}^{n}_{H_0,\sigma_0}}(k) \widehat{h^{n}_\partial}(k).
\end{align*}
Combining these identities, the left hand side of Equation \eqref{eq:lemma:decaying:nabla_like_lastterm} reduces to 
\begin{align*}
2 \sum_{k \geq n } \widehat{\widetilde{g}^{n}_{H_0,\sigma_0}}(k) \widehat{h^{n}_\partial}(k)
-\frac{2}{n} \sum_{k= 0}^{n-1} |k| \widehat{\widetilde{g}^{n}_{H_0,\sigma_0}}(k) \widehat{h^{n}_\partial}(k)+\frac{\widehat{\widetilde{g}^{n}_{H_0,\sigma_0}}(0) \widehat{h^{n}_\partial}(0)}{n}.
\end{align*}
It is bounded for $\delta > 0$ by
$C 
( 
	\sum_{k \geq n } k^{\delta-2}
	+
	n^{-1} 
	\sum_{k= 0}^{n-1} k^{\delta-1}
) \leq Cn^{\delta-1}
= o ( n^{-1/2} )
$, which proves \eqref{eq:lemma:decaying:nabla_like_lastterm}.

We come back now to $\transp{v} \nabla \widetilde{U}_n (H_0,\sigma_0)$. Indeed, by what precedes, we rewrite it as
\begin{align*}
\transp{v} \nabla \widetilde{U}_n (H_0,\sigma_0) 
=
\frac{1}{2} (
\ieEX{H_0,\sigma_0}{n}{Q_n( \mathbf{\widetilde{Y}}^n, h^n )}
-
Q_n( \mathbf{\widetilde{Y}}^n, h^n )
)
+
o( n^{-1/2} )
\end{align*}
where 
$h^n = v_1 h^n_{\partial_H} + v_2 h^n_{\partial_{\widetilde \nu}}$.
We define 
$h_\partial$ as the limit of 
$h^{n}_\partial$, i.e. we have
$h_{\partial_{\widetilde{\nu}}}(\lambda) =
2\sigma_0^{-3} f_{H_0}(\lambda)^{-1}$ 
if 
$ \partial = \partial_{\widetilde{\nu}}$ 
and 
$h_{\partial_H}(\lambda) = \sigma_0^{-2} f_{H_0}(\lambda)^{-2}\partial_H f_{H_0} (\lambda) $ 
if 
$\partial = \partial_H$. Similarly, we write $h = v_1 h_{\partial_H} + v_2 h_{\partial_{\widetilde \nu}}$. By Proposition \ref{propo:decaying:clt_quadratic} (that applies if we take $r < \frac{1}{2}$ in Assumption \ref{assumption:decaying:psd:bounds_derf}), we see that under $\zproba_{H_0, \sigma_0}^{n}$, we have
\begin{align*}
\sqrt{n}
(
\ieEX{H_0,\sigma_0}{n}{Q_n( \mathbf{\widetilde{Y}}^n, h^n)}
-
Q_n( \mathbf{\widetilde{Y}}^n, h^n)
)
\to
\gaussian{0}{I_\partial}
\text{ where }
I_\partial
&= \frac{\sigma_0^4}{\pi} \int_{-\pi}^\pi 
(
f_{H_0}(\lambda) h(\lambda) )^2 d\lambda.
\end{align*}
Explicit computations gives  $I_\partial = 4\transp{v}Jv$, which proves Proposition
 \ref{propo:decaying:conv_nabla_like}.

%
%
%

\subsection{Proof of technical Lemma \ref{lemma:decaying:technical}}
\label{subsec:decaying:technical}

\subsubsection*{Proof of \ref{lemma:decaying:technical:1}}

Recall from Assumption \ref{assumption:decaying:psd} that we have
$
c_1  | \lambda  |^{-\alpha(H)}
\leq 
f_H(\lambda)
\leq 
c_2  | \lambda  |^{-\alpha(H)}
$
uniformly for $( H, \lambda ) \in  [ H_-, H_+  ] \times [-\pi, \pi] \backslash \{ 0 \}$. Using $n^{\gamma ( H_0 )}\tau_n \to 0$, $l(\lambda) = O( \lambda^2 )$ and setting $L^{-1} \leq  \widetilde{\nu} \leq L$, we have
$\widetilde{c_1}  | \lambda  |^{-\alpha(H)}
\leq 
\widetilde{g}^{n}_{H,\widetilde \nu}(\lambda)
\leq 
\widetilde{c_2}  | \lambda  |^{-\alpha(H)}
$ for some constants $\widetilde{c_1}$ and $\widetilde{c_2}$, uniformly for $( H, \widetilde{\nu} , \lambda ) \in 		 [ H_-; H_+  ]
		\times 
		 [ L^{-1}, L  ]
 \times ( [-\pi, \pi] \backslash \{ 0 \} )$. Note that these constants are independent from $\delta$ but (can) depend on $L$. Let $( H, \widetilde{\nu} ) 
		\in 
		 [ H_-, H_-(\delta)  ]  \times 
		 [ L^{-1}, L  ]$. Then we have $\alpha(H) \leq \alpha(H_0) - 1 + \delta$ so for all $0 < |\lambda| < \pi$ we have 
$
 | \lambda  |^{\alpha ( H )} 
\geq  
 | \lambda  |^{\alpha ( H_0 ) - 1 + \delta}
 | \pi  |^{\alpha ( H_- ) - \alpha ( H_0 ) + 1 - \delta}
$.
In particular, we deduce that 
\begin{align*}
\widetilde{U}_n ( H, \widetilde{\nu} ) 
&\geq 
\frac{1}{4\pi}
\int_{-\pi}^\pi
\log \big(
\widetilde{c_1}  | \lambda  |^{-\alpha(H)}
\big)
+
\frac{
I_n\big( \lambda, \mathbf{\widetilde{Y}}^n \big) 
}{
\widetilde{c_2}  | \lambda  |^{-\alpha(H)}
}
d\lambda
\\
&\geq 
\frac{1}{4\pi}
\int_{-\pi}^\pi
\log (
\widetilde{c_1} 
)
-(2H-1)
\log (
 | \lambda  |
)
+
 | \pi  |^{\alpha ( H_- ) - \alpha ( H_0 ) + 1 - \delta}
\frac{
I_n\big( \lambda, \mathbf{\widetilde{Y}}^n \big) 
 | \lambda  |^{\alpha ( H_0 ) - 1 + \delta}
}{
\widetilde{c_2}
}
d\lambda
\\
&\geq
\frac{\log (
\widetilde{c_1} 
)}{2}
- (2H_+-1) ( \log \pi - 1 )
+ C
\int_{-\pi}^\pi
I_n\big( \lambda, \mathbf{\widetilde{Y}}^n \big) 
 | \lambda  |^{\alpha ( H_0 ) - 1 + \delta}
d\lambda
\end{align*}
where $C$ is a positive constant independent from $H$ and $\widetilde{\nu}$. Indeed, except the right-most integral, everything is deterministic, so the lemma is proved if we show that for all $r$, there exist $\delta_0 > 0$ such that for $\delta < \delta_0$ we have
\begin{align*}
\zproba_{H_0, \sigma_0}^{n}
\bigg(
	\frac{1}{2\pi}\int_{-\pi}^\pi
I_n\big( \lambda, \mathbf{\widetilde{Y}}^n \big) 
 | \lambda  |^{\alpha ( H_0 ) - 1 + \delta}
	d\lambda
\leq
	r
\bigg)
\to
0.
\end{align*}
Applying Proposition \eqref{propo:decaying:conv_quad:simple}, we have in $L^2(\zproba_{H_0, \sigma_0}^{n})$:
\begin{align*}
\frac{1}{2\pi}\int_{-\pi}^\pi
I_n\big( \lambda, \mathbf{\widetilde{Y}}^n \big) 
 | \lambda  |^{\alpha ( H_0 ) - 1 + \delta}
	d\lambda 
\to
\frac{\sigma_0^2}{2\pi}
\int_{-\pi}^\pi
f_{H_0}(\lambda)  | \lambda  |^{\alpha ( H_0 ) - 1 + \delta}
d\lambda.
\end{align*}
We conclude since Assumption \ref{assumption:decaying:psd} yields
\begin{align*}
\frac{\sigma_0^2}{2\pi}
\int_{-\pi}^\pi
f_{H_0}(\lambda)  | \lambda  |^{\alpha ( H_0 ) - 1 + \delta}
d\lambda &\geq 
\frac{c_1 \sigma_0^2}{2\pi}
\int_{-\pi}^\pi
 | \lambda  |^{-1 + \delta }
d\lambda
= 
\frac{c_1 \sigma_0^2}{\pi}
\frac{ | \pi  |^{\delta }}{\delta } \to \infty \text{ when } \delta \to 0.
\end{align*}

\subsubsection*{Proof of \ref{lemma:decaying:technical:2}}

Let $( H, \widetilde{\nu} ) 
		\in 
		 [ H_-( \delta ), H_+  ]  \times  [ \sigma_-n^{H_0-H_+}, L^{-1}  ]$. We have
\begin{align*}
\widetilde{U}_n ( H, \widetilde{\nu} ) 
&\geq 
\frac{1}{4\pi}
\int_{-\pi}^\pi
\log (
\widetilde \nu^2
\widetilde{f}_{H}^n(\lambda)
)
+
\frac{
I_n\big( \lambda, \mathbf{\widetilde{Y}}^n \big) 
}{
\widetilde \nu^2
f_{H}(\lambda) + \tau_n^2 n^{2\gamma ( H_0)}l(\lambda)
}
d\lambda
\\
&\geq 
\frac{1}{4\pi\widetilde \nu^2}
\int_{-\pi}^\pi
\frac{
I_n\big( \lambda, \mathbf{\widetilde{Y}}^n \big) 
}{
f_{H}(\lambda) + {\tau_n^2 n^{2\gamma ( H_0)}\widetilde \nu^{-2} l(\lambda)}
}
d\lambda
+ \log ( \widetilde \nu )
+
\frac{1}{4\pi}
\int_{-\pi}^\pi
\log (
f_{H}(\lambda)
)
d\lambda
\\
&\geq 
\frac{1}{4\pi\widetilde \nu^2}
\int_{-\pi}^\pi
\frac{
I_n\big( \lambda, \mathbf{\widetilde{Y}}^n \big) 
}{
f_{H}(\lambda) + \sigma_-^{-2} \tau_n^2 n^{2\gamma ( H_+)}l(\lambda)
}
d\lambda
+ \log ( \widetilde \nu )
+
C^*
\end{align*}
where $C^* = \min_{h\in [ H_-, H_+  ]} ( (4\pi)^{-1} \int \log ( f_h(\lambda)) d\lambda) > - \infty$ does not depend on $L$. Moreover, since $l(\lambda) = O( \lambda^2 )$ and $n^{2 \gamma ( H_+)} \tau_n^2 = O(1)$, we know from Assumption \ref{assumption:decaying:psd} that there exists $c_1$ independent from $L$ and $H$ such that $f_{H}(\lambda) + \sigma_-^{-2} \tau_n^2 n^{2 \gamma ( H_+)}l(\lambda) \leq c_1 |\lambda|^{-\alpha(H)}$ for any $0 < |\lambda| < \pi$. $\alpha$ is increasing and $|\lambda| \leq \pi$ so we get
\begin{align*}
\widetilde{U}_n ( H, \widetilde{\nu} )
&\geq
\frac{1}{4\pi c_1\widetilde \nu^2}
\int_{-\pi}^\pi
|\lambda|^{\alpha(H)} 
I_n\big( \lambda, \mathbf{\widetilde{Y}}^n \big) 
d\lambda
+ \log ( \widetilde \nu )
+
C^*
\\
&\geq
\frac{1}{4\pi c\widetilde \nu^2}
\int_{-\pi}^\pi
|\lambda|^{\alpha(H_+)} 
I_n\big( \lambda, \mathbf{\widetilde{Y}}^n \big) 
d\lambda
+ \log ( \widetilde \nu )
+
C^*
\end{align*}
for some constant $c>0$. This last quantity does not depend on $H$ and is minimal when we take
$\widetilde \nu^*
=
\widetilde \nu^*_n
=
\big(
(2\pi c)^{-1}
\int_{-\pi}^\pi
|\lambda|^{\alpha(H_+)} 
I_n\big( \lambda, \mathbf{\widetilde{Y}}^n \big) 
d\lambda
\big)^{1/2}
$. Thus $\salS( 
		 [ H_-( \delta ), H_+  ] \times  [ \sigma_-n^{\gamma ( H_0 )-\gamma ( H_+)}, L^{-1}  ]$ as defined in Lemma \ref{lemma:decaying:technical} is bounded by
\begin{align*}
&\zproba_{H_0, \sigma_0}^{n}
\Big(
\frac{L^2}{4\pi c}
\int_{-\pi}^\pi
|\lambda|^{\alpha(H_+)} 
I_n\big( \lambda, \mathbf{\widetilde{Y}}^n \big) 
d\lambda
- \log ( L )
+
C
\leq
	\widetilde{U}_\infty ( H_0, \sigma_0 ) + \varepsilon\;,\;\widetilde \nu^*_n > L^{-1} 
\Big)
\\&\;\;\;\;+
\ieproba{H_0, \sigma_0}{n}
{
	\log ( (\widetilde{\nu}^*_n)^2
)/2 +
	1/2 + C^*
\leq
	\widetilde{U}_\infty ( H_0, \sigma_0 ) + \varepsilon\;,\;\widetilde \nu^*_n \leq L^{-1} 
}.
\end{align*}

From now, it is enough to prove that there exists $r>0$ such that 
\begin{align*}
\zproba_{H_0, \sigma_0}^{n}
\bigg(
	\frac{1}{2\pi}
\int_{-\pi}^\pi
 | \lambda  |^{\alpha(H_+)}
I_n\big( \lambda, \mathbf{\widetilde{Y}}^n \big) 
d\lambda
 \leq r
\bigg)
\to 0.
\end{align*}
We conclude by Proposition \ref{propo:decaying:conv_quad:simple}, with $\varphi^n(\lambda) = \varphi(\lambda) =  | \lambda  |^{\alpha(H_+)}$ and $\beta = - \alpha(H_+) < 1 - \alpha(H_0)$.

\subsubsection*{Proof of \ref{lemma:decaying:technical:3}}

As for the proof of Lemma \ref{lemma:decaying:technical:1}, we start by recalling Assumption \ref{assumption:decaying:psd}:
\begin{align*}
\forall\, ( H, \lambda ) \in  [ H_-, H_+  ] \times [-\pi, \pi] \backslash \{ 0 \}, \;
c_1  | \lambda  |^{-\alpha(H)}
\leq 
f_H(\lambda)
\leq 
c_2  | \lambda  |^{-\alpha(H)}.
\end{align*}
Using both $n^{\gamma ( H_0)}\tau_n \to 0$ (since $( n^{\gamma ( H_+)}\tau_n )_n$ is bounded) and $l(\lambda) = O( \lambda^2 )$, we have
\begin{align*}
\forall\, ( H , \lambda ) \in  [ H_-, H_+  ] \times ( [-\pi, \pi] \backslash \{ 0 \} ), \;
\widetilde{c_1}  | \lambda  |^{-\alpha(H)}
\leq 
f_H(\lambda) + n^{2\gamma ( H_0)}\sigma_-^{-2}\tau_n^2 l(\lambda)
\leq 
\widetilde{c_2}  | \lambda  |^{-\alpha(H)}.
\end{align*}
Let now $( H, \widetilde{\nu} ) 
		\in 
		 [ H_-,  H_-( \delta )  ] \times  [ \sigma_-n^{H_0-H_+}, L^{-1}  ]$. We have
\begin{align*}
\widetilde{U}_n ( H, \widetilde{\nu} ) 
&\geq
\log ( \widetilde \nu )
+
\frac{1}{4\pi}
\int_{-\pi}^\pi
\log (
f_{H}(\lambda)
)
d\lambda
+
\frac{1}{4\pi\widetilde{\nu}^2}
\int_{-\pi}^\pi
\frac{
I_n\big( \lambda, \mathbf{\widetilde{Y}}^n \big) 
}{
f_H(\lambda) + n^{2\gamma ( H_0)}\widetilde{\nu}^{-2}\tau_n^2 l(\lambda)
}
d\lambda.
\end{align*}
Notice also $C^* = \min_{h\in [ H_-, H_+  ]} ( (4\pi)^{-1} \int \log ( f_h(\lambda)) d\lambda) > - \infty$ does not depend on $L$, $\epsilon$ or $\delta$. Remark that $ \min_{x>0} x^2 \log(x) = -(2e)^{-1}$ and $\sigma_- n^{\gamma ( H_0) - \gamma ( H_+)} \leq \widetilde{\nu} \leq L^{-1}$, so that we get
\begin{align*}
\widetilde{U}_n ( H, \widetilde{\nu} ) 
&\geq
\frac{-1}{2e\widetilde{\nu}^2}
+
C
+
\frac{1}{4\pi\widetilde{\nu}^2}
\int_{-\pi}^\pi
\frac{
I_n\big( \lambda, \mathbf{\widetilde{Y}}^n \big) 
}{
f_H(\lambda) + n^{2\gamma ( H_0)}\widetilde{\nu}^{-2}\tau_n^2 l(\lambda)
}
d\lambda
\\
&\geq 
\frac{1}{\widetilde{\nu}^2}
\bigg(
\frac{-1}{2e}
+
\frac{1}{4\pi\widetilde{c_1}}
\int_{-\pi}^\pi
 | \lambda  |^{\alpha(H)}
I_n\big( \lambda, \mathbf{\widetilde{Y}}^n \big) 
d\lambda
\bigg)
+
C^*.
\end{align*}

Moreover, notice that as $H \in  [ H_-,  H_-( \delta )  ]$, we have $\alpha(H) < \alpha(H_0) -1 + \delta$ so that we get: $|\lambda|^{\alpha(H)} \geq \pi^{\alpha(H) - \alpha(H_0) + 1 - \delta} |\lambda|^{\alpha(H_0) -1 + \delta} \geq |\lambda|^{\alpha(H_0) -1 + \delta}$. We obtain
\begin{align}
\label{eq:decaying:technical:3:eq1}
\widetilde{U}_n ( H, \widetilde{\nu} ) 
&\geq
\frac{1}{\widetilde{\nu}^2}
\bigg(
\frac{-1}{2e}
+
\frac{1}{4\pi\widetilde{c_1}}
\int_{-\pi}^\pi
|\lambda|^{\alpha(H_0) -1 + \delta}
I_n\big( \lambda, \mathbf{\widetilde{Y}}^n \big) 
d\lambda
\bigg)
+
C^*.
\end{align}
Applying Proposition \ref{propo:decaying:conv_quad:simple} to $\varphi^n(\lambda) = \varphi(\lambda) =  | \lambda  |^{\alpha(H_0) - 1 + \delta}$ with $\beta = 1 - \alpha(H_0) - \delta < 1 - \alpha ( H_0 )$,  we have the following convergence in $L^2$
\begin{align}
\label{eq:decaying:technical:3:eq2}
\frac{1}{2\pi}
\int_{-\pi}^\pi
|\lambda|^{\alpha(H_0) -1 + \delta}
I_n\big( \lambda, \mathbf{\widetilde{Y}}^n \big) 
d\lambda
\to
\frac{\sigma_0}{2\pi}
\int_{-\pi}^\pi
|\lambda|^{\alpha(H_0) -1 + \delta}
f_{H_0}(\lambda)
d\lambda.
\end{align}
As in Lemma \ref{lemma:decaying:technical:1}, we also have:
\begin{align}
\label{eq:decaying:technical:3:eq3}
\frac{\sigma_0}{2\pi}
\int_{-\pi}^\pi
|\lambda|^{\alpha(H_0) -1 + \delta}
f_{H_0}(\lambda)
d\lambda
\geq
\frac{c_1\sigma_0^2}{\pi}\frac{\pi^\delta}{\delta}
\to_{\delta\to 0} \infty.
\end{align}
We can easily conclude combining \eqref{eq:decaying:technical:3:eq1}, \eqref{eq:decaying:technical:3:eq2} and \eqref{eq:decaying:technical:3:eq3}.


\subsubsection*{Proof of \ref{lemma:decaying:technical:4}}

Let $( H, \widetilde{\nu} ) 
		\in 
		 [ H_-, H_+  ]   \times  [ L, \sigma_+n^{\gamma( H_0 ) - \gamma( H_+)} ]$ so that \begin{align*}
\widetilde{U}_n ( H, \widetilde{\nu} ) 
\geq 
(4\pi)^{-1} \log ( \widetilde{\nu} f_H(\lambda) ) d\lambda
\geq \log L + C^*
\end{align*} where $C^* = \min_h ( (4\pi)^{-1} \log ( f_h(\lambda)) d\lambda) > - \infty$ does not depend on $L$. As $\log L \to \infty$ as $L\to \infty$, we can conclude.

\section{Proof of Theorem \ref{thm:decaying:LAN}}
\label{sec:proof:LAN}

\subsection{Outline of the proof}

For conciseness we will write $ \theta = (H, \sigma)$ and $\theta_0 = (H_0, \sigma_0)$ throughout the proof. Let $u=(u_H, u_\sigma) \in \mathbb{R}^2$. We define $Z_n(u) = dP^n_{\theta_0 + \varphi_n u}/dP^n_{\theta_0}$ as in the definition \ref{def:decaying:LAN}. We also introduce the $-2$-log-likelihood $\salL_n(\theta)$ corresponding to the Gaussian observations $Y^{n}$. By Taylor formula, we have
\begin{align*}
-2 \log ( Z_n(u) ) 
&=
\transp{u} 
\transp{\varphi_n} 
\nabla \salL_n ( \theta_0 )
+
\frac{1}{2}
\transp{u} 
\transp{\varphi_n} 
\nabla^2 \salL_n ( \theta_0 )
\varphi_n
u
\\
&\;\;\;\;+
\tfrac{1}{2}
\sum_{i,j,k \in \{ H, \sigma \}}
\int_0^1
(1-s)^2
\partial^3_{ijk}
\salL_n (\theta_0 + s\varphi_n u)
( \varphi_n u )_i
( \varphi_n u )_j
( \varphi_n u )_k
ds.
\end{align*}
The proof is therefore an easy consequence of the following lemma:
\begin{lemma}
\label{lemma:lan:sketch}
\begin{enumerate}
[label=(\Roman*), ref=\concatenate{\ref{lemma:lan:sketch}}{{(\Roman*)}}] \setlength{\itemsep}{0pt}
  \setlength{\parskip}{0pt}
\item 
\label{lemma:lan:sketch:1}
$\transp{\varphi_n} 
\nabla \salL_n ( \theta_0 ) \to \gaussian{0}{4I}$ in distribution, under $P^n_{\theta_0}$.
\item 
\label{lemma:lan:sketch:2}
$\transp{\varphi_n} 
\nabla^2 \salL_n ( \theta_0 )
\varphi_n \to 2I$ in distribution, under $P^n_{\theta_0}$.
\item 
\label{lemma:lan:sketch:3} For any $i,j,k$ corresponding either to "$H$" (or $1$) or "$\sigma$" (or $2$), we have
\begin{align*}
\sup_{0 \leq s \leq 1}
 \Big|
\partial^3_{ijk}
\salL_n (\theta_0 + s\varphi_n u)
( \varphi_n u )_i
( \varphi_n u )_j
( \varphi_n u )_k
 \Big|
\to 0 \text{ under } P^n_{\theta_0}.
\end{align*}
\end{enumerate}
\end{lemma}

\subsection{Preliminary: Toeplitz matrices}

The purpose of this section is to prove the convergence of the trace of products/inverse of Toeplitz matrices that is often used in the proof of Lemma \ref{lemma:lan:sketch}. This result is a generalization of Theorem 2.3 of \cite{cohen2011lan} that was not found in the literature by the author. Let us first state this theorem that is the base of our work:

\begin{proposition}
\label{propo:toeplitz:cohen}
Let $\Theta^* \subset \mathbb{R}^m$ be a compact set and let $p \in \setN \backslash \{0\}$. Consider for any $1 \leq l \leq p$ and any $\theta \in \Theta^*$ some even functions $f_{l, \theta}: [-\pi, \pi] \to [0, \infty]$ and $g_{l, \theta}= [-\pi, \pi] \to [-\infty, \infty]$ such that the following conditions hold
\begin{tight_enumerate}
\item \label{thm:toeplitz:cohen:hyp1} for any $1 \leq l \leq p$ and any $\theta \in \Theta^*$, $f_{l, \theta}$ and $g_{l, \theta}$ are differentiable on $[-\pi, \pi] \backslash \{0\}$.
\item \label{thm:toeplitz:cohen:hyp2} for any $1 \leq l \leq p$, any $n \geq 1$ and any $\theta \in \Theta^*$, the functions
$(\theta, \lambda) \mapsto f_{l, \theta}(\lambda)$,
$(\theta, \lambda) \mapsto \partial_\lambda f_{l, \theta}(\lambda)$,
$(\theta, \lambda) \mapsto g_{l, \theta}(\lambda)$, and
$(\theta, \lambda) \mapsto \partial_\lambda g_{l, \theta}(\lambda)$
are continuous on  on $\Theta^* \times ( [-\pi, \pi] \backslash \{0\}) $.
\item \label{thm:toeplitz:cohen:hyp3} There exist continuous functions $\alpha: \Theta^* \to (-1, 1)$ and $\beta: \Theta^* \to (-\infty, 1)$ such that for any $\delta > 0$, for every $\theta \in \Theta^*$,  any $\lambda \in [-\pi, \pi] \backslash \{0\}$ and any $1 \leq l \leq p$, there exists $c, C > 0$ constants depending only on $\delta$ and on $\Theta^*$ such that:
$c | \lambda |^{-\alpha(\theta) + \delta}
\leq
f_{l,\theta}(\lambda)
\leq
C |\lambda|^{-\alpha(\theta) - \delta}
$, $
 |
\partial_\lambda f_{l, \theta}(\lambda)
 |
\leq
C |\lambda|^{-\alpha(\theta) - \delta - 1}
$ and 
$ | g_{l,\theta}(\lambda)
 |
\leq
C |\lambda|^{-\beta(\theta) - \delta}
$.
\item \label{thm:toeplitz:cohen:hyp4} For any $\theta \in \Theta^*$, we have $p(\beta(\theta) - \alpha(\theta))<1$.
\end{tight_enumerate}

Then we have the following limit:
\begin{align*}
\lim\limits_{n\to\infty}
\sup_{\theta \in \Theta^*}
 \bigg|
\frac{1}{n} \Tr \bigg(
	\prod_{j=1}^p
	\Sigma_n( f_{j,\theta} )^{-1}
	\Sigma_n( g_{j,\theta} )
\bigg)
-
\frac{1}{2\pi}
\int_{-\pi}^{\pi}
\prod_{j=1}^p
	f_{j,\theta}( \lambda)^{-1}
	g_{j,\theta}( \lambda)
d\lambda
 \bigg|
=
0.
\end{align*}
\end{proposition}

We give here a version of this result where we consider sequences of functions $f_\theta$ and $g_\theta$:

\begin{proposition}
\label{propo:toeplitz:inversetoplitz}
Let $\Theta^* \subset \mathbb{R}^m$ be a compact set and let $p \geq 1$ an integer. Consider for any $1 \leq l \leq p$, any $n \geq 1$ of $n = \infty$ and any $\theta \in \Theta^*$ some even functions $f^{n}_{l, \theta}= [-\pi, \pi] \to [0, \infty]$ and $g^{n}_{l, \theta}= [-\pi, \pi] \to [-\infty, \infty]$ such that the following conditions hold
\begin{tight_enumerate}
\item For each $n \in ( \setN \backslash \{ 0 \} ) \cup \{ \infty \}$ the functions $f^{n}$ and $g^{n}$ satisfy Assumptions \ref{thm:toeplitz:cohen:hyp1} to \ref{thm:toeplitz:cohen:hyp4} given in Proposition \ref{propo:toeplitz:cohen} for some functions $\alpha$, $\beta$ independant of $n$.
\item \label{thm:toeplitz:inversetoplitz:hyp_limit} For any sequence $\theta_n \to \theta$ in $\Theta^*$ and any sequence $\lambda_n \to \lambda \neq 0$, we have:
\begin{align*}
\begin{cases}
f_{l, \theta_n}^{n}(\lambda_n) \to f_{l, \theta}^{(\infty)}(\lambda)
\;\;\text{ and }\;\;
\partial_\lambda f_{l, \theta_n}^{n}(\lambda_n) \to \partial_\lambda f_{l, \theta}^{(\infty)}(\lambda),
\\
g_{l, \theta_n}^{n}(\lambda_n) \to g_{l, \theta}^{(\infty)}(\lambda)
\;\;\text{ and }\;\;
\partial_\lambda g_{l, \theta_n}^{n}(\lambda_n) \to \partial_\lambda g_{l, \theta}^{(\infty)}(\lambda).
\end{cases}
\end{align*}
\item \label{thm:toeplitz:inversetoplitz:hyp_int}
$
 \big|
\frac{1}{2\pi}
\int_{-\pi}^{\pi}
\prod_{j=1}^p
	f^{n}_{j,\theta}( \lambda)^{-1}
	g^{n}_{j,\theta}( \lambda) 
-
\prod_{j=1}^p
	f^{(\infty)}_{j,\theta}( \lambda)^{-1}
	g^{(\infty)}_{j,\theta}( \lambda)
d\lambda
 \big|
\to
0
$ uniformly over $\theta \in \Theta^*$.
\end{tight_enumerate}
Then $
 |
n^{-1} \Tr ( 
	\prod_{j=1}^p
	\Sigma_n( f_{j,\theta}^{n} )^{-1}
	\Sigma_n( g_{j,\theta}^{n} )
)
-
\frac{1}{2\pi}
\int_{-\pi}^{\pi}
\prod_{j=1}^p
	f^{(\infty)}_{j,\theta}( \lambda)^{-1}
	g^{(\infty)}_{j,\theta}( \lambda) 
d\lambda
 | \to
0
$ uniformly for $\Theta^*$.
\end{proposition}
Note in addition that Assumptions \ref{thm:toeplitz:inversetoplitz:hyp_int} hold whenever for any $0 < \eta < \pi$, we have concurrently
$
 | 
	f_{l, \theta}^{n}(\lambda)  
	- 
	f_{l, \theta}^{(\infty)}(\lambda)  | \to 0
$ and $
 | 
	g_{l, \theta}^{n}(\lambda)  
	- 
	g_{l, \theta}^{(\infty)}(\lambda)  | \to 0
$ uniformly for $\theta \in \Theta^*$ and $|\lambda| \geq \eta$.

\begin{proof}
Let $K =  \{ n^{-1} , n \geq 1 \backslash \{ 0 \} )  \}\cup \{ 0 \} $ which is a compact subset of $\mathbb{R}$. For simplicity, we map directly $n$ with $n^{-1}$ in the following. Let now $\Theta = \Theta^* \times K$ which is still a compact subset of $\mathbb{R}^{m+1}$. Then Assumption \ref{thm:toeplitz:inversetoplitz:hyp_limit} in Proposition \ref{propo:toeplitz:inversetoplitz} ensures we can still apply Proposition \ref{propo:toeplitz:cohen} with $\Theta$ instead of $\Theta^*$. Then we have
\begin{align*}
\lim\limits_{n\to\infty}
\sup_{\theta \in \Theta^*}
\sup_{m \in \setN^* \cup \{ \infty \}}
 \bigg|
\frac{1}{n} \Tr \bigg( 
	\prod_{j=1}^p
	\Sigma_n\big( f_{j,\theta}^{(m)} \big)^{-1}
	\Sigma_n\big( g_{j,\theta}^{(m)} \big)
\bigg)
-
\frac{1}{2\pi}
\int_{-\pi}^{\pi}
\prod_{j=1}^p
	f_{j,\theta}^{(m)}( \lambda)^{-1}
	g_{j,\theta}^{(m)}( \lambda)
d\lambda
 \bigg|
=
0.
\end{align*}
Thus we obtain
\begin{align*}
&\sup_{\theta \in \Theta^*}
 \bigg|
\frac{1}{n} \Tr \bigg(
	\prod_{j=1}^p
	\Sigma_n( f_{j,\theta}^{n} )^{-1}
	\Sigma_n( g_{j,\theta}^{n} )
\bigg)
-
\frac{1}{2\pi}
\int_{-\pi}^{\pi}
\prod_{j=1}^p
	f^{(\infty)}_{j,\theta}( \lambda)^{-1}
	g^{(\infty)}_{j,\theta}( \lambda)
d\lambda
 \bigg|
\leq
\dots
\\
&\;\;\;\;\leq
\sup_{\theta \in \Theta^*}
 \bigg|
\frac{1}{n} \Tr \bigg(
	\prod_{j=1}^p
	\Sigma_n( f_{j,\theta}^{n} )^{-1}
	\Sigma_n( g_{j,\theta}^{n} )
\bigg)
-
\frac{1}{2\pi}
\int_{-\pi}^{\pi}
\prod_{j=1}^p
	f^{n}_{j,\theta}( \lambda)^{-1}
	g^{n}_{j,\theta}( \lambda)
d\lambda
 \bigg|
\\
&\;\;\;\;\;\;\;\;+
\sup_{\theta \in \Theta^*}
 \bigg|
\frac{1}{2\pi}
\int_{-\pi}^{\pi}
\prod_{j=1}^p
	f^{n}_{j,\theta}( \lambda)^{-1}
	g^{n}_{j,\theta}( \lambda)
d\lambda
-
\frac{1}{2\pi}
\int_{-\pi}^{\pi}
\prod_{j=1}^p
	f^{(\infty)}_{j,\theta}( \lambda)^{-1}
	g^{(\infty)}_{j,\theta}( \lambda)
d\lambda
 \bigg|
\end{align*}
which converges to $0$.
\end{proof}

\subsection{Proof of Lemma \ref{lemma:lan:sketch}}

\paragraph*{Proof of Lemma \ref{lemma:lan:sketch:1}}

We want to prove $\transp{\varphi_n} 
\nabla \salL_n ( \theta_0 ) \to \gaussian{0}{4I}$ under $P^n_{\theta_0}$. Let $(u,v) \in \mathbb{R}^2$. We show first
\begin{align*}
\zeta_n := \zeta_n(u,v) =
\begin{pmatrix}
u&v
\end{pmatrix}
\transp{\varphi_n} 
\nabla \salL_n ( \theta_0 ) \to \gaussian{0}{{J}_{u,v}}
\end{align*}
under $P^n_{\theta_0}$ which proves $\transp{\varphi_n} 
\nabla \salL_n ( \theta_0 ) \to \gaussian{0}{4I}$ and identification of ${J}_{u,v}$ gives explicitly $I$. First, notice that
\begin{align*}
\varphi_n
\begin{pmatrix}
u\\v
\end{pmatrix}
=
\frac{1}{\sqrt{n}}
\begin{pmatrix}
\widetilde{\alpha_n}
\\
\widetilde{\beta_n}
\end{pmatrix}
\text{ with } 
\begin{cases}
\widetilde{\alpha_n} = u \alpha_n + v \overline{\alpha_n},
\\
\widetilde{\beta_n} = u \beta_n + v \overline{\beta_n}.
\end{cases}
\end{align*}
Recall in addition that for any $\theta$, $\salL_n(\theta)$ is given by
\begin{align*}
\salL_n( \theta ) = \transp{Y^{n}}\Sigma_n^{-1}\big( f^{n}_\theta\big) Y^{n} + \log \det \big( \Sigma_n\big(f^{n}_\theta\big)\big).
\end{align*}
Thus we can compute explicitly the derivatives of $\salL_n$. If $\partial$ stands either for $\partial_H$ or $\partial_\sigma$, we have
\begin{align*}
\partial \salL_n ( \theta )
=
-
\transp{Y^{n}}
\Sigma_n^{-1} \big( f^{n}_\theta\big)
\Sigma_n\big( \partial f^{n}_\theta\big) 
\Sigma_n^{-1}
Y^{n}
+ \Tr ( 
\Sigma_n^{-1} \big( f^{n}_\theta\big)
\Sigma_n\big( \partial f^{n}_\theta\big) 
).
\end{align*}
Therefore, we get
\begin{align*}
\zeta_n(u,v)
&=
n^{-1/2}
\Tr \big(
\Sigma_n^{-1} \big( f^{n}_{\theta_0}\big)
\Sigma_n\big( \widetilde{\alpha_n} \partial_H f^{n}_{\theta_0} + \widetilde{\beta_n} \partial_\sigma f^{n}_{\theta_0} \big)
\big)
\\
&\;\;\;\;-
n^{-1/2}
\transp{Y^{n}}
\Sigma_n^{-1} \big( f^{n}_{\theta_0}\big)
\Sigma_n\big( \widetilde{\alpha_n} \partial_H f^{n}_{\theta_0} + \widetilde{\beta_n} \partial_\sigma f^{n}_{\theta_0}\big)
\Sigma_n^{-1}
Y^{n}.
\end{align*}

As $Y^{n}$ is a centered Gaussian with covariance $\Sigma_n \big(f_{\theta_0}^{n}\big)$, we can compute the cumulants of $\zeta_n(u,v)$ using Lemma \ref{lemma:moments_quad_form}. This variable is centered and its cumulants are given for $k \geq 2$ by
\begin{align*}
\kappa_k ( \zeta_n(u,v) )
=
n^{-\frac{k}{2}}
\Tr \big[ \big( 
\Sigma_n^{-1} \big( f^{n}_{\theta_0}\big)
\Sigma_n\big( \widetilde{\alpha_n} \partial_H f^{n}_{\theta_0} + \widetilde{\beta_n} \partial_\sigma f^{n}_{\theta_0} \big) \big)^k
\big].
\end{align*}
Since $f_{H, \sigma}^{n} = \sigma^2 n^{-2\gamma(H)} f_H + \tau_n^2l$ by Equation \eqref{eq:decaying:psd}, we have
\begin{align*}
\begin{cases}
f_{H, \sigma}^{n}(\lambda) = \sigma^2 n^{-2\gamma(H)} f_H(\lambda) + \tau_n^2 l(\lambda),
\\
\partial_H f_{H, \sigma}^{n}(\lambda) = \sigma^2 n^{-2\gamma(H)} \partial_H f_H(\lambda) - 2\gamma'(H)\sigma^2\log( n ) n^{-2\gamma(H)} f_H(\lambda) ,
\\
\partial_\sigma f_{H, \sigma}^{n}(\lambda) = 2\sigma n^{-2\gamma(H)} f_H(\lambda).
\end{cases}
\end{align*}
Thus we rewrite $\widetilde{\alpha_n} \partial_H f_{\theta_0} + \widetilde{\beta_n} \partial_\sigma f_{\theta_0}$ as follows:
\begin{align*}
\widetilde{\alpha_n} \partial_H f_{\theta_0} + \widetilde{\beta_n} \partial_\sigma f_{\theta_0}
=
{n^{-2\gamma(H_0)}}
\Big(
\widetilde{\alpha_n} \sigma_0^2 \partial_H f_{H_0}(\lambda) + 2 \sigma_0^2 ( - \widetilde{\alpha_n}\gamma'(H_0)\log n   +  \widetilde{\beta_n}\sigma_0^{-1}) f_{H_0}(\lambda)
\Big).
\end{align*}
We write $h_{\theta_0}^{n} = \widetilde{\alpha_n} \sigma_0^2 \partial_H f_{H_0} + 2 \sigma_0^2 ( - \widetilde{\alpha_n}\gamma'(H_0)\log n    +   \widetilde{\beta_n}\sigma_0^{-1} ) f_{H_0}$. By the assumptions in Theorem \ref{thm:decaying:LAN}, $h_{\theta_0}^{n} \to h_{\theta_0}^{(\infty)} := \widetilde{\alpha} \sigma_0^2 \partial_H f_{H_0} + 2 \sigma_0^2 \widetilde{\gamma} f_{H_0}$ on compact subsets of $[-\pi, 0) \cup (0, \pi]$ where $\widetilde{\gamma} = u \widetilde{\alpha} + v \widetilde{\beta}$. Moreover, $n^{2\gamma(H_0)} f^{n} \to \sigma_0^2 f_{H_0}$ uniformly on $[-\pi, \pi]$ so we eventually get, using Proposition \ref{propo:toeplitz:inversetoplitz}
\begin{align*}
\kappa_k \big( \zeta_n(u,v) \big)
&
=
2^{k-1} (k-1)!
n^{-\frac{k}{2}}\Tr \big[\big(
\Sigma_n^{-1} \big( n^{2\gamma(H_0)} f^{n}_{\theta_0}\big)
\Sigma_n\big( h_{\theta_0}^{(\infty)}(\lambda) \big) \big)^k
\big]
\\
&\to
\bigg(
\frac{1}{\pi}
\int_{-\pi}^\pi
\Big(
\sigma_0^{-2} f_{H_0}(\lambda)^{-1} h_{\theta_0}^{(\infty)}(\lambda) 
\Big)^2
d\lambda
\bigg)\delta_{k,2}
\end{align*}
where $\delta_{k,2} = 1$ if $k=2$ and $0$ otherwise. We recognise here the cumulants of a Gaussian variable so we conclude by the methods of moments (see for instance \cite{billingsley2008probability}).
We also identify $J_{u,v}$:
\begin{align*}
J_{u,v}
&=
\frac{1}{\pi}
\int_{-\pi}^\pi
\Big(
	u
	\frac
		{\alpha \partial_H f_{H_0} (\lambda) + 2 \gamma f_{H_0} (\lambda) }
		{f_{H_0}(\lambda) }
+
	v
	\frac
		{\overline{\alpha} \partial_H f_{H_0} (\lambda) + 2 \overline{\gamma} f_{H_0} (\lambda) }
		{f_{H_0}(\lambda) }
\Big)^2
d\lambda
\end{align*}
and we eventually express $I$ in terms of $J$:
$I_{1,1} = \frac{1}{4} J_{1,0}$, $
I_{1,2} = \frac{1}{8} ( J_{1,1} - J_{1,0} - J_{0,1} )$ and $
I_{2,2} = \frac{1}{4} J_{0,1}$.

\paragraph*{Proof of Lemma \ref{lemma:lan:sketch:2}}

We now prove that $\transp{\varphi_n} 
\nabla^2 \salL_n ( \theta_0 )
\varphi_n \to I$ in $P^n_{\theta_0}$-distribution. Indeed, we can compute the second order derivative of $\salL_n $. For conciseness, $\Sigma_n^{-1}$ will refer to $\Sigma_n^{-1}\big( f_{\theta}^{n} \big)$ and $f_\theta$ to $f_\theta^{n} $ in the following expression only. Let $i$ being either $H$ or $\sigma$ and $j$ also being either $H$ or $\sigma$. We have
\begin{align*}
\partial^2_{ij} \salL_n ( \theta )
&=
\transp{Y}
\Sigma_n^{-1}
\Sigma_n\big( \partial_i f_\theta\big) 
\Sigma_n^{-1}
\Sigma_n\big( \partial_j f_\theta\big) 
\Sigma_n^{-1}
Y
+
\transp{Y}
\Sigma_n^{-1}
\Sigma_n\big( \partial_j f_\theta\big) 
\Sigma_n^{-1}
\Sigma_n\big( \partial_i f_\theta\big) 
\Sigma_n^{-1}
Y
\\&\;\;\;\;
-
\transp{Y}
\Sigma_n^{-1}
\Sigma_n\big( \partial^2_{ij} f_\theta\big) 
\Sigma_n^{-1}
Y
-
\Tr ( 
\Sigma_n^{-1}
\Sigma_n\big( \partial_i f_\theta\big) 
\Sigma_n^{-1}
\Sigma_n\big( \partial_j f_\theta\big) 
)
+ 
\Tr ( 
\Sigma_n^{-1}
\Sigma_n\big( \partial^2_{ij} f_\theta\big) 
).
\end{align*}
Moreover, $\transp{\varphi_n} 
\nabla^2 \salL_n ( \theta_0 )
\varphi_n $ is composed of the three following terms:
\begin{align*}
\begin{cases}
A_n = n^{-1} \big(
	\alpha_n^2 \partial^2_{HH} \salL_n
	+
	2\alpha_n\beta_n \partial^2_{H\sigma} \salL_n
	+
	\beta_n^2 \partial^2_{\sigma\sigma} \salL_n
\big),
\\
B_n = n^{-1} \big(
	\alpha_n\overline{\alpha_n} \partial^2_{HH} \salL_n
	+
	(\alpha_n\overline{\beta_n} + \overline{\alpha_n}\beta_n) \partial^2_{H\sigma} \salL_n
	+
	\beta_n\overline{\beta_n} \partial^2_{\sigma\sigma} \salL_n
\big),
\\
C_n = n^{-1} \big(
	\overline{\alpha_n}^2 \partial^2_{HH} \salL_n
	+
	2\overline{\alpha_n}\overline{\beta_n} \partial^2_{H\sigma} \salL_n
	+
	\overline{\beta_n}^2 \partial^2_{\sigma\sigma} \salL_n
\big).
\end{cases}
\end{align*}

We shall consider the convergence in distribution of these terms as $n \to \infty$. As the limit is expected to be constant (it should converge towards the coefficients of $I$), so we can study these terms separately. Let us focus for instance on the first one. We should consider and deal with the two others similarly. We write
$A_n = n^{-1}
\big(
\transp{Y}^{n}
\Sigma_n^{-1}
\big( f_{\theta_0}^{n}\big)
\Lambda_n
Y^{n}
+ \Tr \big( \Gamma_n \big)
\big)
$
where $\Lambda_n$ and $\Gamma_n$ are the matrices given by
\begin{align*}
\begin{cases}
\Lambda_n = 2
\big(
\Sigma_n\big( \alpha_n
\partial_H f_{\theta_0}^{n}
+
\beta_n
\partial_\sigma f_{\theta_0}^{n}\big) 
\Sigma_n^{-1}
\big)^2
-
\Sigma_n\big( 
	\alpha_n^2
	\partial^2_{HH} f_{\theta_0}^{n}
	+
	2 \alpha_n\beta_n
	\partial^2_{\sigma H} f_{\theta_0}^{n}
	+
	\beta_n^2
	\partial^2_{\sigma\sigma} f_{\theta_0}^{n}
\big) 
\Sigma_n^{-1},
\\
\Gamma_n = 
-
\big(
\Sigma_n\big( \alpha_n
\partial_H f_{\theta_0}^{n}
+
\beta_n
\partial_\sigma f_{\theta_0}^{n}\big) 
\Sigma_n^{-1} 
\big)^2
+\Sigma_n\big( 
	\alpha_n^2
	\partial^2_{HH} f_{\theta_0}^{n}
	+
	2 \alpha_n\beta_n
	\partial^2_{\sigma H} f_{\theta_0}^{n}
	+
	\beta_n^2
	\partial^2_{\sigma\sigma} f_{\theta_0}^{n}
\big) 
\Sigma_n^{-1}
\end{cases}
\end{align*}
and where $\Sigma_n^{-1}$ stands for $\Sigma_n^{-1}\big( f_{\theta_0}^{n} \big)$.
Since $Y^{n}$ is a centered Gaussian with covariance $\Sigma_n\big( f_{\theta_0}^{n} \big)$ under $P^n_{\theta_0}$, we deduce from the results of Lemma \ref{lemma:moments_quad_form} that:
\begin{align*}
\begin{cases}
\EX{A_n} = n^{-1} \Tr ( \Lambda_n + \Gamma_n ) = n^{-1} \Tr\Big[ \big(
\Sigma_n\big( \alpha_n
\partial_H f_{\theta_0}^{n}
+
\beta_n
\partial_\sigma f_{\theta_0}^{n}\big)
\Sigma_n^{-1} \big( f_{\theta_0}^{n} \big)
\big)^2\Big],
\\
\Var{A_n} = \frac{1}{n^2} \Tr \big( \Lambda_n^2 \big).
\end{cases}
\end{align*}
We conclude with Proposition \ref{propo:toeplitz:inversetoplitz} as in the first step. We deal with the two other terms similarly.

\paragraph*{Proof of Lemma \ref{lemma:lan:sketch:3}}

We still need to prove that for any $i,j,k$ corresponding either to $H$ or to $\sigma$,
$\sup_{0 \leq s \leq 1}
 |
\partial^3_{ijk}
\salL_n (\theta_0 + s\varphi_n u)
( \varphi_n u )_i
( \varphi_n u )_j
( \varphi_n u )_k
 |
\to 0
$ under $P^n_{\theta_0}$. Firstly notice that there exists a constant $C$ depending only on the choice of the sequences $( \alpha_n )_n$, $( \overline{\alpha_n} )_n$, $( \beta_n )_n$ and $( \overline{\beta_n} )_n$ but not on $n$ such that $\norm{\varphi_n u} \leq C \log (n) n^{-1} \norm{u}$. Thus we only need to prove that in distribution, under $P^n_{\theta_0}$, we have $\log^3 (n)n^{-3}
\sup_{0 \leq s \leq 1}
 |
\partial^3_{ijk}
\salL_n (\theta_0 + s\varphi_n u)
 |
\to 0
$.

For simplicity, we focus here on the term $i=j=k=H$ here, though the other terms are similar. Then $\partial^3_{HHH} \salL_n ( \theta )$ is expressed as
\begin{align*}
\partial^3_{HHH} \salL_n ( \theta )
&=
3
\transp{Y}
\Sigma_n^{-1}
\Sigma_n( \partial_H f_\theta) 
\Sigma_n^{-1}
\Sigma_n( \partial^2_{HH} f_\theta) 
\Sigma_n^{-1}
Y
+
3
\transp{Y}
\Sigma_n^{-1}
\Sigma_n( \partial^2_{HH} f_\theta) 
\Sigma_n^{-1}
\Sigma_n( \partial_H f_\theta) 
\Sigma_n^{-1}
Y
\\&\;\;\;\;
- 6
\transp{Y}
\Sigma_n^{-1}
(
\Sigma_n( \partial_H f_\theta) 
\Sigma_n^{-1}
)^3
Y
-
\transp{Y}
\Sigma_n^{-1}
\Sigma_n( \partial^3_{HHH} f_\theta) 
\Sigma_n^{-1}
+
\Tr ( 
\Sigma_n^{-1}
\Sigma_n( \partial^3_{HHH} f_\theta) 
)
\\&\;\;\;\;
+2
\Tr 
(
( 
\Sigma_n^{-1}
\Sigma_n( \partial_H f_\theta) 
)^3
)
-2
\Tr ( 
\Sigma_n^{-1}
\Sigma_n( \partial_HH f_\theta) 
\Sigma_n^{-1}
\Sigma_n( \partial^2_{HH} f_\theta) 
).
\end{align*}

We deal with each term separately. Again for conciseness, we consider only the limits of the terms $\sup_{s}
 |
\transp{Y}^{n}
\Sigma_n^{-1}
(
\Sigma_n( \partial_H f^{n}_{\theta_0 + s\varphi_n u}) 
\Sigma_n^{-1}
)^3
Y^{n}
 |$ 
and 
$\sup_{s}
 |
\Tr 
(
( 
\Sigma_n^{-1}
\Sigma_n( \partial_H f^{n}_{\theta_0 + s\varphi_n u}) 
)^3
)
 |$, although the five other terms appearing in this sum should be considered and controlled with the same tools. These two terms are studied in the following lemmas and their proof concludes the proof of Lemma \ref{lemma:lan:sketch:3}.

\begin{lemma}
\label{lemma:LAN:partial3H}
For $r > 0$ small enough, $
\log^3 (n)n^{-3}
\sup_{\theta \in B(\theta_0, r)}
 |
\transp{Y}^{n}
\Sigma_n^{-1}
(
\Sigma_n( \partial_H f_{\theta}^{n}) 
\Sigma_n^{-1}
)^3
Y^{n}
 |\to 0
$ in $P^n_{\theta_0}$-probability where $B(\theta_0, r)$ is the closed ball centered at $\theta_0$ with radius $r$. 
\end{lemma}

\begin{lemma}
\label{lemma:LAN:Tr3H}
$\log^3 (n)n^{-3}
\sup_{0 \leq s \leq 1}
 |
\Tr 
(
( 
\Sigma_n^{-1}
\Sigma_n( \partial_H f_{\theta_0 + s\varphi_n u}) 
)^3
)
 |
\to 0
$
in $P^n_{\theta_0}$-probability.
\end{lemma}

Before we prove them, we introduce the spectral norm on Toeplitz matrices. Let $A$ be a $n \times n$ matrix. We write $\norm{A}_{2,n}$ its spectral norm. More precisely, it is defined by
\begin{align}
\label{eq:LAN:operatordef}
\norm{A}_{2,n} = \sup_{x \neq 0} \bigg( \frac{x^*A^*Ax}{x^*x}\bigg)^{1/2}
\end{align}
where $x$ is seen as a $d$-dimensional complex number and $x^*$ and $A^*$ refer to the conjugate of $x$ and $A$ respectively. This norm has interesting properties: if $A$ and $B$ are $n \times n$ matrices, then $\norm{AB}_{2,n} \leq \norm{A}_{2,n} \norm{B}_{2,n} $ and if $x\in\mathbb{R}^n$, then $\transp{x}Ax \leq \norm{A}_{2,n}\norm{x}^2$. Recall that $\Sigma_n(f) = ( \widehat{f}(i-j) )_{1 \leq i,j \leq n}$. When $f$ is even, $\Sigma_n(f)$ is positive and real so the supremum in  \eqref{eq:LAN:operatordef} can be taken with respect to real vectors and the transposition symbol can replace the conjugation one.

\begin{proof}[Proof of Lemma \ref{lemma:LAN:partial3H}]
Let 
$
H_n(\theta) =  |
\transp{Y}^{n}
\Sigma_n^{-1} (  f_{\theta}^{n} )
(
\Sigma_n( \partial_H f_{\theta}^{n}) 
\Sigma_n^{-1} (  f_{\theta}^{n} )
)^3
Y^{n}
 |
$
be defined 
for any $\theta = (H,\sigma) \in B(\theta_0, r)$, 
so that we have
\begin{align*}
H_n(\theta)
&=
\bignorm{\Sigma_n^{-1/2} \big(  f_{\theta_0}^{n} \big) Y^{n}}^2
\bignorm{
\Sigma_n^{1/2} \big(  f_{\theta_0}^{n} \big)
\Sigma_n^{-1} \big(  f_{\theta}^{n} \big)
\big(
\Sigma_n\big( \partial_H f_{\theta}^{n}\big) 
\Sigma_n^{-1} \big(  f_{\theta}^{n} \big)
\big)^3
\Sigma_n^{1/2} \big(  f_{\theta_0}^{n} \big)
}_{2,n}
\\
&\leq
\bignorm{\Sigma_n^{-1/2} \big(  f_{\theta_0}^{n} \big) Y^{n}}^2
\bignorm{
\Sigma_n^{-1/2} \big(  f_{\theta}^{n} \big)
\Sigma_n\big( \partial_H f_{\theta}^{n}\big) 
\Sigma_n^{-1/2} \big(  f_{\theta}^{n} \big)
}_{2,n}^3
\\&\;\;\;\;\times
\bignorm{
\Sigma_n^{1/2} \big(  f_{\theta_0}^{n} \big)
\Sigma_n^{-1/2} \big(  f_{\theta}^{n} \big)
}_{2,n}
\bignorm{
\Sigma_n^{-1/2} \big(  f_{\theta}^{n} \big)
\Sigma_n^{1/2} \big(  f_{\theta_0}^{n} \big)
}_{2,n}
\\
&\leq
\bignorm{\Sigma_n^{-1/2} \big(  f_{\theta_0}^{n} \big) Y^{n}}^2
\bignorm{
\Sigma_n^{-1/2} \big(  f_{\theta}^{n} \big)
\Sigma_n^{1/2} \big( \partial_H f_{\theta}^{n}\big) 
}_{2,n}^6
\bignorm{
\Sigma_n^{1/2} \big(  f_{\theta_0}^{n} \big)
\Sigma_n^{-1/2} \big(  f_{\theta}^{n} \big)
}_{2,n}^2.
\end{align*}
Moreover, we have
\begin{align*}
\bignorm{
\Sigma_n^{-1/2} \big(  f_{\theta}^{n} \big)
\Sigma_n^{1/2} \big( \partial_H f_{\theta}^{n}\big) 
}_{2,n}^2
&=
\sup_{x \neq 0} \frac{
	\transp{x}
	\Sigma_n^{-1/2} \big(  f_{\theta}^{n} \big)
	\Sigma_n \big( \partial_H f_{\theta}^{n}\big) 
	\Sigma_n^{-1/2} \big(  f_{\theta}^{n} \big)
x}{\transp{x}x}
=
\sup_{x \neq 0} \frac{
	\transp{x}
	\Sigma_n \big( \partial_H f_{\theta}^{n}\big) 
x}{\transp{x}	\Sigma_n \big(  f_{\theta}^{n} \big)
x}
\end{align*}
and for any $x \in \mathbb{R}^n$, we notice that
\begin{align*}
\transp{x}
	\Sigma_n \big( \partial_H f_{\theta}^{n}\big) 
x
&=
\frac{1}{2\pi}
\int_{-\pi}^\pi
\partial_H f_{\theta}^{n}(\lambda)  \Big| \sum_{k} e^{ik\lambda} x_k  \Big|^2
d\lambda
\\
&\leq 
\frac{1}{2\pi}
\int_{-\pi}^\pi
 \big| \partial_H f_{\theta}^{n}(\lambda)  \big|  \Big| \sum_{k} e^{ik\lambda} x_k  \Big|^2
d\lambda
=
\transp{x}
	\Sigma_n \big(  \big| \partial_H f_{\theta}^{n}  \big| x\big) 
x
\end{align*}
so that finally we get
\begin{align*}
\bignorm{
\Sigma_n^{-1/2} \big(  f_{\theta}^{n} \big)
\Sigma_n^{1/2} \big( \partial_H f_{\theta}^{n}\big) 
}_{2,n}^2
\leq 
\bignorm{
\Sigma_n^{-1/2} \big(  f_{\theta}^{n} \big)
\Sigma_n^{1/2} \big(   \big| \partial_H f_{\theta}^{n}   \big|\big) 
}_{2,n}^2.
\end{align*}
Let $\varepsilon > 0$. Provided $r$ is chosen small enough, Assumption \ref{assumption:decaying:psd} ensures that we have some constants $c_1$ and $c_2$ (that do not depend on $\theta$ or $\lambda$) such that for any $\theta \in B(\theta_0, r)$ and any $\lambda \neq 0$, we have
\begin{align*}
\begin{cases}
n^{2H}
 |f_{\theta}^{n}(\lambda) | \geq c_1  | \lambda  |^{-\alpha(H_0) + \varepsilon},
\\
n^{2H}
 |f_{\theta}^{n}(\lambda) | \leq c_2  | \lambda  |^{-\alpha(H_0) - \varepsilon},
\\
n^{2H}
 |\partial_H f_{\theta}^{n}(\lambda) | \leq c_2  | \lambda  |^{-\alpha(H_0) - 2\varepsilon}.
\end{cases}
\end{align*}
Thus we can apply the following lemma proven in \cite{brouste2018lan} (Lemma 2.1):
\begin{lemma}
\label{lemma:LAN:uniform_operator_control}
Let $f$ and $g$ be nonnegative symmetric functions defined on $[-\pi, \pi]$. Assume that there exist $c_1, c_2 > 0$ and $\beta_1, \beta_2 < 1$ such that 
$f(x) \geq c_1 |x|^{-\beta_1}$ and $g(x) \leq c_2 |x|^{-\beta_2}$. Then there exists a constant $K$ depending only on $c_1, c_2, \beta_1, \beta_2$ such that for any integer $n \geq 1$, we have
\begin{align*}
\bignorm{\Sigma_n(f)^{-1/2}\Sigma_n(g)^{1/2}}_{2,n}
=
\bignorm{\Sigma_n(g)^{1/2}\Sigma_n(f)^{-1/2}}_{2,n}
\leq 
K
n^{\max (\beta_2-\beta_1, 0 )/2}.
\end{align*}
\end{lemma}

Thus there exists a constant $K$ that does not depend on $n$ and $\theta$ such that for any $\theta \in B(\theta_0, r)$, 
$\bignorm{
\Sigma_n^{-1/2} (  f_{\theta}^{n} )
\Sigma_n^{1/2} (   | \partial_H f_{\theta}^{n}   |) 
}_{2,n}^2
\leq 
Kn^{3\varepsilon}
$ and $ \bignorm{
\Sigma_n^{1/2} (  f_{\theta_0}^{n} )
\Sigma_n^{-1/2} (  f_{\theta}^{n} )
}_{2,n}^2
\leq 
Kn^{\varepsilon}
$.
We obtain
\begin{align*}
H_n(\theta) &\leq 
\bignorm{\Sigma_n^{-1/2} \big(  f_{\theta_0}^{n} \big) Y^{n}}^2
\bignorm{
\Sigma_n^{-1/2} \big(  f_{\theta}^{n} \big) 
\Sigma_n^{1/2} \big( \partial_H f_{\theta}^{n}\big)  
}_{2,n}^6
\bignorm{
\Sigma_n^{1/2} \big(  f_{\theta_0}^{n} \big) 
\Sigma_n^{-1/2} \big(  f_{\theta}^{n} \big) 
}_{2,n}^2
\\&
\leq
Kn^{10\varepsilon}
\bignorm{\Sigma_n^{-1/2} \big(  f_{\theta_0}^{n} \big)  Y^{n}}^2.
\end{align*}

Finally, recall $Y^{n}$ is a centered Gaussian vector with covariance $\Sigma_n \big(  f_{\theta_0}^{n} \big) $ so $\bignorm{\Sigma_n^{-1/2} \big(  f_{\theta_0}^{n} \big)  Y^{n}}^2$ has a 
chi-squared distribution with $n$ degrees of freedom, so that $\bignorm{\Sigma_n^{-1/2} \big(  f_{\theta_0}^{n} \big)  Y^{n}}^2 \sim 2n$ under $\zproba^{n}_{\theta_0}$. This concludes the proof as we take $\varepsilon$ small enough so that the convergence to $0$ holds.

\end{proof}

\begin{proof}[Proof of Lemma \ref{lemma:LAN:Tr3H}]

We consider $\Theta^* = B(\theta_0, r)$. We aim at applying Proposition \ref{propo:toeplitz:inversetoplitz} with $p=3$ to the functions $n^{H} f^{n}_{\theta}$ and $n^{H} \partial_H f^{n}_{\theta}$. Proceeding as in the last subsection, we check that all assumptions of this proposition hold and therefore we can easily prove Lemma \ref{lemma:LAN:Tr3H}.

\end{proof}

\section{Proof of Theorem \ref{thm:slow:optimal}}

\subsection{Outline of the proof}

It relies heavily on the properties of the preaveraged data
$\widehat{Z}_i^{n} 
=
\sigma k^{-1}
\sum_{j=0}^{k-1} W_{\frac{ik+j}{n}}^{H} 
+ \tau k^{-1/2}\widehat{\xi}^n_i$ introduced in section \ref{subsec:slow:construction} to obtain a quickly decaying noise.  Recall that the increments $Y_i^{n} = \Delta \widehat{Z}_i^{n} = \widehat{Z}_{i+1}^{n} -\widehat{Z}_{i}^{n}$ are stationary 
and by scale invariance their spectral density is given by $f^{n}_{H,\sigma} 
=
\sigma^2( k/n )^{2H} f_{H, k} + \tau^2k^{-1} l$ where $f_{H, k}$ is the power spectral density of the stationary process 
$ k^{-1}
\sum_{j=0}^{k-1} ( W_{i+1+j/k}^{H} - W_{i+j/k}^{H})_i$  and where $l( \lambda ) = 2 (1 - \cos \lambda)$
We start this proof by retrieving some properties of this spectral density.

\begin{lemma}
\label{lemma:slow:preaveraged_psd}
Consider $k \geq 1$ and $W^H$ a fractional Brownian motion with Hurst index $H$.
Define $X^{(k)}_n :=
k^{-1}
\sum_{j=0}^{k-1}
W^H_{n + j/k}$ and $Y^{(k)} = \Delta X^{(k)}$ its (stationary) increments. We also define $Y^{(\infty)}_n
=
\int_n^{n+1} W^H_{s} - W^H_{s-1} \,ds
$ which is a stationary process as well. Then the following holds:
\begin{enumerate}[label=\alph*), ref=\concatenate{\ref{lemma:slow:preaveraged_psd}}{{\alph*)}}]
  \setlength{\itemsep}{0pt}
  \setlength{\parskip}{0pt}
\item
\label{lemma:slow:preaveraged_psd:explicit_psd} 
\label{lemma:slow:preaveraged_psd:explicit_psd2} 
The power spectral density of $Y^{(k)}$, denoted $f_{H,k}$ is given for any $0 < |\lambda| < \pi$ by
\begin{align*}
f_{H, k}(\lambda)
=
\Gamma(2H+1) \sin ( \pi H )
\sum_{j \in \mathbb{Z}}
\frac{( 1 - \cos \lambda )^2}{k^2 \sin ( \frac{\lambda}{2k} + \frac{\pi j}{k} )^2}
\cdot
\frac{1}{ | \lambda + 2\pi j  |^{1+2H}}.
\end{align*}

\item
\label{lemma:slow:preaveraged_psd:explicit_psd_lim} 
The power spectral density of $Y^{(\infty)}$, denoted $f_{H,\infty}$ is given for any $0 < |\lambda| < \pi$ by
\begin{align*}
f_{H, \infty}(\lambda)
=
4\Gamma(2H+1) \sin ( \pi H )
( 1 - \cos \lambda )^2
\sum_{j \in \mathbb{Z}}
\frac{1}{ | \lambda + 2\pi j  |^{3+2H}}.
\end{align*}

\item
\label{lemma:slow:preaveraged_psd:uniform_conv}
$ |
f_{H,k}(\lambda) - f_{H,\infty}(\lambda)
 |
\to 0$ uniformly for $ \in [H_-, H_+]$ and $\lambda \neq 0$. More precisely,  that there exists a constant $C$ independent of $H \in [H_-, H_+]$ and of $k \geq 2$ such that
\begin{align*}
\forall\, \lambda \neq 0, \;
 | 
	f_{H, k}(\lambda)
	-
	f_{H, \infty}(\lambda)
 |
&\leq
C
\frac{|\lambda|^{2 \wedge ( 3-2H )}}{k^{1 \wedge 2H}}.
\end{align*}

\item
\label{lemma:slow:preaveraged_psd:uniform_conv_derivatives}
Let $j = 1$ or $2$. Then $
 |
\partial_H^j  f_{H,k}(\lambda) - \partial_H^j  f_{H,\infty}(\lambda)
 |
\to 0$ uniformly for $ \in [H_-, H_+]$ and $\lambda \neq 0$. More precisely for any $r > 0$, There exists a constant $C = C(r)$ independent of $H$ and of $k \geq 2$ such that
\begin{align*}
\forall\, \lambda \neq 0, \;
 | 
	\partial_H^j f_{H, k}(\lambda)
	-
	\partial_H^j f_{H, \infty}(\lambda)
 |
&\leq
C
\frac{|\lambda|^{2 \wedge ( 3-2H ) - r} \ln^j k}{k^{1 \wedge 2H}}.
\end{align*}

\item
\label{lemma:slow:preaveraged_psd:control_diff_psd} 
There exist $c_1$ and $c_2$ independent of $H$ and of $k$ such that for any $H \in [H_-, H_+]$, any $k \in \setN^* \cup \{ \infty \}$ and any $0 < |\lambda| < \pi$, we have
\begin{align*}
c_1 |\lambda|^{1-2H} \leq | f_{H, k}(\lambda) | \leq c_2 |\lambda|^{1-2H}.
\end{align*}
Moreover, for any $r > 0$, there exists $c_3 = c_3(r)$ a constant depending only of $r$ such that for any $( H, \lambda ) \in \Theta_H \times [-\pi, \pi] \backslash \{ 0 \}$, $k \geq 2$, $j=0,1,2,3$ and $m=0,1,2$, we have
\begin{align*}
 \Big|
\frac{\partial^{j+m} f_{H, k}(\lambda)}{\partial \lambda^m \partial H^j} 
 \Big| \leq \frac{c_3(r)}{ | \lambda  |^{2H - 1 + m + r}}
.
\end{align*}

\end{enumerate}
\end{lemma}

The proof is then split into three parts. First we show that the first estimator of $H$, namely $\widehat{H}_n$, is consistent by proving Proposition \ref{propo:slow:first}. 

The next part is devoted to refining Proposition \ref{propo:slow:first} to obtain a better convergence rate.  The idea behind it is that if we knew a \textit{deterministic} sequence $H(n) \to H_0$ that could replace $H_+$, and if this sequence satisfied some good properties, then we could easily build an oracle with the same definition than $\widehat{H}_n$ (which would not really be an estimator as some information about $H_0$ is needed beforehand). More precisely, we introduce additional parameters $\gamma > \gamma^* > 0$ and we consider $H(n) = H_0 + \varepsilon(n)$ where $\varepsilon(n) 
\sim
( 2H_0 + 1 ) ( \log \gamma - \log \gamma^* )/(2 \log n )$. Then we define
\begin{align*}
k(n) =  \lfloor \gamma^{\frac{-1}{2H(n)+1}} n^{\frac{2H(n)}{2H(n)+1}}  \rfloor \text{ and } N(n) =  \lfloor n/k(n)  \rfloor.
\end{align*}
It is clear that $k(n)$ and $N(n)$ don't meet the conditions of Proposition \ref{propo:slow:first}. More precisely, $k(n), N(n) \to \infty$ and we have
\begin{align*}
\frac{n^{2H(n)}}{k(n)^{2H(n)+1}} \to \gamma
\text{ and }
\frac{n^{2H_0}}{k(n)^{2H_0+1}} \to \gamma^*
.\end{align*}
Then we define our oracle by
\begin{align}
\label{eq:move:firstest}
( \widehat{H}_n^{oracle}, \widehat{\nu}^{oracle}_n ) = \argmin U_n
\end{align}
where the minimum is taken on $ [H_-, H(n)] \times  [ ( k/n)^{H_+}\sigma_-, ( k/n )^{H_-} \sigma_+  ]$. Of course, our parameter of interest is $\sigma = \nu ( n/k )^{H}$ so we define 
$\widehat{\sigma}^{oracle}_n = \widehat{\nu}_n( n/k )^{\widehat{H}^{oracle}_n}
$. We obtain:
\begin{proposition}
\label{propo:move}
We have the following convergence in distribution, under $\zproba_{H_0, \sigma_0}^n$:
\begin{align*}
\begin{pmatrix}
\sqrt{N}
( \widehat{H}^{oracle}_n - H_0)
\\
\frac{\sqrt{N}
}{\log N}
(
\widehat{\sigma}^{oracle}_n - \sigma_0
)
\end{pmatrix}
\to 
\begin{pmatrix}
X
\\
\sigma_0  X
\end{pmatrix}
\end{align*}
where $X$ is a centered Gaussian variable with variance
\begin{align*}
	\frac{
	\frac{4\pi}{\sigma_0^4}
		\int_{-\pi}^\pi 
		\frac{f_{H_0, \infty}(\lambda)^2}
		{( \sigma_0^2 f_{H_0, \infty}(\lambda) + \gamma^* l(\lambda) )^2} 
		d\lambda
	}{
	(	
	\int_{-\pi}^\pi 
	\frac{f_{H_0, \infty}(\lambda)^2}
	{( \sigma_0^2 f_{H_0, \infty}(\lambda) + \gamma^* l(\lambda) )^2} 
	d\lambda
	)
	(
	\int_{-\pi}^\pi 
	\frac{\partial_H f_{H_0, \infty} (\lambda)^2}
	{( \sigma_0^2 f_{H_0, \infty}(\lambda) + \gamma^* l(\lambda) )^2} 
	d\lambda
	)
	-
	(
	\int_{-\pi}^\pi 
	\frac{\partial_H f_{H_0, \infty} (\lambda)f_{H_0, \infty}(\lambda)}
	{( \sigma_0^2 f_{H_0, \infty}(\lambda) + \gamma^* l(\lambda) )^2} 
	d\lambda
	)^2
	}.
\end{align*}
\end{proposition}

We eventually conclude the proof by showing how Propositions \ref{propo:slow:first} and \ref{propo:move} imply Theorem \ref{thm:slow:optimal}.

\subsection{Properties of the preaveraged data}

We aim here at proving Lemma \ref{lemma:slow:preaveraged_psd}. Indeed, we will only prove two of the five claims, namely \ref{lemma:slow:preaveraged_psd:explicit_psd} and 
\ref{lemma:slow:preaveraged_psd:uniform_conv}
since 
the proofs of 
\ref{lemma:slow:preaveraged_psd:explicit_psd_lim}
and
\ref{lemma:slow:preaveraged_psd:uniform_conv_derivatives}
are easy adaptations of the proven properties and \ref{lemma:slow:preaveraged_psd:control_diff_psd} is a consequence of the rest of the lemma.

\subsubsection*{Proof of Lemma \ref{lemma:slow:preaveraged_psd:explicit_psd}}

Let us fix a complex Gaussian measure $M$, as introduced page 325 in \cite{samorodnitsky2017stable} and recall that for any functions $f, g: \mathbb{R} \to \mathbb{C}$ satisfying $f(-x) = \overline{f(x)}$ and $g(-x) = \overline{g(x)}$, we have:
\begin{align}
\label{eq:cov_complex}
\zEX \bigg[ \int_\mathbb{R} f(x) dM(x)\int_\mathbb{R} g(x) dM(x) \bigg] = \int_\mathbb{R} f(x) \overline{g(x)} dx.
\end{align}
Our result is based on the harmonizable representation of the fractional Brownian motion
\begin{align*}
W^H_t = \frac{1}{C(H)} 
\int_\mathbb{R}
\frac{e^{ixt}-1}{ix}
 | x  |
	^{-H+\frac{1}{2}}
M(dx)
\end{align*}
where $C(H)^2 = \pi ( H\Gamma(2H)\sin( H\pi))^{-1}$ (see 
Proposition 7.2.8 in in \cite{samorodnitsky2017stable}). Basic algebra implies that
\begin{align*}
Y_n^{(k)}
=
\frac{1}{C(H)}
\int_\mathbb{R}
\frac{( e^{ix} - 1 )^2}{ixk( e^{\frac{ix}{k}} - 1 )}
\frac{e^{ix(n-1)}}{|x|^{H-1/2}}
M(dx)
\end{align*}
where $C(H)$ is as previously. We use this representation and \eqref{eq:cov_complex} to compute auto-correlation of this sequence from which we identify the spectral density. For any $\tau$, we have
\begin{align*}
\zEX \big[Y^{(k)}_{\tau+1} Y_1^{(k)}\big]
&=
\frac{1}{C(H)^2}
\int_\mathbb{R}
\frac{ | e^{ix} - 1  |^4}{k^2 | e^{\frac{ix}{k}} - 1  |^2}
\frac{e^{ix\tau}}{|x|^{2H+1}}
dx
=
\frac{4H\Gamma(2H)\sin( H\pi)}{\pi}
\int_\mathbb{R}
\frac{\sin(\frac{x}{2})^4}{k^2\sin(\frac{x}{2k})^2}
\frac{e^{ix\tau}}{|x|^{2H+1}}
dx
\\
&=
\frac{4\Gamma(2H+1)\sin( H\pi)}{2\pi}
\sum_{m\in\mathbb{Z}}
\int_{-\pi}^\pi
\frac{\sin(\frac{x}{2} + \pi m)^4}{k^2\sin(\frac{x}{2k} + \frac{\pi m}{k})^2}
\frac{e^{ix\tau}}{|x + 2\pi m|^{2H+1}}
dx.
\end{align*}

\subsubsection*{Proof of Lemma \ref{lemma:slow:preaveraged_psd:uniform_conv}}

As $f_{H, k}$ and $f_{H, \infty}$ are both even functions, we focus on $\lambda > 0$. We fix $0 < \lambda \leq \pi$ and  we define for $s \in \mathbb{Z}$
\begin{align*}
J_s^{(k)}(\lambda) =  \Big\{
j \in \mathbb{Z}:  \Big| \frac{\lambda + 2\pi j}{2k} + s\pi  \Big| \leq \frac{2\pi}{3}
 \Big\}.
\end{align*}
Clearly, $J_s^{(k)}(\lambda) = J_0^{(k)}(\lambda) - sk$ and there exist $\frac{2}{3} < \alpha < 1$ and $k_0$ both independent of $\lambda$ such that $J_0(\lambda) \subset \lint - \lfloor\alpha k \rfloor, \lfloor \alpha k \rfloor\rint$ whenever $k \geq k_0$. Moreover, $\mathbb{Z} = \cup_s J_s^{(k)}(\lambda)$ any $j \in \mathbb{Z}$ cannot be in more that two different $J_s^{(k)}(\lambda)$ at the same time. Therefore, if $C$ is a constant that can change from line to line, that is independent of $H$ and $k$, we have
\begin{align*}
 | 
	f_{H, k}(\lambda)
	-
	f_{H, \infty}(\lambda)
 |
&\leq
C
 \bigg| 
	\sum_{j \in \mathbb{Z}}
		\frac{( 1 - \cos \lambda )^2}{ | \lambda + 2\pi j  |^{1+2H}}
		\bigg(
			\frac{4}{ | \lambda + 2\pi j  |^{2}}
			-
			\frac{1}{k^2 \sin ( \frac{\lambda}{2k} + \frac{\pi j}{k} )^2}
		\bigg)
 \bigg|
\\&\leq
\frac{C}{k^2}
\sum_{j \in \mathbb{Z}}
		\frac{|\lambda|^4}{ | \lambda + 2\pi j  |^{1+2H}}
 \bigg| 
			\frac{1}{ | \frac{\lambda + 2\pi j}{2k}  |^{2}}
			-
			\frac{1}{ \sin ( \frac{\lambda}{2k} + \frac{\pi j}{k} )^2}
 \bigg|
\\&\leq
\frac{C}{k^2}
\sum_{s \in \mathbb{Z}}
\sum_{j \in J_0^{(k)}(\lambda)}
		\frac{|\lambda|^4}{ | \lambda + 2\pi j - 2\pi sk |^{1+2H}}
 \bigg| 
			\frac{1}{ | \frac{\lambda + 2\pi j - 2\pi sk}{2k}  |^{2}}
			-
			\frac{1}{ \sin ( \frac{\lambda}{2k} + \frac{\pi j}{k} )^2}
 \bigg|
\\&\leq
\frac{C}{k^2}
\sum_{s \in \mathbb{Z}^*}
\sum_{j \in J_0^{(k)}(\lambda)}
		\frac{|\lambda|^4}{ | \lambda + 2\pi j + 2\pi sk |^{1+2H}}
 \bigg| 
			\frac{1}{ | \frac{\lambda + 2\pi j + 2\pi sk}{2k}  |^{2}}
			-
			\frac{1}{ \sin ( \frac{\lambda}{2k} + \frac{\pi j}{k} )^2}
 \bigg|
\\&
\;\;\;\;
+
\frac{C}{k^2}
\sum_{j \in J_0^{(k)}(\lambda)}
		\frac{|\lambda|^4}{ | \lambda + 2\pi j |^{1+2H}}
 \bigg| 
			\frac{1}{ | \frac{\lambda + 2\pi j}{2k}  |^{2}}
			-
			\frac{1}{ \sin ( \frac{\lambda}{2k} + \frac{\pi j}{k} )^2}
 \bigg|.
\end{align*}
Remark that $\sup_{0 < |x| \leq \frac{2\pi}{3}}  | \frac{1}{x^2} - \frac{1}{\sin x^2}  | < \infty$. Applying this to $\frac{\lambda}{2k} + \frac{\pi j}{k}$ whenever $j \in J_0^{(k)}$, we get
\begin{align*}
 | 
	f_{H, k}(\lambda)
	-
	f_{H, \infty}(\lambda)
 |
&\leq
C
\sum_{s \in \mathbb{Z}^*}
\sum_{j \in J_0^{(k)}(\lambda)}
		\frac{|\lambda|^4}{ | \lambda + 2\pi j + 2\pi sk |^{3+2H}}
+
\frac{C}{k^2}
\sum_{j \in J_0^{(k)}(\lambda)}
		\frac{|\lambda|^4}{ | \lambda + 2\pi j |^{1+2H}}
\\&
\;\;\;\;
+
C
\sum_{s \in \mathbb{Z}^*}
\sum_{j \in J_0^{(k)}(\lambda)}
		\frac{|\lambda|^4}{ | \lambda + 2\pi j + 2\pi sk |^{1+2H} | \lambda + 2\pi j  |^{2}}
\\&
\;\;\;\;
+
C
\sum_{s \in \mathbb{Z}^*}
\sum_{j \in J_0^{(k)}(\lambda)}
		\frac{|\lambda|^4}{ | \lambda + 2\pi j + 2\pi sk |^{1+2H}}.
\end{align*}
This gives four terms to control. We start with few remarks:
\begin{tight_itemize}
\item If $j \in J_0^{(k)}(\lambda)$ and $s > 0$, then $ | \lambda + 2\pi j + 2\pi sk | \geq 2\pi (s-\alpha) k$  as $j \geq -\alpha k$ and $\lambda > 0$.
\item Let us fix $\alpha'  \in (\alpha, 1)$. Then for $k \geq k_0$ big enough, if $j \in J_0^{(k)}(\lambda)$ and $s < 0$, then  $ | \lambda + 2\pi j + 2\pi sk | \geq 2\pi (|s|-\alpha) k - \pi \geq 2\pi (|s|-\alpha') k$  as $j \leq \alpha k$ and $\lambda < \pi$.
\item If $j \neq 0$, $ | \lambda + 2\pi j  | \geq \pi$ as $- \pi \leq \lambda \leq \pi$
\item $\#J_0^{(k)} (\lambda) \leq Ck$
\end{tight_itemize}
For $k \geq k_0$, the first term is bounded above by
\begin{align*}
\sum_{s > 0}
\sum_{j \in J_0^{(k)}(\lambda)}
		\frac{|\lambda|^4}{ | 2\pi (s-\alpha) k  |^{3+2H}}
+
\sum_{s > 0}
\sum_{j \in J_0^{(k)}(\lambda)}
		\frac{|\lambda|^4}{ | 2\pi (s-\alpha') k   |^{3+2H}}
&\leq
\frac{C|\lambda|^4}{k^2}.
\end{align*}
The second term is itself controlled by
\begin{align*}
\frac{C}{k^2}
\sum_{j \in J_0^{(k)}, \; j\neq 0}
		\frac{|\lambda|^4}{ | \lambda + 2\pi j |^{1+2H}}
+
\frac{C}{k^2}
|\lambda|^{3-2H}
\leq
\frac{C|\lambda|^4}{k}
+
\frac{C|\lambda|^{3-2H}}{k^2}
.
\end{align*}
The third one is rewritten 
\begin{align*}
\sum_{s \in \mathbb{Z}^*}
		\frac{|\lambda|^4}{ | \lambda + 2\pi sk |^{1+2H} | \lambda  |^{2}}
+
\sum_{s \in \mathbb{Z}^*}
\sum_{j \in J_0^{(k)}(\lambda),\; j\neq 0}
		\frac{|\lambda|^4}{ | \lambda + 2\pi j + 2\pi sk |^{1+2H} | \lambda + 2\pi j  |^{2}}
\end{align*}
which is controlled for $k \geq k_0$ large enough by:
\begin{align*}
&|\lambda|^2
\sum_{s > 0}
		\frac{1}{ | 2\pi sk |^{1+2H}}
+	
|\lambda|^2
\sum_{s > 0}
\frac{1}{ | 2\pi sk - \pi  |^{1+2H}}
\\
&\;\;\;\;+
C
\sum_{s > 0}
\sum_{j \in J_0^{(k)}(\lambda),\; j\neq 0}
		\frac{|\lambda|^4}{ | 2\pi (s-\alpha) k  |^{1+2H} }		
+
C
\sum_{s > 0}
\sum_{j \in J_0^{(k)}(\lambda),\; j\neq 0}
		\frac{|\lambda|^4}{ | 2\pi (s-\alpha') k  |^{1+2H}}	
\\&\;\;\;\;\;\;\;\;
\leq \frac{C|\lambda|^2}{k^{1+2H}} + \frac{C|\lambda|^4}{k^{2H}}	.
\end{align*}
The last term is controlled for $k \geq k_0$ big enough by
\begin{align*}
\sum_{s > 0}
\sum_{j \in J_0^{(k)}(\lambda)}
		\frac{|\lambda|^4}{ | 2\pi (s-\alpha) k  |^{1+2H} }		
+
C
\sum_{s > 0}
\sum_{j \in J_0^{(k)}(\lambda)}
		\frac{|\lambda|^4}{ | 2\pi (s-\alpha') k  |^{1+2H}}	
\leq \frac{C|\lambda|^4}{k^{2H}}	.
\end{align*}

We easily conclude by combining all these results.

\subsection{Proof of Proposition \ref{propo:slow:first}}

It consists in an easy adaptation of the proof of Theorem \ref{thm:decaying}. There are just two major differences in this proof. First, the size of the sample which is now of size $N \sim n/k$, instead of size $n$. Hence we should replace $n$ by $n/k$ whenever it appears in the proof of Theorem \ref{thm:decaying}.

The second difference is the presence of the additional asymptotic parameter $k$ in the power spectral density $f^{n}_{H,\sigma} (\lambda) 
=
\sigma^2( n/k )^{2H} f_{H, k} ( \lambda ) + \tau^2 k^{-1} l( \lambda )$ of the sequence $\mathbf{Y}^n$, see \eqref{eq:slow:spectral}. This is not a big issue since Lemma \ref{lemma:slow:preaveraged_psd} ensures that $ f_{H, k} \to  f_{H, \infty}$ and $ f_{H, k} -  f_{H, \infty}$ is well controlled. Thus in the limits appearing in the proof, we just need to replace $f_H$ by $f_{H,k}$ or $f_{H,\infty}$ and use the controls provided by Lemma \ref{lemma:slow:preaveraged_psd}.

\subsection{Proof of Proposition \ref{propo:move}}

It is again an adaptation of the proof of Theorem \ref{thm:decaying}. However, the space on which the supremum defining $( \widehat{H}_n^{oracle}, \widehat{\nu}^{oracle}_n ) $ in \eqref{eq:move:firstest} depends on $n$. Thus, additional care needs to be taken to ensure the oracle obtained in the proof is in the interior of this domain. 

In view of \eqref{eq:decaying:boundedspace} and \eqref{eq:decaying:boundedest}, we define $H_1(n) = H(n)$ and $H_2(n) = H_+$ and for $i=1$ or $2$
\begin{align*}
\widetilde{\Theta}^{n}_i( \delta, L ) 
&=
([ H_0 + (\delta - 1)/2, H_i(n)  ] \cap [H_-, H_+] )
\times 
(   [ \sigma_- (n/k)^{H_0 - H(n)}, \sigma_+(n/k)^{H_0-H_-}  ] \cap  [ L^{-1},  L ]
),
\\
\widehat{\theta}^{(i)}_n (\delta, L) &= ( \widehat{H}^{(i)}_n(\delta, L),\widehat{\sigma}^{(i)}_n(\delta, L) )
=
\argmin_{( H, \widetilde{\nu} ) \in \widetilde{\Theta}^{n}_i( \delta, L ) }
\widetilde{U}_n ( H, \widetilde{\nu} ).
\end{align*}

We can adapt here the proofs of Steps 1 and 2 in Section \ref{subsec:decaying:completion} to show that $( \widehat{H}^{(i)}_n(\delta, L),\widehat{\sigma}^{(i)}_n(\delta, L) )$ converge toward $(H_0, \sigma_0)$ in $\zproba_{H_0, \sigma_0}^{n}$-distribution for $i=1,2$. This implies that $\widehat{H}^{(i)}_n(\delta, L)$ is in the interior of $[ H_0 + (\delta - 1)/2, H_+  ] \cap [H_-, H_+]$ and therefore we can prove as in Section \ref{subsec:decaying:completion} that 
\begin{align*}
\sqrt{N} 
\begin{pmatrix}
\widehat{H}^{(2)}_n(\delta, L) - H_0
\\
 \widehat{\sigma}^{(2)}_n(\delta, L) - \sigma_0
\end{pmatrix}
\to 
J^{-1} \gaussian{0}{J} \sim \gaussian{0}{J^{-1}}
\end{align*}
where $J = \nabla^2 \widetilde{U}_\infty (H_0,\sigma_0)$. A similar result for  $\widehat{\theta}^{(1)}_n (\delta, L)$ cannot be directly obtained because $H(n) \to H_0$ so $\widehat{\theta}^{(1)}_n (\delta, L)$ is not necessarily in the interior of $\widetilde{\Theta}^{n}_1( \delta, L )$. But using this asymptotic limit together with the assumption $N \varepsilon(n) \to \infty$, we have
\begin{align*}
\zproba_{H_0, \sigma_0}^{n} \big(
\widehat{H}^{(1)}_n(\delta, L) = H(n)
 \big)
&\leq
\zproba_{H_0, \sigma_0}^{n} \big(
\widehat{H}^{(2)}_n(\delta, L) \geq H(n)
\big)
\\
&\leq
\zproba_{H_0, \sigma_0}^{n} \big(
N ( \widehat{H}^{(2)}_n(\delta, L) - H_0 ) \geq N\varepsilon(n)
\big) \to 0.
\end{align*}
This implies that $\widehat{H}^{(1)}_n(\delta, L)$ must be in the interior of $[ H_0 + (\delta - 1)/2, H(n)  ] \cap [H_-, H_+] $ so that we can now adapt the rest of the proof of Theorem \ref{thm:decaying} and conclude with Proposition \ref{propo:move}.

\subsection{Completion of the proof}

First, we define 
$h_0 = h_0(n) = H_- + i_0m^{-1}( H_+ - H_- )$ and $h_1 = h_1(n) = h_0 + m^{-1}(H_+ - H_-)$
where
$i_0 = i_0(n) =  \big\lceil m(n) \frac{H_0 - H_- + q(n)}{H_+ - H_-}  \big\rceil$. We also define
$
i_U = i_U(n) =  \big\lceil m(n) \frac{H_0 - H_- + q(n)}{H_+ - H_-} + U  \big\rceil 
$. For any $n$, there exists $0 \leq c_n < 1$ deterministic such that $i_0 = m(n) \frac{H_0 - H_- + q(n)}{H_+ - H_-} + c_n $. Notice that we have $i_U = i_0$ if $U \leq c_n$ and $i_U = i_0 + 1$ otherwise. Moreover, since $\log(n) q(n) \to \delta^*$ and $\log (n) m^{-1} \to 0$, we must have
\begin{align*}
\log(n) ( h_0(n) - H_0 ) \to \delta^*
\text{ and } 
\log(n) ( h_1(n) - H_0 ) \to \delta^*.
\end{align*}

The next lemma shows that essentially, $\widehat{i_n}$ can be replaced by $i_U$ that is still random but that does not depend on the first estimator of $H$.

\begin{lemma}
\label{lemma:optimal:2}
$\ieproba{H_0, \sigma_0}{(n)}{
\widehat{i}_n = i_U
}
\to 1$.
\end{lemma}

\begin{proof}
Let $w_n = m^{-1} v_n (H_+ - H_-)$. For any $\eta > 0$, we have:
\begin{align*}
\ieproba{H_0, \sigma_0}{(n)}{
\widehat{i}_n = i_U
}
&=
\ieproba{H_0, \sigma_0}{(n)}{
\widehat{i}_n = i_0,\;  U \leq c_n
}
+
\ieproba{H_0, \sigma_0}{(n)}{
\widehat{i}_n = i_0+1, \; U > c_n
}
\\
&=
\ieproba{H_0, \sigma_0}{(n)}{
w_n ( c_n - 1 - U ) < v_n ( \widehat{H}_n - H_0 ) \leq w_n ( c_n - U ),\;  U \leq c_n
}
\\
&\;\;\;\;
+
\ieproba{H_0, \sigma_0}{(n)}{
w_n ( c_n - U ) < v_n ( \widehat{H}_n - H_0 ) \leq w_n ( c_n + 1 - U ),\;  U > c_n
}
\\
&\geq
\ieproba{H_0, \sigma_0}{(n)}{
 | v_n ( \widehat{H}_n - H_0 )  | \leq w_n \eta, \; \eta \leq U \leq c_n - \eta
}
\\
&\;\;\;\;
+
\ieproba{H_0, \sigma_0}{(n)}{
 | v_n ( \widehat{H}_n - H_0 )  | \leq w_n \eta, \; c_n + \eta \leq U \leq 1 - \eta
}
\\
&\geq
( ( c_n - 2 \eta )_+ + ( 1 - c_n - 2 \eta )_+ )
\ieproba{H_0, \sigma_0}{(n)}{
 | v_n ( \widehat{H}_n - H_0 )  | \leq w_n \eta}
\\
&\geq
( 1 - 4 \eta )
\ieproba{H_0, \sigma_0}{(n)}{
 | v_n ( \widehat{H}_n - H_0 )  | \leq w_n \eta}.
\end{align*}
We easily conclude since $v_n ( \widehat{H}_n - H_0 )$ is bounded in probability, $w_n \eta \to \infty$ and we can take $\eta$ as small as we want.
\end{proof}

\begin{proof}[Proof of Theorem \ref{thm:slow:optimal}]
For $j=0$ or $1$, let 
$
( \widehat{H}_n^{j}, \widehat{\nu}_n^{j} ) = \argmin U_n
$
on $[H_-, h_j(n)] \times  [ ( k/n)^{H_+}\sigma_-, ( k/n )^{H_-} \sigma_+  ]$
and let $\sigma^{j}_n = \widehat{\nu}_n( \frac{n}{k})^{\widehat{H}^{j}_n}$. By Proposition \ref{propo:move}, we know that 
\begin{align*}
\begin{pmatrix}
n^{1/(4H_0+2)}
( \widehat{H}_n^{j} - H_0)
\\
n^{1/(4H_0+2)}
\log (n)^{-1}
(
\widehat{\sigma}_n^{j} - \sigma_0
)
\end{pmatrix}
\to 
\begin{pmatrix}
(\gamma^*)^{-1/(4H_0+2)} X
\\
(\gamma^*)^{-1/(4H_0+2)} \sigma_0  X
\end{pmatrix}
\end{align*}
where $X$ in as in Theorem \ref{thm:slow:optimal}. Moreover, if $\widehat{i}_n = i_j$, then $( \widehat{H}_n^{\mathrm{opt}}, \widehat{\nu}_n^{\mathrm{opt}} ) = ( \widehat{H}_n^{j}, \widehat{\nu}_n^{j} )$. We conclude using Lemma \ref{lemma:optimal:2} and using that $U$ is independent from both $( \widehat{H}_n^{\mathrm{opt}}, \widehat{\nu}_n^{\mathrm{opt}} )$ and $( \widehat{H}_n^{j}, \widehat{\nu}_n^{j} )$ for $j=0,1$.

\end{proof}

\section{Proof of Theorem \ref{thm:slow:optimality}}

\subsection{Introduction and notations}

Recall that the observations $\mathbf{Z}^{n}$ are given by 
$Z_i^{n} = \sigma W_{i/n}^H + \tau \xi_i^n
$. We denote $\zproba_{H,\sigma}^{n}$ the law of these observations. For any function $f$, we denote $\zproba_f^n$ the law of these observations given $\sigma W^H = f$. 

Let $D( \mu, \nu)
=
\int \log ( \frac{d\mu}{d\nu} ) d\mu$ denote the Kullback-Leibler divergence between two probability measures $\mu$ and $\nu$ and let $\norm{\mu }_{TV} = \sup_{\norm{f}_{\infty} \leq 1} | \int f d\mu |$ be the the total variation of a signed measure $\mu$. Recall that the Pinsker's inequality links these two quantities: $\norm{\mu - \nu}_{TV}^2 \leq 2D( \mu, \nu)$.

Let $\salC^0$ be the space of continuous functions on $[0,1]$. For any $0 \leq \alpha < 1$, we also define the $\alpha$-Hölder norm of $f\in \salC^0$ by
\begin{align*}
\norm{f}_{\salH^\alpha} = \norm{f}_{\infty} + \sup_{s \neq t}
\frac{|f(s) - f(t) |}{|s - t|^\alpha}.
\end{align*}

For $\alpha = 0$, this is a norm on $\salC^0$ that is equivalent to $\norm{\cdot}_{\infty}$. We denote $\salH^\alpha$ the set of $f \in \salC^0$ such that $\norm{f}_{\salH^\alpha} $ is finite. For such a function $f$ satisfying $f(0) = 0$, note that we have:
\begin{align}
\label{eq:holder_cont_L2}
\Bigg | \frac{1}{n} \sum_{k=0}^n f( k/n ) ^2 - \norm{f}_{2}^2 \Bigg | \leq \frac{1}{n^\alpha} \norm{f}_{\salH^\alpha}^2
\end{align}

Following \cite{gloter2007estimation}, let $( \sigma_0, H_0)$ be a point in the interior of $\Theta$. Set, for $I > 0$ well chosen, $\varepsilon_n = I^{-1} n^{-1/(4H_0 + 2)}$; $
H_1 = H_0 + \varepsilon_n$ and $
\sigma_1 = \sigma_0 2^{j_0 \varepsilon_n}
$ where
$j_0 = \left \lfloor \log_2 ( n^{1/(2H_0+1)} ) \right \rfloor$.

\subsection{Key lemmas}

\begin{lemma}
\label{lemma:optimal:bounds}
For any functions $f$ and $g$ such that $f(0) = g(0) = 0$, we have a constant $c$ and a universal non increasing strictly positive function $R$ such that
\begin{align*}
\norm{\zproba_f^n - \zproba_g^n}_{TV}
\leq c n^{\frac{1}{2}} \norm{f-g}_\infty
\;\text{and}\;
1 - \frac{1}{2} \norm{\zproba_f^n - \zproba_g^n}_{TV} \geq 
R ( 
n \norm{f-g}_2^2
+ 
n^{1-\alpha} \norm{f-g}_{\salH^\alpha}^2
)
\end{align*}
\end{lemma}

\begin{proof}
First, notice that for any functions $f$ and $g$, we have $
D ( \zproba_f^n, \zproba_g^n ) = \sum_{k=0}^n ( f( k/n ) - g( k/n ) )^2
$, since under $\zproba_f^n$, the family $( Z^{n}_i )_i$ is composed of $n+1$ independent Gaussian variables with mean $( f( \frac{i}{n} ))_i$ and variance $\tau^2$. Plugging this into the proof of Proposition 4 in \cite{gloter2007estimation} together with the bound \eqref{eq:holder_cont_L2} yields to the result.
\end{proof}

\begin{lemma}
\label{lemma:optimal:key}
For $I$ large enough, there exists a sequence of probability $( X_n, \salX_n, \mathbf{P}^{n} )$ on which can be defined two sequences of stochastic processes, $( \xi^{i,n}_t )_{t\in [0, 1]}$ and a measurable transformation $T^n: X_n \to X_n$ such that the following hold:
\begin{enumerate}[label=(\roman*)]
  \setlength{\itemsep}{0pt}
  \setlength{\parskip}{0pt}
\item \label{lemma:optimal:key:enum:1} For $\alpha < H_0$, the sequences $\norm{\xi^{0,n}}_{\salH^\alpha}$ and $\norm{\xi^{1,n}}_{\salH^\alpha}$ are tight under $\mathbf{P}^{n}$.
\item \label{lemma:optimal:key:enum:2} If $P^{i,n}( \cdot ) = \int
\ieproba{\xi^{i,n}(\omega)}{n}{ \cdot} \mathbf{P}^{n} ( d\omega) $, then $\norm{P^{i,n} - \zproba_{H,\sigma}^{n}}_{TV} \to 0$
\item \label{lemma:optimal:key:enum:3} The sequence $
n \norm{
\xi^{1,n}(\omega)
-
\xi^{0,n}( T^n\omega )
}_2^2
$
is tight under $\mathbf{P}^{n}$.
\item \label{lemma:optimal:key:enum:4}If $n$ is large enough, the probability measure $\mathbf{P}^{n}$ and its image $T^n \mathbf{P}^{n}$ are equivalent on $( X^n, \salX_n)$ and there exists $0 < c^* < 2$ such that
$\norm{
\mathbf{P}^{n}
-
T^n \mathbf{P}^{n}
}_{TV} \leq 2 - c^* < 2
$ for $n$ big enough.
\item \label{lemma:optimal:key:enum:5} The sequence $
n^{1-\alpha} \norm{
\xi^{1,n}(\omega)
-
\xi^{0,n}( T^n\omega )
}_{\salH^\alpha}^2
$
is tight under $\mathbf{P}^{n}$.
\end{enumerate}
\end{lemma}

This result is very similar to Proposition 5 from \cite{gloter2007estimation} . Points \ref{lemma:optimal:key:enum:2} \ref{lemma:optimal:key:enum:3} and \ref{lemma:optimal:key:enum:4} are the same as those in \cite{gloter2007estimation}
while \ref{lemma:optimal:key:enum:1} considers also the case of $\alpha \leq \frac{1}{2}$. Property \ref{lemma:optimal:key:enum:5} is completely new here. We will not prove pints \ref{lemma:optimal:key:enum:1} to \ref{lemma:optimal:key:enum:4} since they are proved in \cite{gloter2007estimation} in the case $H > 1/2$ and the adaptations to handle $H \leq 1/2$ are straightforward.

\begin{proof}[Proof of \ref{lemma:optimal:key:enum:5}]
First, the space $( X_n, \salX_n, \mathbf{P}^{n} )$, the stochastic processes $( \xi^{i,n}_t )_{t\in [0, 1]}$ and the measurable mapping $T^n: X_n \to X_n$ are defined as in \cite{gloter2007estimation}, section 7.2.

Then we just need to prove that 
$n^\beta \norm{
\xi^{1,n}(\omega)
-
\xi^{0,n}( T^n\omega )
}_{\salH^\alpha}$ is bounded in $L^1$, with $\beta = (1-\alpha)/2 > 0$. Markov inequality for positive random variables will ensure tightness. By (38) in \cite{gloter2007estimation}, we have 
\begin{align*}
\xi^{1,n}_t(\omega)
-
\xi^{0,n}_t( T^n\omega )
&=
\sum_{j\geq j_0}
\sum_{|k|\leq 2^{j+1}}
\sigma_1 2^{-j( H_1 + 1/2 )}
( \psi_{j,k}^{H_1}(t) - \psi_{j,k}^{H_1}(0) ) \varepsilon_{j,k}(\omega)
\\&\;\;\;\;-
\sum_{j\geq j_0}
\sum_{|k|\leq 2^{j+1}}
\sigma_0 2^{-j( H_0+ 1/2 )}
( \psi_{j,k}^{H_0}(t) - \psi_{j,k}^{H_0}(0) ) \varepsilon_{j,k}(\omega)
\end{align*}
where the functions $\psi_{j,k}^{H_0}$ and $\psi_{j,k}^{H_1}$ are defined in \cite{gloter2007estimation}, section 7.1 and the $\varepsilon_{j,k}$ are i.i.d standard Gaussian under $\mathbf{P}^{n}$. Therefore, we can bound $\norm{
\xi^{1,n}(\omega)
-
\xi^{0,n}( T^n\omega )
}_{\salH^\alpha}$ by the following two terms:
\begin{align*}
&\sum_{j\geq j_0}
\sum_{|k|\leq 2^{j+1}}
\sigma_1 2^{-j( H_1 + 1/2 )}
\norm{ \psi_{j,k}^{H_1}}_{\salH^\alpha} \left | \varepsilon_{j,k}(\omega) \right |
\text{ and }\sum_{j\geq j_0}
\sum_{|k|\leq 2^{j+1}}
\sigma_0 2^{-j( H_0+ 1/2 )}
\norm{ \psi_{j,k}^{H_0} }_{\salH^\alpha} \left |\varepsilon_{j,k}(\omega)\right |.
\end{align*}
Both these terms are handled similarly so we will focus on the first one here.
By Lemma 11 in the annex of \cite{gloter2007estimation}, there exists $C(\omega)$ a positive real random variable having finite moments for any order and a real deterministic $c > 0$ such that  for any $j \geq 0$ and $|k| \leq 2^{j+1}$, we have:
\begin{align*}
\left |\varepsilon_{j,k}(\omega)\right | 
\leq 
C(\omega) c (1+j)^c(1+|k|)^c.
\end{align*}

Moreover, $2^{-j( H_1 + 1/2 )} \leq 2^{-j_0( H_1 + 1/2 )}$ and we can apply Lemma 5 of \cite{gloter2007estimation} to get that for any $M>0$, there exists $c(\alpha, M) >0$ such that for any $n$:
\begin{align*}
\sum_{j\geq j_0}
\sum_{|k|\leq 2^{j+1}}
\sigma_1 2^{-j( H_1 + 1/2 )}
\norm{ \psi_{j,k}^{H_1}}_{\salH^\alpha} \left | \varepsilon_{j,k}(\omega) \right |
\leq
c(\alpha, M)
\sigma_1 2^{-j_0( H_1 + 1/2 )}
2^{-Mj_0} \times C(\omega).
\end{align*}

Since $C(\omega)$ is bounded in $L^1$, for any $M>0$, there exists a constant $c' = c'(\alpha, M)$ such that
\begin{align*}
\EX{
\norm{
\xi^{1,n}(\omega)
-
\xi^{0,n}( T^n\omega )
}_{\salH^\alpha}
}
\leq c' \sigma_1 2^{-j_0( M + H_1 + \frac{1}{2} )}.
\end{align*}
But $H_1 > H_0$, $\sigma_1$ is bounded and $j_0 = \left \lfloor \log_2 ( n^{1/(2H_0+1)} ) \right \rfloor$. Thus $\EX{
\norm{
\xi^{1,n}(\omega)
-
\xi^{0,n}( T^n\omega )
}_{\salH^\alpha}
}\leq c'' n^{-\beta}$ is obtained taking $M$ big enough and we can conclude.
\end{proof}

\subsection{Conclusion}

The same procedure applies for $H$ and $\sigma$ so we focus here on the efficient rate for $H$. Pick an arbitrary estimator $\widehat{H}_n$ of $H$  Pick also $I > 0$ large enough so that Lemma \ref{lemma:optimal:key} holds and choose then $M < I^{-1}/2$. Following \cite{gloter2007estimation}, section 6.2, we can show that if $X_r^{n}$ is the set of $\omega \in X^{n}$ such that 
\begin{align*}
n \norm{\xi^{0,n}(T^n\omega)-\xi^{1,n}(\omega)}_2^2
\text{ and }
n^{1-\alpha} \norm{\xi^{0,n}(T^n\omega)-\xi^{1,n}}_{\salH^\alpha}^2
\end{align*}
are bounded by $r > 0$, then for any $\lambda > 0$, the following inequality holds:
\begin{align*}
\sup_{(H,\sigma)} &\ieproba{H,\sigma}{n}{
n^{\frac{1}{4H+2}} \left | \widehat{H}_n - H \right |
\geq M
}
\\&\geq
\frac{e^{-\lambda}}{2}
\int_{X_r^{n}}
(
	\ieproba{\xi^{0,n}(T^n\omega)}{n}{A^{0}}
	+
\ieproba{\xi^{1,n}(\omega)}{n}{A^{0}} 
)
\1_{
	\frac
	{d\mathbf{P}^{n}( T^n\omega)}
	{dT^{n}\mathbf{P}^{n}} \geq e^{-\lambda}
}
 \mathbf{P}^{n}( d\omega)
\end{align*}
where $A^i = \left \{ n^{\frac{1}{4H_i+2}} \left | \widehat{H}_n - H _i\right |
\geq M\right \}$. Notice here that the set $X_r^{n}$ is not defined by the same bounds as in \cite{gloter2007estimation}, due to the difference between Lemma \ref{lemma:optimal:bounds} and Proposition 5 of \cite{gloter2007estimation}.
Then following \cite{gloter2007estimation}, we can conclude with the following lemma:

\begin{lemma} 
\begin{enumerate}[label=(\roman*.)]
  \setlength{\itemsep}{0pt}
  \setlength{\parskip}{0pt}
\item \label{lemma:optimality:conclusion:1} For any $r > 0$, there exists $c(r) > 0$ such that on $X_r^{n}$, we have
\begin{align*}
\ieproba{\xi^{0,n}(T^n\omega)}{n}{A^{0}}
	+
\ieproba{\xi^{1,n}(\omega)}{n}{A^{0}} 
\geq c(r)
.
\end{align*}
\item \label{lemma:optimality:conclusion:2} For $n$ large enough, 
$\mathbf{P}^{n}
(
X_r^{n}
\cap
\left \{
\frac
	{d\mathbf{P}^{n}( T^n\omega)}
	{dT^{n}\mathbf{P}^{n}}
	\geq e^{-\lambda}
\right \}
)
\geq 
\mathbf{P}^{n}
( 
X_r^{n}
)
-
e^{-\lambda}
- 1
+
\frac{c^*}{2}$.
\item \label{lemma:optimality:conclusion:3} $\lim\limits_{r \to \infty}
\liminf_{n\to\infty}
\mathbf{P}^{n}
( 
X_r^{n}
)
=
1$.
\end{enumerate}
\end{lemma}

\begin{proof}
The proofs of the first to pars are similar to Lemmas 2 and 3 of \cite{gloter2007estimation}, but here we also use Lemma \ref{lemma:optimal:bounds} to handle the case $H \leq \frac{1}{2}$. For the last part, we apply points \ref{lemma:optimal:key:enum:3} and  \ref{lemma:optimal:key:enum:5} of Lemma \ref{lemma:optimal:key}.
\end{proof}

\section*{Acknowledgements}

I am grateful to Marc Hoffmann and Mathieu Rosenbaum for helpful remarks and discussions.

\bibliographystyle{alpha}
\bibliography{library}

\end{document}